\documentclass{amsart}

\usepackage{geometry}
\usepackage{verbatim}
\usepackage{xy}
\renewcommand{\epsilon}{\varepsilon}
\newcommand{\Map}{\mathrm{Map}}
 \input xy
 \xyoption{all}

\newcommand{\xtor}{x-\mathrm{tor}}
\newcommand{\uadd}{\mathcal{U}_{\mathrm{wadd}}}
 \renewcommand{\hom}{\mathrm{Hom}}
\newcommand{\scr}[1]{\mathscr{#1}}

\newcommand{\catst}{\mathrm{Cat}^{\mathrm{perf}}_{\infty}}
\newcommand{\Catst}{\widehat{\mathrm{Cat}}^{\mathrm{perf}}_{\infty}}

\renewcommand{\sp}{\mathrm{Sp}}
\newcommand{\cbaug}{\mathrm{CB}^\bullet_{\mathrm{aug}}}
\newcommand{\cb}{\mathrm{CB}^\bullet}
\newcommand{\tow}{\mathrm{Tow}}
\newcommand{\townil}{\mathrm{Tow}^{\mathrm{nil}}}
\newcommand{\towconst}{\mathrm{Tow}^{\mathrm{const}}}

\newcommand{\towen}{\mathrm{Tow}^{\epsilon,\mathrm{nil}}}

\newcommand{\prls}{\mathcal{P}\mathrm{r}^L_{\mathrm{st}}}
\newcommand{\mot}{\mathrm{Mot}}

\newcommand{\perf}{\mathrm{Perf}}

\usepackage{Definitions}
\usepackage{Environments2}
\usepackage{PageSetup}
\usepackage{amssymb}

\usepackage[toc,page]{appendix} 

\usepackage[
disable=false, 
color=orange!80, 
bordercolor=black,
textwidth=.8in,
textsize=small]
{todonotes}

\makeatletter \providecommand\@dotsep{5}
\makeatother

\newcommand{\fun}{\mathrm{Fun}}

\newcommand{\md}{\mathrm{Mod}}

\newcommand{\Coind}{\mathrm{Coind}}
\newcommand{\Res}{\mathrm{Res}}

\usepackage{url} 

\theoremstyle{definition}

\newtheorem{cons}[thm]{Construction}

\Crefname{thmA}{Theorem}{Theorems}
\Crefname{conj}{Conjecture}{Conjectures}

\newcommand{\e}[1]{\mathbb{E}_{#1}}

\newcommand{\ko}{\mathit{ko}}
\newcommand{\ku}{\mathit{ku}}
\newcommand{\KO}{\mathit{KO}}
\newcommand{\KU}{\mathit{KU}}
\newcommand{\THH}{\mathit{THH}}
\newcommand{\TC}{\mathit{TC}}
\newcommand{\tmf}{\mathit{tmf}}
\newcommand{\Tmf}{\mathit{Tmf}}
\newcommand{\TMF}{\mathit{TMF}}

\renewcommand{\mod}{\mathrm{Mod}}

\renewcommand{\theenumi}{\arabic{enumi}}

\makeatletter
\renewcommand{\p@enumii}{\theenumi.}
\makeatother
\newcommand{\triplearrows}{\begin{smallmatrix} \to \\ \to \\ 
\to \end{smallmatrix} }

\newcommand{\CAlg}{\mathrm{CAlg}}

\newcommand{\clg}{\CAlg}

\begin{document}
	\title{Descent in algebraic $K$-theory and a conjecture of Ausoni-Rognes}
	\author{Dustin Clausen}
	\address{University of Copenhagen\\ Copenhagen, Denmark}
	\email{dustin.clausen@math.ku.dk}
	\urladdr{http://dtclausen.tumblr.com}

	\author{Akhil Mathew}
	\address{University of Chicago\\ Chicago, USA }
	\email{amathew@math.uchicago.edu}
	\urladdr{http://math.uchicago.edu/~amathew/}

	\author{Niko Naumann}
	\address{University of Regensburg\\
	NWF I - Mathematik; Regensburg, Germany}
	\email{Niko.Naumann@mathematik.uni-regensburg.de}
	\urladdr{http://homepages.uni-regensburg.de/~nan25776/}

	\author{Justin Noel}
	\address{University of Regensburg\\
	NWF I - Mathematik; Regensburg, Germany}
	\email{justin.noel@mathematik.uni-regensburg.de}
	\urladdr{http://nullplug.org}

	\date{\today}

	\begin{abstract}
	Let $A \to B$ be a $G$-Galois extension of rings, or more generally of
	$\mathbb{E}_\infty$-ring spectra in the sense of Rognes. A basic question
	in algebraic $K$-theory asks how close the map $K(A) \to K(B)^{hG}$ is to
	being an equivalence, i.e., how close algebraic $K$-theory is to satisfying
	Galois descent. An elementary argument with the transfer shows that this
	equivalence is true rationally in most cases of interest. Motivated by the
	classical descent theorem of Thomason, one also expects such a result after
	periodic localization.

	We formulate and prove a general result which enables one to promote
	rational descent statements as above into descent statements after
	periodic localization.  This reduces the localized descent problem to establishing an elementary condition on $K_0(-)\otimes \mathbb{Q}$. As applications, we prove various
	descent results in the periodically localized $K$-theory, $\TC$, $\THH$, etc.\ of
	structured ring spectra, and verify several cases of a conjecture of Ausoni and Rognes.
	\end{abstract}

	\maketitle
	
\tableofcontents

\section{Introduction}
\subsection{Motivation}
Let $X$ be a noetherian scheme. 
A subtle and important invariant of $X$ is given by the 
\emph{algebraic $K$-theory}
groups
$\left\{K_n(X)\right\}_{n \geq 0}$. As $X$ varies, the
groups $\left\{K_n(X)\right\} $ behave
something like a cohomology theory in $X$. 
For example, they form a contravariant functor in $X$ and there is
an analog of the classical Mayer-Vietoris sequence 
thanks to the localization properties of algebraic $K$-theory
\cite[Thm.~10.3]{TT90}. 
A highbrow formulation of this property is that the groups $\left\{K_n(X)\right\}$
arise as the homotopy groups of a \emph{spectrum} $K(X)$, and that the contravariant
functor
\[ K(-)\colon \mathrm{Sch}^{\mathrm{op}} \to \sp_{\geq 0}  \]
forms a sheaf of \emph{connective spectra} on the Zariski site of
$X$.\footnote{To obtain a sheaf of spectra, one has to work with the
non-connective version $\mathbb{K}$ of $K$-theory.}

As is well-known, however, the Zariski topology of $X$ 
is too coarse to have a strong analogy with algebraic topology: a more
appropriate topology is given by the \emph{\'etale} topology.
One might hope that $K$ is a sheaf (i.e., behaves `like a cohomology
theory') for the \'etale topology; if so, one
could then hope for a local-to-global spectral sequence (an analog of the
Atiyah-Hirzebruch spectral sequence for topological $K$-theory) 
beginning with \'etale cohomology and ultimately converging to algebraic
$K$-theory. Indeed, the convergence properties of such a spectral sequence are
the subject of the Quillen-Lichtenbaum conjecture \cite{Licht73,QuillenICM}, which is a consequence of the Rost-Voevodsky Norm Residue theorem (see \cite{Kolster} for a recent survey).

The problem is that $K$-theory is not a sheaf for the \'etale
topology. If $E \to F$ is a $G$-Galois extension of fields,
one has
a $G$-action on $K(F)$ and 
a canonical map 
\begin{equation} \label{descmap} K(E) \to K(F)^{hG}, \end{equation} 
but this need not be an equivalence\footnote{see here: https://mathoverflow.net/questions/239393/simplest-example-of-failure-of-finite-galois-descent-in-algebraic-k-theory/240025\#240025}, contradicting \'etale descent. 
In fact, since algebraic $K$-theory satisfies Nisnevich descent (cf.\ \cite{Nis89} and
\cite[Thm.~10.8]{TT90}), the failure of 
descent along Galois extensions of commutative rings is the \emph{only}
obstruction to satisfying \'etale descent \cite[Cor.\ 4.24]{DAGXI}.

In the foundational paper \cite{Thomason85} and in the later extension
\cite{TT90}, Thomason showed that these
problems disappear after a localization, after which
maps of the form \eqref{descmap} become equivalences. 
Specifically, let $X$ be a scheme where a fixed prime number $\ell$ is invertible and suppose $X$ contains the $\ell$th roots of unity.
If $\ell \geq  5$ (as we assume for simplicity), there is a canonical element $\beta \in \pi_2 ( K(X)/\ell)$ called the
\emph{Bott element.}  
Under these assumptions, $K(X)/\ell$ inherits the structure of a ring spectrum up to homotopy and one can form
the localization $(K(X)/\ell) [\beta^{-1}]$. 

\begin{thm}[{Thomason \cite[Thm.~2.45]{Thomason85} and \cite[Thm.~11.5]{TT90}}] \label{thm:thomason}
Suppose $X$ is a noetherian scheme of finite Krull dimension 
over $\mathbb{Z}[1/\ell, \mu_\ell]$, satisfying Thomason's technical hypotheses on
the existence of  `Tate-Tsen filtrations' of uniformly
bounded length on the residue fields.\footnote{This implies that the residue fields of $X$ have
uniformly bounded $\ell$-torsion Galois cohomological dimension. To the authors' knowledge, no counterexample to the converse implication is known.}
Then the functor $Y \mapsto
(K(Y)/\ell)[\beta^{-1}]$ is a sheaf of spectra on the small \'etale site of $X$. 
Moreover, there exists a descent spectral sequence
\[  E_2^{s,2t}=H^s_{\textrm{\'et}}(X, (\widetilde{\pi}_{2t}K/\ell)[\beta^{-1}]) \cong H^s_{\textrm{\'et}} ( X , \mathbb{Z}/\ell(t)) \Longrightarrow \pi_{2t-s}
(K(X)/\ell)[\beta^{-1}]. \] The differentials run $d_n\colon E_n^{s,t}\to E_n^{s+n,t+n-1}$.
\end{thm} 

Thomason observed that there is another construction  of
Bott inverted $K$-theory \cite[\S A.14]{Thomason85} that makes no reference
to Bott elements. One can first form mod
$\ell$ $K$-theory by smashing with the Moore spectrum $S/\ell$ and then
obtain Bott-periodic $K$-theory by inverting the Adams self-map $v$ of the Moore spectrum. In the setting above we have an equivalence: \[(K(-)/\ell)[\beta^{-1}]\simeq K(X)\otimes S/\ell[v^{-1}].\] 

Now if we invert those maps which become an equivalence after smashing with
$T(1):=S/\ell[v^{-1}]$, we obtain a localization functor $L_{T(1)}$. It
follows from Thomason's theorem that $L_{T(1)}K(-)$ satisfies \'etale descent
under the above hypotheses. 

The main result of this paper is a
generalization of this statement below to the
world of structured ring spectra.
In particular, we prove the following result, a derived version of finite flat
descent of telescopically localized $K$-theory. 
\begin{thm} 
Let $A$ be an $\mathbb{E}_\infty$-ring spectrum and let $B$ be an
$\mathbb{E}_\infty$-$A$-algebra such that $\pi_*(B)$ is finite and
faithfully flat as a
$\pi_*(A)$-module. Then after telescopic localization $L_{T(n)}$ (at any
implicit prime and height $n$), algebraic
$K$-theory satisfies descent along $A \to B$. That is, one can recover
$L_{T(n)} K(A)$ as the homotopy limit
\[ L_{T(n)}  K(A) \simeq \mathrm{Tot} \left( L_{T(n)} K(B) \rightrightarrows L_{T(n)}
K(B \otimes_A B) \triplearrows \dots  \right).\]
\end{thm} 

Our methods also apply to 
certain generalizations of finite flat extensions, such as the Galois extensions
of Rognes \cite{Rognes08}. We describe our results further below. 

\subsection{Extending Thomason's rational descent argument}
As the above notation indicates, there is an infinite family of such
telescopic localization functors $\{L_{T(n)}\}$ indexed over the integers
$n\geq 0$ and primes $\ell$. 
When $n=0$,
$T(0)$-equivalences are precisely rational equivalences (for every prime
$\ell$). Thomason observed that proving the rational analog of
\Cref{thm:thomason} is actually quite easy and reduces to a transfer argument
for finite Galois extensions \cite[Thm.\ 2.15]{Thomason85}. 

We will show that
Thomason's argument for the rational case actually implies \'etale
descent for  $L_{T(1)}K(-)$ by exploiting the fact that the algebraic
$K$-theory of commutative rings does not just take values in homotopy
commutative ring spectra, but rather in (vastly more structured) $\mathbb{E}_\infty$-ring spectra. In this setting we can apply the May nilpotence conjecture \cite{MNN_nilpotence}.

Since Thomason's argument is so simple and central to the motivation of this
paper, we will recall it in a modernized form. Consider the case of a
$G$-Galois extension of fields $A \to B$. Using techniques of Merling
\cite{Mer15}, Barwick \cite{Bar15}, and Barwick-Glasman-Shah \cite{BGS15},
one constructs an $\mathbb{E}_\infty$-algebra $R = K(B; G)$ in the $\infty$-category
of $G$-spectra such that for any subgroup
$H \leq G$, we have $R^H = K(B; G)^H = K(B^{hH})$. 
The descent comparison map now becomes the classical comparison map
\begin{equation}\label{eq:comp}
 K(A)=R^G \to R^{hG}=K(B)^{hG}
 \end{equation}
from fixed points to homotopy fixed points. 
It is a general fact about $G$-spectra that the fiber of this comparison map
is a module over the \emph{cofiber} $C=K(A)/K(B)_{hG}$  of the transfer map
\[ K(B)_{hG} \to K(A), \]
which arises from restriction of scalars along $A \to B$.
The ring spectrum $C$ has an $\mathbb{E}_\infty$-structure: in the language of equivariant
stable homotopy theory, we have $C = (R \otimes \widetilde{EG})^G$.
Now for Galois extensions of fields, the induced map 
\begin{equation}\label{eq:surj}
K_0(B) \to K_0(A)\end{equation} 
is the map $\mathbb{Z} \stackrel{|G|}{\to} \mathbb{Z}$. 
It follows that $\pi_0 C \simeq \mathbb{Z}/|G|$, so that $C$ is in
particular rationally trivial. 
We use the $C$-module structure on the fiber of
\eqref{eq:comp} to obtain Thomason's equivalences:
\begin{equation}\label{eq:q-descent} K(A)\otimes \mathbb{Q}
\xrightarrow{\simeq} (K(B)^{hG})\otimes \mathbb{Q}\xrightarrow{\simeq}
(K(B)\otimes \mathbb{Q})^{hG}. \end{equation} 

These equivalences are direct consequence of the ring structures and the
rational surjectivity of \eqref{eq:surj}. Moreover, the equivalences in
\eqref{eq:q-descent} \emph{imply} that the rationalized transfer map
$K(B)_{hG}\otimes \mathbb{Q}\to K(A)\otimes \mathbb{Q}$ is an equivalence
which, in turn, implies the surjectivity in the rationalization of
\eqref{eq:surj}.  So the rational surjectivity of \eqref{eq:surj} is a \emph{necessary and sufficient} condition for the descent equivalences in \eqref{eq:q-descent}.

We now build on Thomason's argument by using the $\mathbb{E}_\infty$-structure on $C$. The proof of the May conjecture from  \cite{MNN_nilpotence} says the rational triviality of $C$ is equivalent to the triviality of $L_T C$ for \emph{every} telescopic localization of $C$. It follows that the rational surjectivity of the transfer is equivalent to the $L_{T(0)}=\mathbb{Q}$-equivalences in \eqref{eq:q-descent} which are, in turn, equivalent to those equivalences after we replace $L_{T(0)}$ with any telescopic localization. 

\subsection{Methods}
In order to consider non-Galois extensions, we will not use the technology of equivariant algebraic $K$-theory or the language of $G$-spectra in this paper.  Instead we will work with the symmetric monoidal stable $\infty$-category $\mot_A$ of non-commutative $A$-linear motives developed in \cite{BGT13, BGT14, HSS15} as a replacement for the equivariant stable homotopy category. We will now include a brief sketch of our methods. 

By construction, non-commutative motives form the universal stable $\infty$-category for studying weakly additive invariants (see \Cref{weaklyadditive}) such as $K$-theory, $THH$, or $\TC$. To be more specific, let $E$ be such a functor valued in spectra. Then the restriction of $E$ to a functor on commutative $A$-algebras canonically factorizes as $\CAlg(A)\xrightarrow{[-]}\mot_A\xrightarrow{\widetilde{E}}\mathrm{Sp},$ where $[-]$ takes tensor products of $A$-algebras to tensor products in $\mot_A$ and $\widetilde{E}$ is an exact functor uniquely determined by $E$. This allows us to restrict our attention to exact functors out of motives.

Fix $B \in \CAlg(A)$.
Consider the full subcategory $\mathcal{I}\subset\mot_A$ of all those $M\in \mot_A$ for which the augmented cosimplicial object
(in which we have suppressed the codegeneracies)
\[ M\rightarrow M\otimes [B]\rightrightarrows M\otimes [B]\otimes [B]\mathrel{\vcenter{\mathsurround0pt
                \ialign{#\crcr
                        \noalign{\nointerlineskip}$\rightarrow$\crcr
                        \noalign{\nointerlineskip}$\rightarrow$\crcr
                        \noalign{\nointerlineskip}$\rightarrow$\crcr
                }%
        }}\cdots\] 
        becomes a limit diagram after applying \emph{any} exact functor.
		  Since $\mathcal{I}$ is a $\otimes$-ideal, we know that the monoidal
		  unit $[A]$ belongs to $\mathcal{I}$ precisely when the symmetric
		  monoidal Verdier quotient $\mot_A/\mathcal{I}$ is trivial\footnote{To
		  be precise, one should first pass to suitably small subcategories
		  before taking this quotient.}, which is, in turn, equivalent to the
		  triviality of the ring spectrum
		  $R'=\hom_{\mot_A/\mathcal{I}}([A],[A])$. This ring  spectrum will play the same role as $C$ did in Thomason's argument above.
		  Since $R'$ is defined as the endomorphisms of the unit in a symmetric
		  monoidal $\infty$-category, $R'$ is canonically an
		  $\mathbb{E}_\infty$-ring.

        Using the universal properties of $K$-theory we obtain a ring map
		  $K_0(A)\to \pi_0 R'$ which, since $[B]$ belongs to $\mathcal{I}$, sends
		  anything in the image of a `transfer map' $K_0(B)\to K_0(A)$ (which
		  makes sense if there is a morphism $[B] \to [A]$)  to zero.
        In particular, if there is a surjective transfer map $K_0(B)\twoheadrightarrow K_0(A)$, then $[A]\in\mathcal{I}$.

For our localized descent results we want to show that $[A]$ lies in $\mathcal{I}$ `up to telescopic localization', which will be equivalent to asking for the desired equivalences 
\[ L_T \widetilde{E}([A])\xrightarrow{\simeq} L_T \mathrm{Tot}
\left(\widetilde{E}([B]^{\otimes \bullet +1})\right)\xrightarrow{\simeq}
\mathrm{Tot} \left(L_T \widetilde{E}([B]^{\otimes \bullet +1})\right)\] for any exact functor $\widetilde{E}$ and any telescopic localization $L_T$. This will be formalized in \Cref{sec:epsilon-nil} by saying that $[A]$ lies in the $\epsilon$-enlargement $\mathcal{I}_\epsilon$ of $\mathcal{I}$.

To see that $[A]$ lies in $\mathcal{I}_\epsilon$, we no longer need to check
the triviality of $R'$. Instead, we need to check the triviality of each of the
telescopic localizations of $R'$. Since $R'$ is an $\mathbb{E}_\infty$-ring
spectrum, the solution to May's conjecture \cite{MNN_nilpotence} implies that this
condition is equivalent to the rational triviality of $\pi_0(R')$. We can now argue as above and see that this condition follows from the existence of a rationally surjective transfer map $K_0(B)\otimes \mathbb{Q}\twoheadrightarrow K_0(A)\otimes \mathbb{Q}$.

\subsection{The $K$-theory of structured ring spectra and a conjecture of Ausoni and Rognes}

The above argument shows that one can drop some of Thomason's technical
hypotheses and still obtain \'etale descent for $L_{T(n)}K(-)$, for every
$n\geq 0$ and implicit prime $\ell$. On the one hand, when $n\geq 2$, a result of Mitchell
\cite{Mitchell90} shows $L_{T(n)}K(X)$ is trivial for every scheme $X$, so
our argument provides no information.
On the other hand, the arguments above are very
robust and one can hope that they will generalize to other contexts.

Waldhausen \cite{WalLocal} proposed such a context, namely that the $K$-theory of rings could be extended to `brave new rings'
(now called structured ring spectra). This proposal has been realized in
the work of many people (cf., e.g., \cite{EKMM97,BGT13}) and numerous tools have been developed for this
generalization. This has led to deep calculations in certain important cases. We refer to \cite{RognesICM} for a recent survey.
Note in particular that the conclusion of Mitchell's theorem entirely fails
for the $K$-theory of ring spectra. 

To further extend the analogy with algebra, Rognes \cite{Rognes08} formulated a notion of a Galois extension of $\mathbb{E}_\infty$-ring spectra generalizing the
classical notion in commutative algebra.
A fundamental example of such an extension is the complexification map $\KO \to \KU$ from real to complex topological $K$-theory. There are higher chromatic analogs of this example coming from Lubin-Tate spectra and the theory of topological modular forms.

In this setting, Ausoni and Rognes made the following descent conjecture, which we can view as the higher chromatic analog of
the statement of \Cref{thm:thomason}:

\begin{conj}[Ausoni and Rognes \cite{AR}]\label{conj:ar}
Let $A \to B$ be a $K(n)$-local $G$-Galois extension of $\mathbb{E}_\infty$-rings. 
 Let $T(n+1)$ be a telescope of a $v_{n+1}$-self map of a type $(n+1)$-complex.
Then the map 
\[ T(n+1) \otimes  K(A)  \to T(n+1) \otimes  K(B)^{hG}  \]
is an equivalence. 
\end{conj}

The following main result, which is proven in the body of the paper as \Cref{finitedesc}, will imply several important cases of \Cref{conj:ar}, but also applies to non-Galois extensions.

\begin{thm}\label{thm:main}
Suppose $A \to B$ is a morphism of $\mathbb{E}_\infty$-rings such that $B$ is a perfect
$A$-module and such that the rationalized
restriction of scalars map $K_0(B) \otimes \mathbb{Q} \to K_0(A) \otimes \mathbb{Q}$
is surjective. Let $E(-)$ be either algebraic $K$-theory $K(-)$,
non-connective algebraic $K$-theory $\mathbb{K}(-)$, topological
Hochschild homology $\THH(-)$, or topological cyclic homology $\TC(-)$ and let $L_T$ denote one of the following periodic localization functors: $L_{T(n)}$, $L_{K(n)}$, $L_n^f$, or $L_n$ (taken at an implicit prime $\ell$ for some $n\geq 0$).

Then the map 
\begin{equation} \label{intro:descmap} E(A) \to \mathrm{Tot}\left( E(B)
\rightrightarrows E(B \otimes_A B ) \triplearrows \dots \right)  \end{equation}
becomes an equivalence after 
$L_T$-localization (which can be performed either inside or outside the
totalization). Moreover, the associated $\mathrm{Tot}$/\v{C}ech-spectral sequence collapses at a finite page with a horizontal vanishing line\footnote{Here we are following the standard convention for drawing spectral sequences with Adams indexing; the $y$-axis indexes the \emph{filtration degree} $s$ and the $x$-axis indexes the \emph{total degree} $t-s$.}. 
\end{thm} 

In particular, if $A \to B$ is a $G$-Galois extension satisfying the above hypothesis on
$K_0(-) \otimes \mathbb{Q}$, then the Ausoni-Rognes Conjecture \ref{conj:ar} holds for this extension. In this case, the associated spectral sequence is the homotopy fixed point spectral sequence: \[H^s(G;\pi_t L_T E(B))\Rightarrow \pi_{t-s} L_T E(A). \]

\Cref{thm:main} reduces the localized descent problem to a question about $K_0(-)\otimes \mathbb{Q}$. This condition is relatively accessible and can be checked in many examples of interest:
\begin{thm}
	The hypotheses of \Cref{thm:main} are satisfied for the following maps of $\mathbb{E}_\infty$-ring spectra:
	\begin{itemize}
		\item Any finite \'etale cover $A\to B$ (see \Cref{finiteetalesatisfied} for a more general condition).
		\item The complexification maps of topological $K$-theory spectra: $\KO\to \KU$ or $\ko\to \ku$ (\Cref{ex:ko,ex:koconn}).
		\item The $G$-Galois extensions $E_n^{hG}\to E_n$ where $G\subset
		\mathbb{G}_n$ is a finite subgroup of the extended Morava stabilizer
		group (\Cref{cor:higher_real_transfer}).\footnote{We refer to
		\cite[Sec. 3.1]{HMS15}  for a proof that this is a global extension.}
		\item Any finite $G$-Galois extension of the following variations on topological modular forms: $\TMF[1/n], \Tmf_0(n),$ or $\Tmf_1(n)$ (\Cref{TMFcb,thm:log-tmf}), where $n\geq 1$.
		\item The extension $\tmf[1/3]\to \tmf_1(3)$ of connective topological modular forms (\Cref{ex:tmf13}).
	\end{itemize}
\end{thm}

For extensions defined by certain higher real $K$-theories one has sharper results (see \Cref{f2modulekoku}).

\subsection{Further remarks}
In the algebraic $K$-theory of \emph{connective} ring spectra, 
the theory of \emph{trace methods} is a fundamental tool.
Using it, one can try to answer the above question by combining Thomason's results
together
with comparisons with topological cyclic homology. 
However, the most interesting Galois extensions arise from nonconnective ring
spectra (in fact, all Galois extensions of connective ring spectra can be
determined in terms of pure algebra by \cite[Thm.\ 6.17]{MGal} and \cite[Ex.\ 5.5]{MM15}). 
As a result, our approach in this paper is completely different. 

We also emphasize that our methods \emph{do not} recover Thomason's spectral sequence, whose $E_2$-term is the sheaf cohomology. The convergence of that spectral sequence requires \'etale \emph{hyperdescent}, rather than the descent result we prove. It seems to us that any hyperdescent statement for non-rational localizations will require additional tools. 
Similarly, we cannot treat the case of a pro-Galois extension with profinite
Galois group. 

However, if we ignore the issue of sheaf vs.\ hypersheaf, then our methods show that
none of the technical hypotheses on $X$ imposed in Thomason's theorem are
actually necessary.  Thus, we do obtain new descent results even in the algebraic
$K$-theory of discrete rings; see also Example \ref{projectivedescent}.

Finally, we remark that another crucial aspect of Thomason's spectral sequence is the identification of the \' etale sheafified homotopy groups of $K/\ell(-)[\beta^{-1}]$ with the $\mathbb{Z}/\ell(n)$ spaced out in even degrees, so that the $E_2$-term can be written explicitly.  This identification follows from the Gabber-Suslin rigidity theorem, which we do not know the analog for in the setting of structured ring spectra and higher chromatic localizations.

\subsection{Outline}
In \Cref{sec:epsilon-nil} we introduce the notion of $\epsilon$-objects and
$\epsilon$-equivalences. This robust theory allows us to  formulate
rigorously
what it means for a homotopy limit to quickly converge modulo objects which are
invisible to chromatic homotopy theory. In \Cref{sec:nc_motives} we review the
relative theory of non-commutative motives as developed by \cite{BGT13,HSS15}.
This section is technical and can be skipped on a first reading. 
The only minor variation
in our treatment allows us to consider additive invariants which do not commute with all filtered colimits (such as $\TC$).

In \Cref{sec:descent_results}, we show how the nilpotence criterion for
$\mathbb{E}_\infty$-rings of \cite{MNN_nilpotence} provides a simple criterion
for establishing descent results.  In \Cref{sec:examples} we establish our
aforementioned examples and prove that periodically localized $K$-theory is a
sheaf for the finite flat topology on $\mathbb{E}_\infty$-ring spectra. As a
corollary, we obtain in \Cref{sec:etaledesc} an analog for spectral algebraic spaces of Thomason's result that the $K$-theory of
schemes satisfies \'etale descent after periodic localization.
\Cref{sec:appendix} by Meier, Naumann, and Noel shows the existence of finite even complexes with specific Morava $K$-theories. This might be of independent interest, but serves the immediate purpose of establishing descent for the algebraic $K$-theory of higher real $K$-theories.

\subsection*{Notation}
\begin{enumerate}
\item  
We will write $\catst$ for the $\infty$-category of idempotent-complete,
small stable $\infty$-categories and exact functors between them. 
We recall that $\catst$ has the structure of a presentable, symmetric monoidal
$\infty$-category under the Lurie tensor product. 
We refer to \cite{BGT13} for an account of the general features of $\catst$. 
\item
We let $\clg(\catst)$ denote the $\infty$-category of commutative algebra
objects in $\catst$: equivalently, this is 
the $\infty$-category of small symmetric monoidal, idempotent-complete stable
$\infty$-categories $(\mathcal{C}, \otimes, \mathbf{1})$ where $\otimes$ is
exact in each variable. The morphisms in $\clg(\catst)$ are the symmetric
monoidal exact functors. Given  an object $X \in \mathcal{C}$, we will write $\pi_k (X) = \pi_k
\hom_{\mathcal{C}}(\mathbf{1}, X)$.
\item We will let $\prls$ denote the category of presentable, stable
$\infty$-categories and cocontinuous functors between them, with the Lurie
tensor product. 
\item 
 We will also need to consider not
necessarily small,
idempotent-complete stable $\infty$-categories and exact functors between 
them, and we denote the $\infty$-category of such by $\Catst$.
Thus $\Catst$ includes both small and presentable idempotent-complete, stable
$\infty$-categories. 
\item 
Although due to set-theoretic technicalities we will not consider $\Catst$ as a symmetric monoidal $\infty$-category, we will 
abuse notation and write $\clg(\Catst)$ for the $\infty$-category of symmetric
monoidal, stable, idempotent-complete $\infty$-categories with a biexact
tensor product, and symmetric monoidal exact functors between them. 
\item Throughout this paper we will use Lurie's definitions of the  \'etale and Nisnevich topologies (see \cite[\S 1.2.3 and \S 3.7]{lurie_sag} and \cite[App.~C]{hoyois_quadratic} for further discussion).
\end{enumerate}

We write $\sp$ for the $\infty$-category of spectra and $\sp^\omega \subset
\sp$ for the subcategory of finite spectra. Given an $\mathbb{E}_\infty$-ring spectrum $R$, we write
$\perf(R)$ for the $\infty$-category of perfect (i.e., compact) $R$-modules. 
We will write $\mathbb{D}$ for the dual of an object (e.g., the
Spanier-Whitehead dual of a finite spectrum). 
Given objects $X, Y$ of a stable $\infty$-category $\mathcal{C}$, we will write
$\hom_{\mathcal{C}}(X, Y) \in \sp$ for the mapping \emph{spectrum}. 

\subsection*{Acknowledgments}
The first author benefited from a stimulating visit to the University of Oslo, and would like to thank John Rognes for sharing his hospitality and mathematical energy with equal measures of generosity.
The second and fourth authors would like to thank the Hausdorff Institute of Mathematics
for its hospitality during the trimester program on `Homotopy theory,
manifolds, and field theories.'
The second author would also like to thank the University of Copenhagen for its hospitality
during a very productive weeklong visit in 2015, and for the invitation to give a lecture
series on this work in July 2016.

We would like to thank Clark Barwick, Bhargav Bhatt,  Lars Hesselholt, Mike
Hopkins, Marc Hoyois, Jacob Lurie, Lennart Meier, Peter Scholze, and Georg Tamme for helpful
discussions. 
We would like to thank John Rognes for many helpful comments on a
draft of this paper. We thank two anonymous referees for valuable comments.
The first author was supported by Lars Hesselholt's Niels Bohr Professorship.
The second author was supported by the NSF Graduate Fellowship under grant
DGE-114415, and was a Clay Research Fellow as this paper was finished. The third author was partially supported by SFB 1085 - Higher Invariants, Regensburg. The fourth author was partially supported by the DFG grants: NO 1175/1-1 and SFB 1085 - Higher Invariants, Regensburg.

\section{$\epsilon$-nilpotence}\label{sec:epsilon-nil}

Let $\mathcal{C} \in \Catst$ be a (not necessarily small) stable, idempotent-complete $\infty$-category and let $\mathcal{T} \subset
\mathcal{C}$ be a full subcategory. 
Recall that $\mathcal{T}$ is called \emph{thick} 
if $\mathcal{T}$ is a stable subcategory which is also idempotent-complete,
so $\mathcal{T} \in \Catst$. 
Using the nilpotence theorem \cite{DHS88}, Hopkins and Smith gave in \cite[Thm.\ 
7]{HS98}
a complete classification of thick subcategories of the category of finite spectra. 
The classification of thick subcategories in general is a problem that has been
further studied in many different contexts and is closely
related to questions of nilpotence.

A key technique of this paper uses a specific enlargement $\mathcal{T}_{\epsilon}$ of a thick subcategory $\mathcal{T}\subset \mathcal{C}$.
Roughly speaking, $\mathcal{T}_{\epsilon}$ is the maximal enlargement of
$\mathcal{T}$ which is no different from $\mathcal{T}$ from the point of view
of any periodic localization. 
We will discuss 
some important examples of this construction and,
ultimately, a basic criterion that enables one to 
check whether an object belongs to $\mathcal{T}_{\epsilon}$ by performing a
\emph{rational} calculation (\Cref{nilpotenceforclg}).
All the results in the present section are fairly formal
consequences of the nilpotence
technology of \cite{DHS88, HS98}. 

\subsection{$\epsilon$-enlargements}

In this subsection, we make the basic definition. In the remaining subsections,
we will explore this further for specific choices of thick subcategories. 

\begin{definition} \label{Tplusepsilon}
 Let $\mathcal{C} \in \Catst$ and let $\mathcal{T} \subset
 \mathcal{C}$ be a thick subcategory. 
We will define several enlargements of $\mathcal{T}$. Recall that $\mathcal{C}$
is canonically tensored over finite spectra.
\begin{enumerate}
\item Given a finite spectrum $F$, define $\mathcal{T}_{F}$ to be the
smallest thick subcategory of $\mathcal{C}$ containing $\mathcal{T}$ and $\{F \otimes C\}_{C \in \mathcal{C}}$.
\item
Let $\Sigma$ be a finite set of prime numbers. 
Define thick subcategories $\mathcal{T}_{\epsilon, \Sigma}$ via 
\[ \mathcal{T}_{\epsilon, \Sigma} = \bigcap_F \mathcal{T}_F,  \]
where $F$ ranges over all finite spectra whose
$p$-localization is nontrivial for every $p \in \Sigma$. 
\item Finally, define the thick subcategory $\mathcal{T}_{\epsilon}$ via
\[ \mathcal{T}_{\epsilon} = \bigcup_{\Sigma} \mathcal{T}_{\epsilon, \Sigma},  \]
as $\Sigma$ ranges over all finite sets of prime numbers. 
We will call this the \emph{$\epsilon$-enlargement} of $\mathcal{T}$. 
\end{enumerate}
\end{definition}

The process of $\epsilon$-enlargement interacts well with exact functors.
\begin{prop} 
\label{functorialityplusepsilon}
Suppose $G\colon \mathcal{C} \to \mathcal{D}$ is a morphism in $\Catst$ (i.e., an
exact functor). Suppose $\mathcal{T} \subset \mathcal{C}$ and
$\mathcal{T}' \subset \mathcal{D}$ are thick subcategories and $G(\mathcal{T})
\subset \mathcal{T}'$. Then $G(\mathcal{T}_{\epsilon}) \subset
\mathcal{T}'_{\epsilon}$. Furthermore, for each finite set of prime numbers
$\Sigma$, we have $G( \mathcal{T}_{\epsilon, \Sigma}) \subset
\mathcal{T}'_{\epsilon, \Sigma}$.
\end{prop} 
\begin{proof} 
Let $F$ be any finite spectrum. One shows easily that $G( \mathcal{T}_F)
\subset \mathcal{T}'_F$ (because $G$ commutes with smashing with the finite
spectrum $F$) and then the remaining assertions follow formally by taking
intersections and unions. 
\end{proof} 

\begin{definition} 
\label{epsilonobject}
Let $\mathcal{C} \in \Catst$.
Suppose $\mathcal{T} = \left\{0\right\}$. 
Then write $\mathrm{Nil}_{\epsilon}(\mathcal{C}) \stackrel{\mathrm{def}}{=} \mathcal{T}_{\epsilon}$
and call it the subcategory of \emph{$\epsilon$-objects} of $\mathcal{C}$. 
In the case $\mathcal{C} = \sp$, we will call this the subcategory of
\emph{$\epsilon$-spectra.}
A morphism $f\colon X \to Y$ in ${\mathcal C}$
is an \emph{$\epsilon$-equivalence} if the cofiber is an $\epsilon$-object. 
Finally, we also write $\nil_{\epsilon, \Sigma}(\mathcal{C})$
for $\mathcal{T}_{\epsilon, \Sigma}$ for any finite set $\Sigma$ of prime numbers.
\end{definition} 

Our next result will be  \Cref{epsilonquotient}, which explains that the general
formation of ${\mathcal T}_\epsilon$ can be reduced to the special case that ${\mathcal T}=\{ 0\}$,
using Verdier quotients.

Suppose now that $\mathcal{C}$ is \emph{small}, i.e., $\mathcal{C} \in \catst$.
Suppose $\mathcal{T} \subset \mathcal{C}$ is a thick subcategory. 
Then we recall that we can form the (idempotent-complete) \emph{Verdier quotient}
$\mathcal{C}/\mathcal{T} \in \catst$ receiving an exact functor from
$\mathcal{C}$ (cf.\  \cite[\S
5.1]{BGT13} for a treatment; in particular, $\mathcal{C}/\mathcal{T}$ is the
pushout $\mathcal{C}\sqcup_{\mathcal{T}} 0$ in $\catst$). An object of
$\mathcal{C}$ belongs to $\mathcal{T}$ if and only if its image in the Verdier
quotient vanishes.

Note that \Cref{functorialityplusepsilon} implies that $\epsilon$-objects are
preserved by any exact functor in $\catst$. 
Applying this to the functor ${\mathcal C}\longrightarrow {\mathcal C}/{\mathcal T}$ shows that 
${\mathcal T}_\epsilon$ is mapped to Nil$_{\epsilon}({\mathcal C}/{\mathcal T})$, and to see the converse,
we first need an elementary lemma whose proof we leave to the reader.
We refer to \cite[Sec. 5.3]{Lur09} for the theory of
$\mathrm{Ind}$-completions of $\infty$-categories.

\begin{lemma} 
\label{thickvslocalizing}
Let $\mathcal{C} \in \catst$ and let $\mathcal{T}
\subset \mathcal{C}$ be a thick subcategory. Then $ \mathcal{T}=\mathcal{C} \cap \mathrm{Ind}(\mathcal{T}) 
\subset \mathrm{Ind}(\mathcal{C})$.
\end{lemma} 

\begin{prop} 
\label{epsilonquotient}
Let $\mathcal{C} \in \catst$ and let $\mathcal{T} \subset \mathcal{C}$ be a
thick subcategory. Then an object $X \in \mathcal{C}$ belongs to
$\mathcal{T}_{\epsilon}$ if and only if its image in $\mathcal{C}/\mathcal{T}$
is an $\epsilon$-object.
\end{prop} 
\begin{proof} 
The `only if' implication follows from \Cref{functorialityplusepsilon}, applied to the
canonical functor $\mathcal{C}\longrightarrow \mathcal{C}/\mathcal{T}$. 
For the converse, suppose $\Sigma$ is a finite set of prime numbers and the image $\overline{X}
\in \mathcal{C}/\mathcal{T}$ of $X$ belongs to $\mathrm{Nil}_{\epsilon,
\Sigma}(\mathcal{C}/\mathcal{T})$. 
We show that $X \in \mathcal{T}_{\epsilon, \Sigma}$. 

Let $F$ be any finite spectrum such that $F_{(p)} \neq 0$ for each $p \in \Sigma$. 
We need to show that $X$ belongs to $\mathcal{T}_F$. 
Equivalently, by \Cref{thickvslocalizing}, we need to show that $X$ belongs to $\mathrm{Ind}(\mathcal{T}_F)\subset \mathrm{Ind}(\mathcal{C})$.
In $\mathrm{Ind}(\mathcal{C})$ we can localize to obtain a cofiber sequence,
\[ X' \to X \to X'',  \]
where $X' \in \mathrm{Ind}(\mathcal{T}) \subset \mathrm{Ind}(\mathcal{C})$
while $\hom_{\mathrm{Ind}(\mathcal{C})}(T, X'') = 0$ for $T \in \mathcal{T}$.
Thus, it suffices to show that $X''$ belongs to the localizing subcategory
generated by $\{F \otimes Y\}_{Y \in \mathcal{C}}$.

Finally, we recall (cf.\ \cite[\S 5.1]{BGT13}) that the collection of $Z \in \mathrm{Ind}(\mathcal{C})$
with $\hom_{\mathrm{Ind}(\mathcal{C})}(T, Z) =0 $ for each $T \in \mathcal{T}$ can be
identified with $\mathrm{Ind}(\mathcal{C}/\mathcal{T})$
and $X''$ 
can be identified with 
the image $\overline{X} \in \mathcal{C}/\mathcal{T} \subset
\mathrm{Ind}(\mathcal{C}/\mathcal{T})$. 
So, our assumption that $\overline{X} \in \mathrm{Nil}_{\epsilon, \Sigma}(
\mathcal{C}/\mathcal{T})$ now shows that $X''$ belongs to the localizing
subcategory as desired. 
\end{proof} 

We next observe that being an $\epsilon$-object can be checked locally at one prime at a time.

\begin{prop} 
\label{spliteps}
Let $\mathcal{C} \in \Catst$. 
Fix
an object $X \in \mathcal{C}$. Then $X \in \nil_{\epsilon}(\mathcal{C})$ if and
only if the following conditions are satisfied: 
\begin{enumerate}
\item There exists $N \in \mathbb{Z}_{>0}$ such that $N = 0 \in
\pi_0\hom_{\mathcal{C}}(X, X)$. Therefore, we get a canonical decomposition $X
\simeq \bigvee_{p \mid N} X_{(p)}$ where $X_{(p)}$ is the $p$-localization of $X$. 
\item In the above decomposition, for each $p$ dividing $N$, $X_{(p)} \in \nil_{\epsilon,
\left\{p\right\}} (\mathcal{C})$.
\end{enumerate}
\end{prop} 
\begin{proof} 
We show that every $X \in \nil_{\epsilon}(\mathcal{C})$ satisfies conditions (1) and
(2), and leave the (easier) proof of the converse implication to the reader.
There is a finite set of primes $\Sigma$ such that $X \in \nil_{\epsilon, \Sigma}(\mathcal{C})$. 
Then one sees that there exists $N \in \mathbb{Z}_{>0}$, divisible only by the
prime numbers belonging to $\Sigma$,
such that $N \cdot\mathrm{id}_X = 0 $ as a self-map of $X$, via a thick subcategory
argument. In fact, $X$ belongs to the thick subcategory generated by
$\left\{ \left(\bigvee_{p\in\Sigma} S^0 /p\right) \otimes C\right\}_{p \in \Sigma, C \in \mathcal{C}}$, and any
such object has this property. 
Thus $X$ satisfies (1) and in particular decomposes into a direct sum of its $p$-localizations for $p \in
\Sigma$. 

Fix a prime number $p \in \Sigma$ and a nontrivial $p$-torsion finite complex $F$. 
It suffices to show that $X_{(p)}$ belongs to the thick subcategory generated by
$\left\{F \otimes C\right\}_{C \in \mathcal{C}}$. 
For each $q \in \Sigma \setminus \left\{p\right\}$, choose  a $q$-torsion finite complex  $F_q$. 
Then by assumption, $X$ belongs to the thick subcategory generated by the union
of 
$\left\{F \otimes C\right\}_{C \in \mathcal{C}}$ and $\left\{F_q
\otimes C\right\}_{q \in \Sigma \setminus \left\{p\right\}, C \in \mathcal{C}}$.
Applying $p$-localization annihilates the second set of objects and proves the
desired claim about $X_{(p)}$. 
\end{proof}

We will now give basic criteria for identifying $\epsilon$-objects, which shows that they are  `strongly null-isogenous' or `endomorphism-dissonant.' Here we will freely use the nilpotence technology of \cite{DHS88, HS98}.

Recall that for every $n$ and prime $p$, there exists a type $n$, finite
$p$-local spectrum $F_n$ admitting a non-trivial $v_n$-self map $f$.  Let
$T(n)=F_n[f^{-1}]$ be the mapping telescope on $f$. While the spectrum $T(n)$
depends on the choice of $F_n$, its Bousfield class does not. In other words,
the condition $T(n)_*X=0$ (for a spectrum $X$) does not depend on the choices made.

\begin{prop} 
\label{epsilonobjectend}
Let $\mathcal{C} \in \Catst$ and $X \in \mathcal{C}$.  The following are equivalent:
\begin{enumerate}
\item \label{it:x-eps} The object $X$ is an $\epsilon$-object.
\item \label{it:x-tel} The endomorphism ring spectrum $\mathrm{End}_{\mathcal{C}}(X)$ has the property that $T(n)
_*\mathrm{End}_{\mathcal{C}}(X)  = 0$ for all $n \in [0, \infty)$ and all
primes $p$.  
\item \label{it:x-kn} The endomorphism ring spectrum $\mathrm{End}_{\mathcal{C}}(X)$ has the property that $K(n)
_*\mathrm{End}_{\mathcal{C}}(X)  = 0$ for all $n \in [0, \infty)$ and all
primes $p$.  
\end{enumerate}
\end{prop} 
\begin{proof}
	First we prove (\eqref{it:x-eps}$\Rightarrow$\eqref{it:x-tel}), i.e., if $X$ is an $\epsilon$-object, then $T(n)_*\mathrm{End}_{\mathcal{C}}(X)=0$ for every $n$ and implicit prime $p$. By \Cref{spliteps}, we can suppose that $X\in \nil_{\epsilon,
\left\{p\right\}} (\mathcal{C})$.
	Now for a fixed $n\geq 0$, we choose a type $n$, finite $p$-local spectrum
	$F_n$ with a $v_n$-self map $f$ such that $T(n)=F_n[f^{-1}]$ as above. Let
	$Cf$ denote the cofiber of $f$. We will show that the class $\mathcal{D}$ consisting of all $Y\in
	\mathcal{C}$ such that $T(n)_*\hom_{\mathcal{C}}(X,Y)=0$ contains $X$.
	Evidently, $\mathcal{D}$ is a thick subcategory and by construction of $Cf$,
	\[T(n)_*\hom_{\mathcal{C}}(X,Cf\otimes Z)\simeq T(n)_*\left( Cf\otimes
	\hom_{\mathcal{C}}(X,Z)\right)=0, \quad \text{for any } Z \in \mathcal{C},\] 
	which implies that any $Cf \otimes Z$ belongs to $ \mathcal{D}$. 
	By hypothesis, $X$ belongs to the thick subcategory generated by
	$\{Cf\otimes Z\}_{Z \in \mathcal{C}}$ and thus $X \in \mathcal{D}$, so the claim follows. 

	The implication (\eqref{it:x-tel}$\Rightarrow$\eqref{it:x-kn}) holds because
	any $T(n)_*$-acyclic spectrum $W$ is also $K(n)_*$-acyclic: If $0=T(n)_*W=\pi_*
	(T(n)\otimes W)$, then one concludes $0=K(n)_*(T(n)\otimes W))\simeq
	K(n)_*(T(n))\otimes_{K(n)_*}K(n)_*(W)$. Since $K(n)_*(T(n))\neq 0$ 
	and $K(n)_*$ is a graded field, this implies that $K(n)_*(W)=0$.

For the final implication (\eqref{it:x-kn}$\Rightarrow$\eqref{it:x-eps}), suppose that the 
endomorphism ring spectrum
$\mathrm{End}_{\mathcal{C}}(X)$ is $K(n)$-acyclic for all primes $p$ and 
$0 \leq n < \infty$. First of all, observe that $\mathrm{End}_{\mathcal{C}}(X)$
is therefore a torsion spectrum. 
Breaking $X$ into a direct sum of its $p$-localizations, we may assume that
$\mathrm{End}_{\mathcal{C}}(X)$ is $p$-power torsion. We then need to show that
$X \in \nil_{\epsilon, \left\{p\right\}}(\mathcal{C})$.

Let $F$ be any finite nontrivial $p$-torsion spectrum. We will show that $X$
belongs to the thick subcategory generated by $F \otimes X$. 
Replacing $F$ by $F \otimes \mathbb{D}F = \mathrm{End}_{\sp}(F)$, we may assume $F$ is a ring spectrum
and we let $S^0 \to F$ be the unit map. We let $I \to S^0$ denote the fiber of
the unit map. 
Note that $F$ has nontrivial homology in degree zero, so $I \to S^0$ induces the zero map in
mod $p$ homology. 
We will show that there exists $n$ such that $I^{\otimes n} \otimes X \to X$ is
nullhomotopic. It will follow that $X$ is a summand of the cofiber of this
map, which belongs to the thick subcategory generated by $F\otimes X$ as desired. 

To show that there exists an $n$ such that $I^{\otimes n} \otimes X \to X$ is
nullhomotopic, we observe (cf. \cite[Sec. 2]{HS98}) that it suffices to show (by adjointness) that there exists an $n$, such
that the map of \emph{spectra}
\begin{equation}\label{eq:nilp}
I^{\otimes n} \to S^0 \to \mathrm{End}_{\mathcal{C}}(X)
\end{equation} is nullhomotopic. 
Indeed, maps of spectra $I^{\otimes n} \to \mathrm{End}_{\mathcal{C}}(X)$
correspond to maps $I^{\otimes n} \otimes X \to X$ in $\mathcal{C}$. 
By duality, we need to show that the corresponding map $S^0\to \mathrm{End}_{\mathcal{C}}(X)\otimes \mathbb{D}I^{\otimes n}$ is null for $n\gg 0$. By repeatedly applying the unit map for $\mathrm{End}_{\mathcal{C}}(X)$ we would then obtain a null-homotopic map $S^0\to\mathrm{End}_{\mathcal{C}}(X)^{\otimes n}\otimes \mathbb{D}I^{\otimes n}$. By using the multiplication map on $\mathrm{End}_{\mathcal{C}}(X)$, we see that the null-homotopy of \eqref{eq:nilp} for $n\gg 0$ is \emph{equivalent to} the \emph{smash nilpotence} of
 the map $S^0 \to \mathrm{End}_{\mathcal{C}}(X) \otimes
\mathbb{D}I$. 

By the nilpotence theorem
\cite[Thm.\ 3]{HS98}, it suffices to check this after applying $K(n)_*$ for $n \in [0,
\infty]$. For $n < \infty$, we know that 
$\mathrm{End}_{\mathcal{C}}(X)$ itself is $K(n)_*$-acyclic. For $n = \infty$,
so that $K(\infty) = H \mathbb{F}_p$, we know that $S^0 \to \mathbb{D} I$
induces the zero map in homology, so we are done here too. 
\end{proof}

\begin{cor}\label{detecteps}
Given $\mathcal{C} \in \Catst$ and an object $X\in\mathcal{C}$, the following
are equivalent: 
\begin{enumerate}
\item  
$X$ is an
$\epsilon$-object. 
\item 
For every prime $p$ and height $n\geq 0$ and
every exact functor $F\colon \mathcal{C}\rightarrow L_{T(n)}\sp$, the image of $X$
under $F$ is zero.  
\item  For every prime $p$ and height $n\geq 0$ and
every exact functor $F \colon \mathcal{C}\rightarrow L_{K(n)}\sp$, the image of $X$
under $F$ is zero.  

\end{enumerate}
\end{cor} 
\begin{proof} To see that ((1)$\Rightarrow$(2) and (3)), it suffices by \Cref{functorialityplusepsilon}
to show that there are
no nontrivial $\epsilon$-objects of $L_{K(n)} \sp$ and $L_{T(n)} \sp$. Let $C$
be a type $n + 1$ finite complex.  Then smashing with $C$ annihilates both
$L_{K(n)} \sp$ and $L_{T(n)} \sp$. This easily implies that there are no
nontrivial $\epsilon$-objects as desired. For the other direction, i.e., to
show ((2) or (3)$\Rightarrow$(1)), we use \Cref{epsilonobjectend} 
and consider the functors
$L_{K(n)}\hom_{\mathcal{C}}(X, \cdot) \colon \mathcal{C} \to L_{K(n)}
\sp$ (or the telescopic analog). \end{proof} 

\begin{remark}
	The results above indicate that our methods do not see any possible
	distinction between $T(n)-$ and $K(n)$-localizations. As a consequence, our
	descent results will be applicable to any of the following localization
	functors which we will call \emph{periodic} localization functors (and
	discuss 
	below):
	$\{L_{T(n)}, L_{K(n)}, L_n^f, L_n\}_{n\geq 0, \, p\textrm{ prime}}.$ Of course, the results for all of the telescopic localizations imply the others.
\end{remark}

\begin{example}
Here are some examples of $\epsilon$-objects.

\begin{enumerate}
\item If $\mathcal{C}=\sp^{\omega}$ is the $\infty$-category of finite spectra, then every $\epsilon$-object is $0$: Using \Cref{spliteps} and \Cref{epsilonobjectend}, it suffices to see that if $X$ is the $p$-localization of a finite spectrum with $K(n)_*(X)=0$ for all $0\leq n<\infty$,
then $X$ is contractible. Since for sufficiently large $n$, the Atiyah-Hirzebruch spectral sequence $H^*(X,\mathbb{F}_p)\Rightarrow K(n)^*(X)$ collapses, this is clear.
\item If $\mathcal{C}$ is such that for every prime $p$, the $p$-localization $\mathcal{C}_{(p)}$
is $\mod(L_nS)$ for some $n$ (possibly depending on $p$), then 
every $\epsilon$-object of $\mathcal{C}$ is zero. This is immediate from
\Cref{epsilonobjectend} and applies for example to the $E_n$-local category,
the $K(n)$-local category, the $\infty$-category of $KO$-modules, and the
$\infty$-category of $\Tmf$-modules.
\item If $\mathcal{C}=\sp$, then every $X\in \sp$ which admits the structure
of an $H\mathbb{Z}$-module spectrum and is annihilated by some $N\geq 1$ is an $\epsilon$-object.  So is  any spectrum in the thick subcategory generated by these, e.g., any spectrum with only finitely many non-zero homotopy groups, each of which is annihilated by some $N$.
\item If each mapping spectrum in $\mathcal{C}$ is an $H\mathbb{Z}$-module, for example if $\mathcal{C}$ arises from a dg-category, then the $\epsilon$-objects are exactly the objects annihilated by some $N\in\mathbb{N}$.
\item If $\mathcal{C}$ is a full subcategory of $\mathcal{D}$, then an $X\in\mathcal{C}$ is an $\epsilon$-object of $\mathcal{C}$ if and only if it is an $\epsilon$-object of $\mathcal{D}$.

\item $H \mathbb{Q}/\mathbb{Z} \in \sp$ is an example of a spectrum, even an $H\mathbb{Z}$-module, which is
$T(n)$-acyclic for every $n < \infty$ but is \emph{not} an $\epsilon$-spectrum. 
This follows because for every $N$, multiplication by $N$ is nonzero on
$H\mathbb{Q}/\mathbb{Z}$.
\end{enumerate}
\end{example}

\begin{prop} 
Suppose $\mathcal{C} \in \Catst$.
Given a thick subcategory $\mathcal{T} \subset \mathcal{C}$, we have
$(\mathcal{T}_{\epsilon})_{\epsilon} = \mathcal{T}_{\epsilon}$. In other words,
$\epsilon$-enlargement is an idempotent procedure.
\end{prop} 
\begin{proof} 
Without loss of generality, we can assume $\mathcal{C} \in \catst$ by writing
$\mathcal{C}$ as a union of small subcategories. 
Suppose $X \in (\mathcal{T}_{\epsilon})_{\epsilon}$. 
By \Cref{epsilonquotient}, we need to show that the image of $X$ 
in $\mathcal{C}/\mathcal{T}$ is an $\epsilon$-object. 
Note that the quotient map carries $\mathcal{T}_{\epsilon}$ into
$\mathrm{Nil}_{\epsilon}(\mathcal{C}/\mathcal{T})$ and therefore it carries
$(\mathcal{T}_{\epsilon})_{\epsilon}$ into 
$(\mathrm{Nil}_{\epsilon}(\mathcal{C}/\mathcal{T}))_{\epsilon}$ by
\Cref{functorialityplusepsilon}. 
As a result, we need to show that
$$(\mathrm{Nil}_{\epsilon}(\mathcal{C}/\mathcal{T}))_{\epsilon} =
\mathrm{Nil}_{\epsilon}(\mathcal{C}/\mathcal{T})$$
because this will imply that the image of $X$ in $\mathcal{C}/\mathcal{T}$ is
an $\epsilon$-object.

The upshot of this discussion 
is that, by passage from $\mathcal{C}$ to $\mathcal{C}/\mathcal{T}$, we can assume $\mathcal{T} =0 $ to begin with, and we make this
assumption. So, assume $X \in (
\mathrm{Nil}_\epsilon(\mathcal{C}))_{\epsilon}$; we show that $X$ is itself an $\epsilon$-object.  

By \Cref{detecteps}, it suffices to show that any exact functor  $F \colon \mathcal{C} \to L_{K(n)}
\sp$ annihilates $X$. As we saw above, $F$ annihilates
$\mathrm{Nil}_{\epsilon}(\mathcal{C})$ and therefore factors through 
an exact functor $\overline{F} \colon \mathcal{C}/\mathrm{Nil}_{\epsilon}(\mathcal{C})
\to L_{K(n)} \sp$. However, our assumption is that the image of $X$ in
$\mathcal{C}/\mathrm{Nil}_{\epsilon}(\mathcal{C})$ is an $\epsilon$-object and
therefore $\overline{F}$ annihilates it. 
\end{proof} 

\begin{remark} 
The collection of $\epsilon$-objects is preserved even by arbitrary
\emph{additive} functors between idempotent-complete, stable $\infty$-categories.
Compare \cite[Sec.\ 2]{GGN15} for a treatment of additive $\infty$-categories. 
In fact, if $F\colon \mathcal{C} \to \mathcal{D}$ is an additive functor, then for
each $X \in \mathcal{C}$, we obtain a canonical map of spectra
\[ \tau_{\geq 0}\hom_{\mathcal{C}}(X, X)  \to \tau_{\geq 0}
\hom_{\mathcal{D}}(F(X), F(X)), \]
which carries the unit in $\pi_0$ to the unit. 
Using \Cref{epsilonobjectend} (and the fact that smashing with $K(n)$ annihilates 
bounded-above spectra for $n > 0$) one sees easily that $F(X)$ is also an
$\epsilon$-object.
\end{remark}

We next check that $\epsilon$-enlargements are invisible to all finite chromatic localizations.
We recall briefly the theory of \emph{finite localizations} (cf.\ \cite{Miller92}).
Let $\mathcal{C}$ be a presentable stable $\infty$-category. 
Fix a prime $p$ and a height $n$. Then we define a localization functor $L^f_n\colon
\mathcal{C}\to \mathcal{C}$ as follows. When $n =0$, then $L^f_n$ is simply rationalization. When
$n > 0$,  choose a nontrivial type $n+1$, finite complex $F$.
 Consider the full subcategory of all $X \in \mathcal{C}$ such that $F \otimes X$ is
contractible. This subcategory is closed under all limits and colimits and is
the image of a colimit-preserving, idempotent functor $L^f_n\colon \mathcal{C} \to \mathcal{C}$. 
We note that $L^f_n$ annihilates any object in $\mathcal{C}$ of the form $F
\otimes X$, and that as a result $L^f_n$ annihilates
$\nil_{\epsilon}(\mathcal{C})$. Note also that one has $L_n$ localization
functors on $\mathcal{C}$, given as localization with respect to smashing with
Morava $E$-theory 
$E_n$, and $L_n = L_n L_n^f$. 
We have the following:

\begin{prop} 
\label{Lnepsilon}
Let $\mathcal{C}$ be a presentable, stable $\infty$-category and let
$\mathcal{T} \subset \mathcal{C}$ be a thick subcategory.  For any prime $p$
and height $n$, we have $L_n^f \mathcal{T} = L_n^f \mathcal{T}_{\epsilon}$,
i.e., $\mathcal{T} $ and $ \mathcal{T}_{\epsilon}$ have the same image under the
finite localization functor $L_n^f$.
The same holds for $L_n$ replacing $L_n^f$.
\end{prop} 
\begin{proof} 
We have an inclusion 
$L_n^f \mathcal{T} \subset  L_n^f \mathcal{T}_{\epsilon}$, so it suffices to
show the other inclusion. 
Let $\Sigma$ be any finite set of prime numbers and let $F$ be a finite
torsion spectrum such that $F_{(q)} \neq 0$ for any $q \in \Sigma$ and such
that $L_n^f F  = 0$. Then, it follows that $\mathcal{T}_{\epsilon, \Sigma}
\subset \mathcal{T}_F$ where we  use the notation  of \Cref{Tplusepsilon}. 
One sees now that $L_n^f \mathcal{T}_F = L_n^f \mathcal{T}$, so we are done. 
The assertion for $L_n$ follows because $L_n = L_n L_n^f$. 
\end{proof}

We now consider the case where 
$\mathcal{C} \in \clg(\Catst)$. 
That is,
$(\mathcal{C}, \otimes, \mathbf{1})$ is a not necessarily small, symmetric monoidal,
idempotent-complete stable $\infty$-category, where the tensor structure $\otimes$ is exact in each variable. 
Recall that a thick subcategory $\mathcal{I} \subset \mathcal{C}$ is called a
\emph{thick $\otimes$-ideal} if for each $X \in \mathcal{I} $ and $ Y \in
\mathcal{C}$, we have $X \otimes Y \in \mathcal{I}$. 
\begin{prop}
If $\mathcal{I} \subset \mathcal{C}$ is a thick $\otimes$-ideal, then so is
its $\epsilon$-enlargement $\mathcal{I}_\epsilon\subset\mathcal{C}$.
\end{prop} 
\begin{proof} 
We need to show that if $Y \in \mathcal{C}$, then the exact functor $(-)\otimes Y\colon
\mathcal{C} \to \mathcal{C}$ preserves $\mathcal{I}_{\epsilon}$. This follows
from \Cref{functorialityplusepsilon} because the functor preserves
$\mathcal{I}$. 
\end{proof}

\subsection{$\epsilon$-nilpotent towers}
We recall some basic definitions when working with \emph{towers} of objects. 
We refer to \cite{HPS} or \cite[\S 3]{Ma15} for more details. 
\begin{definition} 
\label{quickconv}
Suppose $\mathcal{C} \in \Catst$.
Let $\tow(\mathcal{C})$ denote the $\infty$-category 
$\mathrm{Fun}(\mathbb{Z}_{\geq 0}^{\mathrm{op}}, \mathcal{C})$
of \emph{towers} in $\mathcal{C}$, and observe that $\tow(\mathcal{C})\in\Catst$.
We will often abuse notation and denote a tower by $\left\{X_i\right\}$, suppressing the maps $X_{i+1} \to X_i$ from the notation.

Recall the following definitions of two important thick subcategories of $\tow(\mathcal{C})$:
\begin{enumerate}
\item  
Let $\townil(\mathcal{C}) \subset \tow(\mathcal{C})$ denote the full subcategory
spanned by those towers $\left\{X_i\right\}$ such that there exists $N \in
\mathbb{Z}_{\geq 0}$ such that the maps $X_{i+N} \to X_i$ are nullhomotopic for
all $i \in \mathbb{Z}_{\geq 0}$.  
Such towers will be called \emph{nilpotent.}
\item Let $\towconst(\mathcal{C})$ denote the thick subcategory of
$\tow(\mathcal{C})$ generated by $\townil(\mathcal{C})$ and the constant towers. 
Such towers will be called \emph{quickly converging.}
\end{enumerate}
\end{definition}

We will now consider \Cref{Tplusepsilon} for these thick subcategories of the
$\infty$-category of towers. 
Recall that given a cosimplicial diagram $F^\bullet  \in \fun( \Delta, \mathcal{C})$,
the \emph{totalization} $\mathrm{Tot}(F^\bullet)$ is the homotopy inverse limit
of $F$ (if it exists). We will also use the \emph{partial totalizations} $\mathrm{Tot}_i( F^\bullet) =
\varprojlim_{\Delta^{\leq i}} F$ of the restriction of $F$ to the subcategory
$\Delta^{\leq i} = \left\{[0], [1], \dots, [i]\right\} \subset \Delta$. 

\begin{definition}

Suppose $\mathcal{C} \in \Catst$.
\begin{enumerate}
\item 
Define $\towen(\mathcal{C})$ via 
\[ \towen(\mathcal{C}) = \left(
\mathrm{Tow}^{\mathrm{nil}}(\mathcal{C}) \right)_{\epsilon} .\]
A tower is \emph{$\epsilon$-nilpotent} if it belongs
to $\towen(\mathcal{C})$.
\item Let $\left\{X_i\right\}$ be a tower in $\mathcal{C}$. Suppose given an
object $X \in \mathcal{C}$ and a  map
of towers $\left\{X\right\} \to \left\{X_i\right\}$ (if $\mathcal{C}$ admits
inverse limits, then this is equivalent to giving a map $X \to \varprojlim
X_i$). We will say that this map exhibits  $X$ as an \emph{$\epsilon$-nilpotent limit}
of the tower $\left\{X_i\right\}$ if the cofiber tower $\left\{X_i/X\right\}$
belongs to $\towen(\mathcal{C})$.
\item Let $X^\bullet \in \fun(\Delta^+, \mathcal{C})$ be an augmented
cosimplicial object of $\mathcal{C}$. 
We say that it is an \emph{$\epsilon$-nilpotent limit diagram} if the natural map of towers
$\{X^{-1}\} \to \left\{\mathrm{Tot}_i( X^\bullet)\right\}$ exhibits $X^{-1}$ as an
$\epsilon$-nilpotent limit of the target.

\end{enumerate}
\end{definition} 

\begin{prop} 
The  $\epsilon$-nilpotent limit diagrams form a thick subcategory of $\fun(\Delta^+, \mathcal{C})$. 
\end{prop}
\begin{proof} 
Consider the functor $F\colon \fun( \Delta^+, \mathcal{C}) \to \tow(\mathcal{C})$ which to $X^\bullet$
associates the cofiber tower $\{ \mathrm{Tot}_i(X^\bullet)/X^{-1}\}$. This is an exact functor, and by definition, 
the $\epsilon$-nilpotent limit diagrams are exactly those objects
such that this cofiber tower belongs to $\mathrm{Tow}^{\epsilon,\mathrm{nil}}(\mathcal{C})$; in particular,
they form a thick subcategory. \end{proof}

Using \Cref{functorialityplusepsilon}, we easily obtain the functoriality of
$\epsilon$-nilpotent towers. 

\begin{prop} 
\label{epsilonnilpfunctoriality}
Let $\mathcal{C}, \mathcal{D} \in \Catst$ and let $F\colon
\mathcal{C} \to \mathcal{D}$ be an exact functor. 
Clearly, $F$ induces an exact functor $F_*\colon \tow(\mathcal{C}) \to \tow(\mathcal{D})$. 
Then: 
\begin{enumerate}
\item $F_*(\townil(\mathcal{C})) \subset \townil(\mathcal{D})$.
\item $F_*( \towen(\mathcal{C})) \subset \towen(\mathcal{D})$.
\item Suppose $\left\{X_i\right\}$ is a tower in $\mathcal{C}$ and
$\left\{X\right\} \to \left\{X_i\right\}$ is a map of towers. Suppose that this
map exhibits $X$ as an $\epsilon$-nilpotent limit of $\left\{X_i\right\}$. Then the
induced map $\left\{F(X)\right\} \to 
\left\{F(X_i)\right\}$
exhibits $F(X)$ as an $\epsilon$-nilpotent limit of $\left\{F(X_i)\right\}$. 
\item Let $X^\bullet \in \fun( \Delta^+, \mathcal{C})$ be an augmented
cosimplicial object. Suppose $X^\bullet$ is an $\epsilon$-nilpotent limit diagram.  
Then the augmented cosimplicial object $F(X^\bullet) \in \fun(\Delta^+,
\mathcal{D})$ is an $\epsilon$-nilpotent limit diagram.
\end{enumerate}
\end{prop} 

For our purposes, it will be important to 
know that $\epsilon$-nilpotent limit diagrams turn into actual limit diagrams after
periodic localization. 
\begin{prop}\label{prop:epsilon-limit}
Suppose $\mathcal{C}$ is a presentable stable $\infty$-category and $X^\bullet \in
\fun( \Delta^+, \mathcal{C})$ is an $\epsilon$-nilpotent limit diagram. 
Let $L^f_n$ be any finite localization functor (associated to a prime $p$ and a 
height $n$). 
Then:
\begin{enumerate}
\item  
The map 
\[ X^{-1} \to \mathrm{Tot}(X^{\bullet})  \]
is an $\epsilon$-equivalence. 
\item The localized Tot tower
$\left\{L_n^f\mathrm{Tot}_i(X^\bullet)\right\}$ is quickly converging. 
\item
The maps
\[ L^f_n X^{-1} \to L^f_n \mathrm{Tot}( X^\bullet) \to \mathrm{Tot}( L^f_n
X^{\bullet})  \]
are equivalences.
\end{enumerate}
\end{prop} 
\begin{proof} 
We consider items (1) and (3) first.
Setting $\{ Y_i\} := \{ \mathrm{Tot}_i(X^\bullet)/X^{-1}\}$, one reduces to showing that if $\left\{Y_i\right\} \in \towen(\mathcal{C})$,
then 
$\varprojlim Y_i$ is an $\epsilon$-object and 
\[ L^f_n \varprojlim Y_i \simeq \varprojlim L^f_n Y_i \simeq 0.  \]

In fact, consider the two exact functors
\[ F_1, F_2\colon \tow(\mathcal{C}) \to \mathcal{C}, \quad F_1( \left\{Y_i\right\})
= L_n^f\varprojlim_i Y_i, \quad F_2(\left\{Y_i\right\}) = \varprojlim_i L_n^f
Y_i . \]
Clearly if $\left\{Y_i\right\} \in \townil(\mathcal{C})$, then
$F_1(\left\{Y_i\right\}) = F_2(\left\{Y_i\right\}) = 0$. 
It follows from \Cref{functorialityplusepsilon} that if 
$\left\{Y_i\right\} \in \towen(\mathcal{C})$, then $F_1(\left\{Y_i\right\}),
F_2(\left\{Y_i\right\})$ are both $\epsilon$-objects of $\mathcal{C}$. Since
they are both $L_n^f$-local, however, it follows that they must be contractible. 
Similarly, considering the functor $F_3 \colon 
\tow(\mathcal{C}) \to \mathcal{C}$ given by $ F_3( \left\{Y_i\right\}) =
\varprojlim Y_i$, we conclude that $\varprojlim Y_i$ is an $\epsilon$-object. 

To prove item (2), it suffices to show that if $\left\{Y_i\right\} \in
\towen(\mathcal{C})$, then the tower $\left\{L_n^f Y_i\right\}$ belongs to $ 
\townil(\mathcal{C})$. This follows from 
\Cref{Lnepsilon}, which implies that $\townil(\mathcal{C})$ and
$\towen(\mathcal{C})$ have the same image under $L_n^f$. 
\end{proof}

\subsection{$(A, \epsilon)$-nilpotent objects}
Let $\mathcal{C} \in \clg(\Catst)$ be a symmetric monoidal, stable,
idempotent-complete $\infty$-category with biexact tensor product and let $A \in
\mathrm{Alg}(\mathcal{C})$. 

\newcommand{\anil}{\mathrm{Nil}^A}
\newcommand{\anile}{\mathrm{Nil}^{A, \epsilon}}

\begin{definition}[Cf. {\cite[Def. 3.7]{Bou79}}] 
The subcategory $\anil \subset \mathcal{C}$ is the thick
$\otimes$-ideal of $\mathcal{C}$ generated by $A$. 
\end{definition} 

We refer to \cite{MNN15i} for a detailed treatment of the theory of
$A$-nilpotence and to \cite{MGal} for some applications to analogs of faithfully
flat descent theorems. 
As before, we can define an $\epsilon$-version of the above following the
same pattern. 
\begin{definition} 
Define $ \anile = (\anil)_{\epsilon}$  as in \Cref{Tplusepsilon}
and call this the subcategory of \emph{$(A, \epsilon)$-nilpotent objects} in
$\mathcal{C}$. 
\end{definition}

We recall the following result. Although the idea is surely classical, 
we refer to \cite[Prop.\ 4.7]{MNN15i} for a modern  exposition. 
Consider the augmented cobar construction  $\cbaug(A)\colon \Delta^+ \to
\mathcal{C}$. The underlying cosimplicial object 
is the cobar construction $\cb(A) \in \fun(\Delta, \mathcal{C})$
\[ A \rightrightarrows A \otimes A \triplearrows \dots,  \]
and the augmentation is from the unit.

\begin{prop} 
Let $\mathcal{C} \in \clg(\Catst)$  and let $A \in
\mathrm{Alg}(\mathcal{C})$. 
Suppose $X \in \mathcal{C}$ is $A$-nilpotent. 
Then the augmented cosimplicial object $\cbaug(A) \otimes X$ is a limit diagram
and the associated $\mathrm{Tot}$ tower is quickly converging. 
\end{prop} 

As a consequence, we can deduce an $\epsilon$-version of the above. 
\begin{prop} 
\label{nilpcobareplimit}
Let $\mathcal{C} \in \clg(\Catst)$ and let $A \in
\mathrm{Alg}(\mathcal{C})$. 
If $X \in \mathcal{C}$ is $(A, \epsilon)$-nilpotent, then  
the augmented cosimplicial object $\cbaug(A) \otimes X$ is an
$\epsilon$-nilpotent limit diagram.  
\end{prop} 
\begin{proof} 
We have an exact functor $\mathcal{C} \to \tow(\mathcal{C})$ which sends
$$X \mapsto \left\{\mathrm{cofib}(X \to \mathrm{Tot}_i( \cb(A) \otimes
X))\right\}.$$ 
This carries $\anil$
into $\townil(\mathcal{C})$  by the previous proposition. 
Therefore, it carries $\anile$ into $\towen(\mathcal{C})$ by
\Cref{functorialityplusepsilon}, which completes the proof. 
\end{proof} 

\section{Noncommutative motives}\label{sec:nc_motives}

\newcommand{\nmot}{\mathcal{N}\mathrm{Mot}}

In this section, we will set up more machinery needed for our descent theorems
in algebraic $K$-theory. These will take place in an appropriate 
$\infty$-category 
of non-commutative motives
which, following \cite{BGT13}, is universal 
as a target for certain invariants of $\infty$-categories with some
additional structure. 

\subsection{General motives}

We begin by reviewing the point of view on connective algebraic $K$-theory described in
the papers of Blumberg-Gepner-Tabuada \cite{BGT13, BGT14} and its
generalization to the setting of stable
$\infty$-categories linear over a fixed base, which has been developed by
Hoyois-Scherotzke-Sibilla \cite{HSS15}. 
Our treatment will essentially follow theirs. 
For our purposes, however, it will be necessary to work with a slight variant since we
want to consider invariants of $\infty$-categories that do not necessarily
commute with filtered colimits, such as topological cyclic homology. 
This will not change the essential ideas. 

As before, we consider the $\infty$-category $\catst$ of small, idempotent-complete, stable $\infty$-categories and
exact functors between them. 
Using the Lurie tensor product, $\catst$ is a symmetric monoidal
$\infty$-category. 
One has a lax symmetric monoidal functor
\[ K(-)\colon   \catst \to \sp \]
given by (connective) $K$-theory. 
In the setup of \cite{BGT13, BGT14}, one constructs a presentable, stable
symmetric monoidal $\infty$-category
$\mathrm{Mot}$ of \emph{non-commutative motives.}
This receives a symmetric monoidal functor ${\mathcal U}_{\mathrm{add}} \colon \catst \to \mathrm{Mot} 
$, and given $\mathcal{C} \in \catst$, we can identify
$K(\mathcal{C})$  with the mapping spectrum from the unit
into the object ${\mathcal U}_{\mathrm{add}}(\mathcal{C} ) \in \mathrm{Mot}$. 
The $\infty$-category $\mathrm{Mot}$ satisfies a 
\emph{universal} property for
receiving maps from $\catst$ that
satisfy certain properties. 
We briefly review a general form of their construction (cf.\  \cite[\S 5]{BGT14}).

Let $\mathcal{C}$ be a small, pointed, symmetric monoidal $\infty$-category and let
$\mathfrak{A} \subset \fun( \Delta^1 \times \Delta^1, \mathcal{C} )$ be a
full subcategory. 
Suppose that: 
\begin{enumerate}
\item  
$\mathfrak{A}$ is closed under tensoring with objects in
$\mathcal{C}$. 
\item
The tensor product of a zero object in $\mathcal{C}$ with any object
is a zero object. 
\end{enumerate}
\begin{definition} 
Let $\mathcal{D}$ be a presentable, stable $\infty$-category. A functor
$F\colon \mathcal{C} \to \mathcal{D}$ is called \emph{$\mathfrak{A}$-admissible} if the following hold:
\begin{enumerate}
\item  Let $0 \in \mathcal{C}$ be a zero object. Then $F(0) \in
\mathcal{D}$ is a zero object.
\item Any diagram $\Delta^1 \times \Delta^1 \to \mathcal{C}$ that belongs to
$\mathfrak{A}$ is carried by $F$ to a pushout square in $\mathcal{D}$. 
\end{enumerate}
Let $\fun^{\mathrm{adm}}(\mathcal{C}, \mathcal{D}) \subset
\fun(\mathcal{C}, \mathcal{D})$ denote the
$\infty$-category of $\mathfrak{A}$-admissible functors from $\mathcal{C}$ to
$\mathcal{D}$. 
\end{definition}

We will now recall the construction of the universal presentable, stable $\infty$-category receiving an $\mathfrak{A}$-admissible
functor from $\mathcal{C}$, which we will denote $\mot_{\mathcal{C},
\mathfrak{A}}$ below. 

\begin{cons}We consider the Yoneda embedding $\mathcal{C} \hookrightarrow
\mathcal{P}(\mathcal{C})$. We recall (cf.\ \cite[Cor.~4.8.1.12]{Lur16} and
\cite{GlasmanDay}) that the target 
inherits a symmetric monoidal structure via \emph{Day convolution.} 
Similarly, the $\infty$-category $\mathcal{P}_{\mathrm{Sp}}(\mathcal{C})$ of
presheaves of spectra inherits a 
bicocontinuous symmetric monoidal structure via
the tensor product
\[ \mathcal{P}_{\mathrm{Sp}}(\mathcal{C})  \simeq \mathcal{P}(\mathcal{C})
\otimes \mathrm{Sp}. \]
We have a canonical symmetric monoidal functor $\mathcal{C} \to
\mathcal{P}_{\mathrm{Sp}}(\mathcal{C})$.
Let the image of $x \in \mathcal{C}$ in 
$\mathcal{P}_{\mathrm{Sp}}(\mathcal{C})$ be denoted $h_x$; these are compact
generators of $\mathcal{P}_{\mathrm{Sp}}(\mathcal{C})$.

We consider the localizing subcategory $\mathcal{I} \subset
\mathcal{P}_{\mathrm{Sp}}(\mathcal{C})$ generated 
by the following two sets of objects: 
\begin{enumerate}
\item  $h_0$, where $0 \in \mathcal{C}$ is a zero object. 
\item For each square
\[ \xymatrix{
a \ar[d] \ar[r] & b \ar[d] \\
c \ar[r] &  d
}\]
which belongs to $\mathfrak{A}$,
the cofiber of $h_b \sqcup_{h_a} h_c \to h_d$.
\end{enumerate}

Our hypotheses on $\mathfrak{A}$ and $\mathcal{C}$ imply that $\mathcal{I}$ is
actually a localizing $\otimes$-ideal and that it is compactly generated. 
Define the symmetric monoidal $\infty$-category $\mot_{\mathcal{C},
\mathfrak{A}}$ of \emph{$(\mathcal{C}, \mathfrak{A})$-motives} as the
localization of $\mathcal{P}_{\mathrm{Sp}}(\mathcal{C})$ at the class of
morphisms $X \to 0$ where $X \in \mathcal{I}$. 
\end{cons}

We note the following 
basic properties of $\mot_{\mathcal{C}, \mathfrak{A}}$: 
\begin{prop} 
\label{axiomaticmotive}
\begin{enumerate}
\item \label{it:monoidal} 
$\mot_{\mathcal{C}, \mathfrak{A}}$ is a presentable, symmetric monoidal stable
$\infty$-category with a bicocontinuous tensor product $\otimes$.
\item   \label{it:admiss}
There is a symmetric monoidal, $\mathfrak{A}$-admissible functor $\mathcal{C} \to \mot_{\mathcal{C},
\mathfrak{A}}$ and its image consists of compact objects of
$\mot_{\mathcal{C}, \mathfrak{A}}$. 
In particular, the unit is compact in $\mot_{\mathcal{C}, \mathfrak{A}}$. 
\item \label{it:adjoint}
Let $\mathcal{D}$ be any presentable stable $\infty$-category. Then there is a natural
equivalence
\[ \mathrm{Fun}^L ( \mot_{\mathcal{C}, \mathfrak{A}},\mathcal{D}) \simeq
\mathrm{Fun}^{\mathrm{adm}}(\mathcal{C}, \mathcal{D})  , \]
where $\mathrm{Fun}^L$ denotes cocontinuous functors. 
If, moreover, ${\mathcal D}$ is presentably symmetric monoidal, then
symmetric monoidal functors
correspond under this equivalence.

\end{enumerate}

\end{prop} 
\begin{proof} 
The first claim follows from the construction of $\mot_{\mathcal{C}, \mathfrak{A}}$ as the accessible localization of $\mathcal{P}_{\mathrm{Sp}}(\mathcal{C})$ at a localizing $\otimes$-ideal.
For the second claim, we note that the functor $\mathcal{C} \to \mathcal{P}_{\mathrm{Sp}}(\mathcal{C})$ is
symmetric monoidal and so is the localization functor 
$\mathcal{P}_{\mathrm{Sp}}(\mathcal{C}) \to \mot_{\mathcal{C}, \mathfrak{A}}$.
Therefore, their composite is a symmetric monoidal functor. Now the Yoneda functor $\mathcal{C} \to
\mathcal{P}_{\mathrm{Sp}}(\mathcal{C})$ takes values in compact objects and, since we are localizing at a class of morphisms generated by maps with compact source and target,
the localization $\mathcal{P}_{\mathrm{Sp}}(\mathcal{C}) \to
\mathrm{Mot}_{\mathcal{C}, \mathfrak{A}}$ preserves compact  objects.
For the final claim, we have  natural equivalences for any presentable stable
$\infty$-category $\mathcal{D}$,
\[ \fun^L( \mathcal{P}_{\mathrm{Sp}}(\mathcal{C}), \mathcal{D}) 
\simeq \fun^L ( \mathcal{P}(\mathcal{C}), \mathcal{D}) 
\simeq \fun(\mathcal{C}, \mathcal{D}).
\]
Here one sees that the subcategory $\fun^{\mathrm{adm}}(\mathcal{C},
\mathcal{D})$ of $\mathfrak{A}$-admissible functors corresponds to the
subcategory of
$\fun^L( \mathcal{P}_{\mathrm{Sp}}(\mathcal{C}), \mathcal{D}) $
which carries the localizing subcategory $\mathcal{I}$ into $0$. 
This proves the result. 
\end{proof}

\subsection{The $\infty$-categories $\mot_{\kappa}(\mathcal{R})$}

We consider the symmetric monoidal $\infty$-category $\catst$ of small stable,
idempotent-complete $\infty$-categories with the Lurie tensor product. 
This is a compactly generated, presentable $\infty$-category (cf.\ \cite[Cor.~4.25]{BGT13}) with a bicocontinuous tensor product.
The compact objects are closed under the tensor
product \cite[Prop.~5.2]{BGT14}. 
\begin{definition}
Given a small symmetric monoidal stable $\infty$-category $\mathcal{R} \in
\mathrm{CAlg}( \catst)$, consider its $\infty$-category $\md_{\mathcal{R}}(\catst)$ of modules in $\catst$, which we will call the $\infty$-category of (small) \emph{$\mathcal{R}$-linear
$\infty$-categories.}

By construction, $\md_{\mathcal{R}}(\catst)$ is a
\emph{compactly generated}  presentable symmetric monoidal $\infty$-category with the relative tensor product $-\otimes_\mathcal{R}-$ commuting with colimits in each variable.  The generators are of the form $\mathcal{R} \otimes
\mathcal{C}$, where $\mathcal{C} \in \catst$ is compact. It follows that there are also internal mapping objects $\Map_\mathcal{R}(-,-)$ in $\md_{\mathcal{R}}(\catst)$.  
\end{definition}

\begin{example} 

Since the unit $\mathcal{R}\in \md_{\mathcal{R}}(\catst)$ corepresents the functor which extracts the underlying $\infty$-groupoid of an object $\mathcal{M}$ of $\md_\mathcal{R}(\catst)$, it follows that the underlying $\infty$-groupoid of $\Map_\mathcal{R}(\mathcal{M},\mathcal{N})$ identifies with the space of maps $\mathcal{M}\rightarrow \mathcal{N}$ in  $\md_{\mathcal{R}}(\catst)$.

\end{example}

\begin{cons}
The symmetric monoidal $\infty$-category $\md_{\mathcal{R}}(\catst)$ is
closed, so we can regard it as enriched over itself. As a
result, we can extract a 2-category from it, as follows. Namely, for
$\mathcal{M},\mathcal{N}\in\md_{\mathcal{R}}(\catst)$, we define the category
of morphisms from $\mathcal{M}$ to $\mathcal{N}$ to be the \emph{homotopy
category}  of
the underlying $\infty$-category of the internal mapping object
$\Map_\mathcal{R}(\mathcal{M},\mathcal{N})\in\md_{\mathcal{R}}(\catst)$,
cf. \cite{lurie_cobordism}[Def. 2.3.13].  \end{cons}

Thus, for a morphism $f\colon \mathcal{M}\rightarrow\mathcal{N}$ in
$\md_{\mathcal{R}}(\catst),$ it makes sense to ask, for example, whether $f$
is right adjointable.

\begin{remark} 
The condition that a morphism $f$ in $\md_{\mathcal{R}}(\catst)$ be right
adjointable is generally stronger than the condition that the underlying
functor of $f$ admit a right adjoint (we need the right adjoint to be
$\mathcal{R}$-linear), though the conditions are equivalent if every object of
$\mathcal{R}$ is dualizable.  To avoid ambiguity, we will use the term \emph{$\mathcal{R}$-linear right adjoint} for the notion coming from the 2-category $\md_{\mathcal{R}}(\catst)$.  For more discussion of this issue, see \cite{Gaits} and \cite[Rem.~ D.1.5.3]{lurie_sag}.
\end{remark} 

Now we recall a class $\mathfrak{A}$ of objects of
$\mathrm{Fun}(\Delta^1\times\Delta^1,\md_{\mathcal{R}}(\catst))$ which
corresponds to the class of \emph{split-exact sequences} of \cite{BGT13, BGT14}, or in other terms to the notion of \emph{semi-orthogonal decomposition} in the theory of triangulated categories (cf.\ \cite[\S 7.2]{lurie_sag}).
Compare \cite[Def.~5.3]{HSS15} for a discussion in the $\mathcal{R}$-linear
setting. 

\begin{definition}
Let $\mathcal{R} \in \clg( \catst)$ and let
$$\mathcal{X} = (\mathcal{M}\overset{i}{\longrightarrow} \mathcal{N}\overset{p}{\longrightarrow}\mathcal{P})$$
a sequence in $\md_{\mathcal{R}}(\catst)$ with null composite.  Note that the space of null-homotopies of any map in $\catst$ is either empty or contractible, so this can be interpreted either as structure or a condition.

Say that $\mathcal{X}$ is a \emph{split-exact sequence} if the following conditions hold:
\begin{enumerate}
\item The functors $i$ and $p$ both have $\mathcal{R}$-linear right adjoints, say $i_r$ and $p_r$, respectively.
\item The unit map $1\rightarrow i_r\circ i$ is an equivalence, i.e., $i$ is fully faithful.
\item The counit map $p\circ p_r\rightarrow 1$ is an equivalence, i.e., $p_r$ is fully faithful.
\item 
The sequence of natural transformations $i \circ i_r \to \mathrm{id}_{\mathcal{N}} \to p_r \circ
p$ is a cofiber sequence of functors $\mathcal{N} \to \mathcal{N}$. Note that
one has a canonical nullhomotopy of the composite because $\mathcal{M} \to
\mathcal{P}$ is the zero functor. 
\end{enumerate}
\end{definition}

In such a situation, it is easy to see that $\mathcal{X}$ is both a fiber
sequence and a cofiber sequence in $\md_{\mathcal{R}}(\catst)$ (and a Verdier
quotient in $\catst$), and the same is true of the sequence obtained by passing to right adjoints.  
We will need to know that 
the $\mathcal{R}$-linear tensor product respects split-exact sequences.
Compare also \cite[Lem. 5.5]{BGT14}.

\begin{lemma} 
Let $\mathcal{M}\to \mathcal{N} \to \mathcal{P}$ be a split-exact sequence
of $\mathcal{R}$-linear $\infty$-categories and let $\mathcal{C}$ be
any $\mathcal{R}$-linear $\infty$-category. Then the sequence
$\mathcal{C}\otimes_{\mathcal{R}}\mathcal{M}\to
\mathcal{C}\otimes_{\mathcal{R}}\mathcal{N}
\to\mathcal{C}\otimes_{\mathcal{R}}\mathcal{P}$ is a split-exact sequence too.
\end{lemma} 
\begin{proof}
The tensor product $-\otimes_\mathcal{R}\mathcal{C}\colon \md_{\mathcal{R}}(\catst)
\to \md_{\mathcal{R}}(\catst)$ is canonically an enriched functor. Therefore,
it can be applied not only to $\mathcal{R}$-linear functors, but
also to natural transformations of functors: it has the
structure of a 2-functor.  In particular, it preserves adjunctions, and hence it preserves the first three conditions in the definition of split-exact sequence.  

To check the last condition, we observe that if $F' \to F \to F''$ is a cofiber
sequence of $\mathcal{R}$-linear functors $\mathcal{C} \to \mathcal{D}$ for
$\mathcal{C}, \mathcal{D} \in \md_{\mathcal{R}}(\catst)$, then 
\[ F' \otimes_{\mathcal{R}} \mathcal{E} \to F \otimes_{\mathcal{R}}
\mathcal{E} \to F'' \otimes_{\mathcal{R}}
\mathcal{E}  \]
is a cofiber 
sequence of $\mathcal{R}$-linear functors $\mathcal{C} \otimes_{\mathcal{R}} \mathcal{E} \to
\mathcal{D} \otimes_{\mathcal{R}} \mathcal{E}$ for any $\mathcal{E} \in
\md_{\mathcal{R}}(\catst)$. 
This follows from the fact that 
the map of spaces 
$$\hom_{\md_{\mathcal{R}}(\catst)}(\mathcal{C}, \mathcal{D})
\to\hom_{\md_{\mathcal{R}}(\catst)}(\mathcal{C} \otimes_{\mathcal{R}}
\mathcal{E}, \mathcal{D} \otimes_{\mathcal{R}}\mathcal{E})$$
arises by taking spaces of objects from a morphism
\[ \mathrm{Map}_{\mathcal{R}}(\mathcal{C}, \mathcal{D}) \to
\mathrm{Map}_{\mathcal{R}}(\mathcal{C} \otimes_{\mathcal{R}} \mathcal{E}, \mathcal{D}
\otimes_{\mathcal{R}} \mathcal{E})  \in \md_{\mathcal{R}}(\catst), \]
i.e., an $\mathcal{R}$-linear exact functor, 
which in particular preserves cofiber sequences on underlying
$\infty$-categories. 
\end{proof}

For $\mathcal{R} \in \clg(\catst)$, 
we now would like to construct an $\infty$-category of
$\mathcal{R}$-linear non-commutative motives from $\md_{\mathcal{R}}(
\catst)$, but we need to make a minor technical detour since $\catst$ is not
essentially small. 
Choose a regular cardinal $\kappa$. Then the  
$\kappa$-compact objects in $\md_{\mathcal{R}}( \catst)$ are closed 
under pushouts, retracts, and $\mathcal{R}$-linear tensor products (cf.\  
\cite[Prop.~5.2]{BGT14}).
For any specific statement (e.g., descent-theoretic assertion), we will end up assuming $\kappa$ is taken large enough such
that our required statement takes place entirely in the world of $\kappa$-compact
objects.
\begin{definition} 
\label{weaklyadditive}
Fix $\mathcal{R} \in \clg(\catst)$ and fix a regular cardinal $\kappa$.
Let $\md_{\mathcal{R}}^{\kappa}( \catst) \subset \md_{\mathcal{R}}( \catst)$
denote the full subcategory spanned by the $\kappa$-compact objects, 
or an equivalent small model.
\begin{enumerate}
\item   
The $\infty$-category $\mot_{\kappa}(\mathcal{R})$ of \emph{non-commutative $\mathcal{R}$-motives} is defined
to be $\mathrm{Mot}_{\md^{\kappa}_{\mathcal{R}}(\catst), \mathfrak{A}}$ where
$\mathfrak{A}$ is the collection of split-exact sequences in
$\md^{\kappa}_{\mathcal{R}}(
\catst)$. 
When $\mathcal{R} = \sp^{\omega}$ and $\kappa = \aleph_0$, then this is
equivalent to the definition of \cite{BGT13} in view of \cite[Prop.
5.6]{HSS15}. Of course, the definition of $\mot_\kappa(\mathcal{R})$ depends on the choice of
$\kappa$. 
\item 
When $R$ is an $\mathbb{E}_\infty$-ring, we will write $\mot_\kappa(R) = \mot_\kappa( \perf(R))$.
\item
Given a $\kappa$-compact $\mathcal{R}$-linear $\infty$-category $\mathcal{M}$, we write
$\uadd(\mathcal{M}) \in \mot_{\kappa}(\mathcal{R})$ for its image and let $K_0'(\mathcal{M}) = \pi_0 \hom_{\mot_{\kappa}(\mathcal{R})}(
\uadd(\mathcal{R}), \uadd(\mathcal{M}))$.
\item A \emph{weakly additive invariant} of ($\kappa$-compact) $\mathcal{R}$-linear $\infty$-categories
is a functor $\md^{\kappa}_{\mathcal{R}}( \catst) \to \mathcal{D}$, where $\mathcal{D}$
is a presentable stable $\infty$-category, which is $\mathfrak{A}$-admissible
for $\mathfrak{A}$ as in part $(1)$. We find that any weakly additive invariant canonically
factors through $\mot_{\kappa}(\mathcal{R})$ by \Cref{axiomaticmotive}. We adopt the term {\em weakly} additive here to remind the reader that commutation with filtered colimits is not stipulated.
\end{enumerate}
\end{definition} 

To conclude this section, we relate $K_0'$ with algebraic $K$-theory.

\begin{cons}
Let $\mathcal{M}$ be an $\mathcal{R}$-linear $\infty$-category.  Let
$\Ex(\mathcal{M})$ denote the $\mathcal{R}$-linear $\infty$-category of cofiber sequences $(X\rightarrow Y\rightarrow Z)$ in $\mathcal{M}$.  
More formally, 
$\Ex(\mathcal{M}) \subset \fun(\Delta^1 \times \Delta^1, \mathcal{M})$ is the subcategory
of those cocartesian diagrams of the following form:
\[ \xymatrix{
X \ar[d] \ar[r] & Y \ar[d]  \\
0 \ar[r] &  Z.
}\]

Note that, because $\mathcal{M}$ is stable,  $\Ex(\mathcal{M}) \simeq \fun(\Delta^1, \mathcal{M})$ via $(X \to Y
\to Z) \mapsto (X \to Y)$.
We have the following split-exact sequence
of $\mathcal{R}$-linear $\infty$-categories
$\mathcal{M} \to \Ex(\mathcal{M}) \to \mathcal{M}$, where: 
\begin{itemize}
\item $\mathcal{M} \to \Ex(\mathcal{M})$ is the functor $X \mapsto
(X = X\to 0)$. This has a right adjoint which sends $(X \to Y\to Z)$ to $X$.
\item $\Ex(\mathcal{M})\to \mathcal{M}$ is the functor $(X \to Y\rightarrow Z)
\mapsto Z$. This has a right adjoint which sends $Z$ to $( 0 \to Z =  Z)$.
\end{itemize}

\begin{remark} 
Note that $\Ex(\mathcal{M}) \simeq \Ex(\sp^\omega) \otimes \mathcal{M} \in
\catst$, so that if $\mathcal{M} \in \md_{\mathcal{R}}(\catst)$ is
$\kappa$-compact, then so is $\Ex(\mathcal{M})$. 
To see this, 
we use the fact that $\Ex(\mathcal{M}) \simeq \fun(\Delta^1, \mathcal{M})$ is
the $\infty$-category of compact objects in $\fun(\Delta^1,
\mathrm{Ind}(\mathcal{M})) \simeq \fun( \Delta^1, \sp) \otimes
\mathrm{Ind}(\mathcal{M})$, where the latter tensor product is taken in the
$\infty$-category $\prls$. Compare \cite[Prop. 2.2.6]{Lurierotation}. 
\end{remark} 

It follows that the functors $\Ex(\mathcal{M}) \rightrightarrows
\mathcal{M}$ given by $(X \to Y\to Z) \mapsto X$ and $(X \to Y\to Z) \mapsto Z$ induce
an equivalence
\[ \uadd(\Ex(\mathcal{M})) \simeq \uadd(\mathcal{M}) \times \uadd(\mathcal{M})
\in \mot_{\kappa}(\mathcal{R}).  \]

\end{cons}

\begin{cons}
\label{basichomomorphism}
Let $\mathcal{M}$ be a $\kappa$-compact $\mathcal{R}$-linear $\infty$-category.  
Construct a natural homomorphism \begin{equation} \label{twoK0} K_0(\mathcal{M}) \to
K_0'(\mathcal{M}), \end{equation} which preserves the unit when $\mathcal{M}=\mathcal{R}$, as follows. 

Given $X \in \mathcal{M}$, we obtain an $\mathcal{R}$-linear functor
$\mathcal{R} \to \mathcal{M}$ sending the unit object to $X$, which defines an
element in $K_0'(\mathcal{M})$. 
Given a cofiber sequence $X\to Y \to Z $, we need to show that the class 
of $Y$
in $K_0'(\mathcal{M})$ is equal to the sum of the classes of $X$ and $Z$. 
By applying the functor of projection to the `middle term,' it suffices to check that the class of the object $(X\rightarrow Y\rightarrow Z)$ in $\Ex(\mathcal{M})$ is the sum of the classes $(X = X\to 0)$ and $(0\to Z = Z)$.  But by the above equivalence in $\mot_{\kappa}(\mathcal{R})$, we can check this after projection to the outer terms, where it is obvious.
\end{cons}

Using the techniques of \cite{BGT13, HSS15}, one can in fact show that $K_0 \simeq
K_0'$. 
Since we will not need this, we omit the proof.

\section{Abstract descent results}\label{sec:descent_results}

In this brief but central section, we describe how the use of
$\mathbb{E}_\infty$-structures enables one to prove abstract descent and
$\epsilon$-nilpotence results in a symmetric monoidal, stable
$\infty$-category. Our basic tool is \Cref{nilpotenceforclg} below.
Throughout, we use the following notation.

\begin{definition} 
Given 
$\mathcal{C}\in\mathrm{CAlg}({\Catst})$ and $M \in \mathcal{C}$, let 
$\langle M \rangle^{\otimes}\subset {\mathcal C}$ denote the thick $\otimes$-ideal generated by
$M$ and 
$\langle M \rangle^{\otimes}_{\epsilon}$ its $\epsilon$-enlargement. 
\end{definition}

\begin{thm} 
\label{nilpotenceforclg}
Suppose $\mathcal{C}\in\mathrm{CAlg}({\Catst})$ with unit $\mathbf{1}$ and $R \in
\clg(\mathcal{C})$. Moreover, suppose there exists $M \in \mathcal{C}$ and a map
$M \to R$ in $\mathcal{C}$ such that the
image of $(\pi_0 M)  \otimes \mathbb{Q} \to (\pi_0 R) \otimes \mathbb{Q}$
contains the unit. Then $R \in \langle M\rangle^{\otimes}_{\epsilon}$.
\end{thm} 

\begin{proof} 
Without loss of generality, we may assume $\mathcal{C}$ is small by writing
$\mathcal{C}$ as a union of small subcategories. 
Let $N$ be a positive integer such that $(\pi_0 M)[N^{-1}] \to
(\pi_0 R)[N^{-1}]$ has image containing the unit. 
Let $\Sigma$ be the set of primes dividing $N$. We will show that in fact 
$R$ belongs to $\langle M\rangle^{\otimes}_{\epsilon, \Sigma}$. 

To see this, let $T$ be any finite spectrum 
whose $p$-localizations for $p \in \Sigma$ are all nontrivial. We need to show that
$R$ belongs to the thick $\otimes$-ideal generated by $M$ and $T\otimes
\mathbf{1}$. 
Let $\mathcal{J}\subset\mathcal{C}$ denote this thick $\otimes$-ideal. 
We can then form the Verdier quotient $\mathcal{C}/\mathcal{J}$, and
equivalently we need to show that the image $\overline{R}$ of $R$ in
$\mathcal{C}/\mathcal{J}$ is
zero. 
This is equivalent to  showing that the 
$\mathbb{E}_\infty$-ring $B = \mathrm{Hom}_{\mathcal{C}/\mathcal{J}}(
\mathbf{1}, \overline{R})$
is contractible: in fact, that will imply that the unit map $\mathbf{1} \to
\overline{R}$ is nullhomotopic, so that $\overline{R} = 0$. 
Note also that $$T \otimes B = T \otimes 
\mathrm{Hom}_{\mathcal{C}/\mathcal{J}}(
\mathbf{1}, \overline{R})
= 
\mathrm{Hom}_{\mathcal{C}/\mathcal{J}}(
\mathbf{1}, \overline{R} \otimes T) = 
0,
$$ since smashing with $T$ annihilates the
$\infty$-category $\mathcal{C}/\mathcal{J}$. 

We now prove that $B$ is contractible. 
By hypothesis, there is a map $f\colon M \to R$ in
$\mathcal{C}$ whose image in
$\pi_0 \otimes \mathbb{Z}[N^{-1}]$ contains the unit. 
In other words, there is a map $g\colon  \mathbf{1} \to M$ such that the composite
$f\circ g = N^k\in\pi_0 R$  for some $k\ge 0$.
After taking the Verdier quotient, $M$ and hence $f$ map to zero, so we find that $N^k=0$ in $\pi_0 B$.
This implies that we have $H\mathbb{Q}\otimes B=0$ and that for all primes
$p$ not dividing $N$, we also have $H\mathbb{F}_p\otimes B=0$.
For each $p \mid N$, however, we claim that we have
$H\mathbb{F}_p \otimes B = 0$ as well. This follows because
we have $B \otimes T = 0$, while $T$ has nontrivial mod $p$ homology.
Using the May nilpotence conjecture (cf. \cite[Thm.~A]{MNN_nilpotence}
for a proof) applied to the unit
in $B$, we conclude that $B
= 0$. 
\end{proof}

We obtain the following consequence for $(A, \epsilon)$-nilpotence. 
\begin{thm} 
\label{maincriterionepsilonnilp}
Suppose $\mathcal{C}\in\mathrm{CAlg}(\Catst)$, $R \in \clg(\mathcal{C})$ and $A\in\mathrm{Alg}(\mathcal{C})$. Suppose that there exists an $A$-module $M$
and a map $M \to R $ in $\mathcal{C}$ such that the image of $\pi_0 M \otimes
\mathbb{Q} \to
\pi_0 R \otimes \mathbb{Q}$ contains the unit. 
Then $R \in \anile$.
In particular, $\cbaug(A) \otimes R$ is an $\epsilon$-nilpotent limit diagram. 
\end{thm} 
\begin{proof} 
\Cref{nilpotenceforclg} implies that $R\in\langle M\rangle^\otimes_\epsilon$. Since $M$ is
an $A$-module, we have $M\in\langle A\rangle^\otimes$, and conclude that $R\in\langle A\rangle_\epsilon^\otimes=\mathrm{Nil}^{A,\epsilon}$. The final claim follows from \Cref{nilpcobareplimit}. 
\end{proof}

We can apply this to obtain an $\epsilon$-nilpotence descent result in
$\mathrm{Mot}_{\kappa}(\mathcal{R})$ in the sense of the previous section. 

\begin{thm} 
\label{mainmotdesc}
Let $\mathcal{R}\in\mathrm{CAlg}(\catst)$ be a small symmetric monoidal,
idempotent-complete stable $\infty$-category and let
$\mathcal{A}\in\mathrm{Alg}(\mathrm{Mod}_{\mathcal{R}}(\catst))$ be an $\mathcal{R}$-linear
monoidal $\infty$-category. Suppose that there exists  an $\mathcal{R}$-linear functor
$\mathcal{A} \to \mathcal{R}$ whose image on $K_0(-) \otimes \mathbb{Q}$
contains the unit. 
Choose a regular cardinal $\kappa$ such that $\mathcal{A}$ is
$\kappa$-compact in $\md_{\mathcal{R}}( \catst)$.
Then the augmented cosimplicial object in $\mot_{\kappa}(\mathcal{R})$,
\[ \uadd(\mathcal{R}) \to \left( \uadd(\mathcal{A}) \rightrightarrows
\uadd(\mathcal{A}
\otimes_{\mathcal{R}}
\mathcal{A}) \triplearrows \dots \right)  \]
is an $\epsilon$-nilpotent limit diagram. 
\end{thm} 
\begin{proof} 
We apply \Cref{maincriterionepsilonnilp} with
$\mathcal{C}:=\mot_{\kappa}(\mathcal{R})$,
$R:=\uadd(\mathcal{R})\in\mathrm{CAlg}(\mathcal{C})$ the unit of $\mathcal{C}$, and $M:=A:=
\uadd(\mathcal{A})\in\mathrm{Alg}(\mathcal{C})$. Then the map $(\pi_0 M\to \pi_0 R)=
(K_0'(\mathcal{A})\to K_0'(\mathcal{R}))$ is seen to be rationally surjective (equivalently,
has image containing the unit) using the map in \Cref{basichomomorphism} and our assumption on $K_0$.
\end{proof}

 To conclude this section, we now establish a result
 that gives a much stronger conclusion about the comparison
 map as above in the special case in which one has a further assumption on the
 image in $\pi_0$ of the map $M\to R$.

\begin{thm} 
\label{Fpmodulespectrumresult}
Suppose $\mathcal{C}$ is a presentable, symmetric monoidal stable
$\infty$-category where the tensor is bicontinuous, and the unit $\mathbf{1}$ is compact. 
Suppose $A \in \mathrm{Alg}(\mathcal{C})$ is dualizable in $\mathcal{C}$. 
Suppose that there exists an $A$-module $M$ and a map $M \to \mathbf{1}$ in
$\mathcal{C}$ such
that the image of  $\pi_0M\to \pi_0\mathbf{1}$ contains the prime $p$.
Then, for any $X \in \mathcal{C}$, the fiber of 
\[ X \to \mathrm{Tot}( \cb(A) \otimes X)   \]
has the structure of an $H\mathbb{F}_p$-module. 
\end{thm} 
\begin{proof} 
We will freely use the language of acyclizations (or cellularizations) and localizations with
respect to a dualizable object, for which
we refer to \cite[\S~2--3]{MNN15i} for an exposition. 
We use $\Map$ for internal mapping objects in $\mathcal{C}$, which
exist since $\mathcal{C}$ is presentable and the tensor product is
bicocontinuous. 
Let $V_A, U_A \in \mathcal{C}$  be the $A$-acyclization and
$A^{-1}$-localization of $\mathbf{1}$, so that we have a cofiber sequence
\[ V_A \to \mathbf{1} \to U_A,  \]
such that $A \otimes U_A = 0$ and such that $V_A$ belongs to the localizing
$\otimes$-ideal generated by $A$. We then have
\[ \mathrm{Tot}(\cb(A) \otimes X ) \simeq \Map_{\mathcal{C}}(V_A, X),  \]
because both sides give the $A$-completion of $X$. In particular, the fiber of $X \to
\mathrm{Tot}( \cb(A) \otimes X)$ can be identified with
the internal mapping object
$\Map_{\mathcal{C}}( U_A, M)$. 

We observe now that $U_A$ is an $\mathbb{E}_\infty$-algebra in $\mathcal{C}$, as the
$A^{-1}$-localization of $\mathbf{1}$. 
Moreover, by assumption $\pi_0 U_A$ is an $\mathbb{F}_p$-algebra. It follows
that the underlying $\mathbb{E}_2$-algebra of $\hom_{\mathcal{C}}( \mathbf{1}, U_A)$ is
an $\mathbb{E}_2$-algebra over $\mathbb{F}_p$ by the Hopkins-Mahowald theorem that the
free $\mathbb{E}_2$-algebra with $p = 0$ is $H \mathbb{F}_p$. See
\cite[Thm.~4.18]{MNN_nilpotence} or \cite[Thm.~5.1]{AB14} for recent expositions of this
result. 
By adjunction, it follows that $U_A \in \mathcal{C}$ is a module over
$H\mathbb{F}_p$, and therefore the fiber 
$\Map_{\mathcal{C}}(  U_A, X)$ is one, too. 
\end{proof}

\section{Examples and applications}\label{sec:examples}

In this section, we will give the primary examples and applications of our descent
results (in particular, \Cref{mainmotdesc}), to algebraic finite flat
extensions of $\mathbb{E}_\infty$-rings and to many 
Rognes-style Galois extensions of $\mathbb{E}_\infty$-rings.

\begin{thm} 
\label{finitedesc}
Let $A $ be an $\mathbb{E}_\infty$-ring. Let $B$ be an $\mathbb{E}_2$-algebra
in the $\infty$-category of $A$-modules. Suppose that $B$ is a perfect
$A$-module and the canonical map $K_0(B) \otimes \mathbb{Q} \to K_0 (A) \otimes
\mathbb{Q}$ has image containing the unit. 
Assume $\kappa$ is such that $\perf(B) $ is $\kappa$-compact
in $\md_{\perf(A)}( \catst)$. 	

Let $F$ be any weakly additive invariant of $\kappa$-compact $A$-linear $\infty$-categories (e.g., 
$K$-theory, non-connective/Bass $K$-theory, homotopy $K$-theory of
$\mathbb{Z}$-linear $\infty$-categories \cite{Tab15}, $\THH$, and $\TC$) taking values in a presentable stable $\infty$-category
$\mathcal{D}$. Then the augmented cosimplicial diagram
\[ F(\perf(A)) \to \left( F(\perf(B)) \rightrightarrows F(\perf(B \otimes_A B)) \triplearrows
\dots \right)  \]
is an $\epsilon$-nilpotent limit diagram. 
In particular, the canonical map 
in $\mathcal{D}$
\[ F(\perf(A)) \to  \mathrm{Tot}\left( F(\perf(B)) \rightrightarrows F(\perf(B
\otimes_A B)) \triplearrows
\dots \right) \]
is an $\epsilon$-equivalence, and the associated $\mathrm{Tot}$ tower becomes
quickly convergent (\Cref{quickconv}) after any periodic localization.
\end{thm} 
\begin{proof} 
We apply \Cref{mainmotdesc} with $\mathcal{R}:=\mathrm{Perf}(A)$, $\mathcal{A}:=
\mathrm{Perf}(B)$ (which is monoidal as $B$ is an $\mathbb{E}_2$-algebra) and the forgetful functor $\mathcal{A}\to\mathcal{R}$, which is 
$\mathcal{R}$-linear. 
Note that the functor $B' \mapsto \perf(B')$ is a symmetric monoidal
functor from the $\infty$-category of $A$-algebras to the $\infty$-category
$\md_{\perf(A)}(\catst)$ (cf.\ \cite[Rem.~4.8.5.17]{Lur16} for a treatment in the
presentable setting).
We conclude that
\[ \uadd( \mathrm{Perf}(A)) \to \left( \uadd(\mathrm{Perf}(B)) \rightrightarrows
\uadd(\mathrm{Perf}(B \otimes_A B)) \triplearrows
\dots \right)  \]
is an $\epsilon$-nilpotent limit diagram in $\mathrm{Mot}_{\kappa}(A)$. The given $F$ factors through
an exact functor $\mathrm{Mot}_\kappa(A)\to\mathcal{D}$, and the result
now follows by
\Cref{epsilonnilpfunctoriality}, part (4) and \Cref{prop:epsilon-limit}, part (2).
\end{proof} 

\begin{remark}\label{rem:connective}
We note that if the hypotheses of \Cref{finitedesc} are satisfied for $A \to
B$, then they are satisfied for $A' \to B \otimes_A A'$ for any
$\mathbb{E}_\infty$-ring
$A'$ under $A$. 
\end{remark}

\subsection{Algebraic examples}
We now show that the desired hypothesis of rational surjectivity in
\Cref{finitedesc} is satisfied for a finite flat
extension of $\mathbb{E}_\infty$-ring spectra. In particular, algebraic $K$-theory, after
any periodic localization, is a sheaf on the finite flat site. Restricted to classical schemes, it is then a sheaf even with respect to the $fppf$-topology \cite[Tag 05WM]{stacks-project}.

We need a lemma first. In the affine case, a more precise result is available, due to Swan \cite[Prop.~10.2]{swan}.

\begin{lemma}\label{algebraicunit}

Let $X$ be a quasi-compact and quasi-separated scheme and $P$ a perfect complex of ${\mathcal O}_X$-modules. If the Euler characteristic 
of $P$ has constant value $n\in\mathbb{Z}$, then the class $[P]-n\in K_0(X)$ is nilpotent. It follows that, if the Euler characteristic of $P$
is nowhere vanishing, then the class $[P]\in K_0(X)\otimes\mathbb{Q}$ is a unit.
\end{lemma}
\begin{proof}
The class $x:=[P]-n$ vanishes at every residue field of $X$, so  there is a
(finite) open cover $X=\bigcup_{i=1}^r U_i$ such that $x$ vanishes in every
$K_0(U_i)$. By Zariski descent in $K$-theory and induction on $r$, we find that
$x^{r}=0$. For the final claim, one reduces to $X$ being connected, and then the nilpotence of $[P]-n$ (for some $0\neq n\in\mathbb{Z}$) implies that
$[P]=n\left(1+\frac{[P]-n}{n}\right)$ is a rational unit.
\end{proof}

\begin{prop} 
\label{finiteetalesatisfied}
Let $A \to B$ be a morphism of $\mathbb{E}_\infty$-ring spectra. Suppose that:
\begin{enumerate}
\item  
The $\pi_0 A$-module $\pi_0 B$ is faithfully flat, finite, and projective. 
\item
The canonical map $\pi_*(A) \otimes_{\pi_0(A)} \pi_0(B) \to \pi_*(B)$ is an
isomorphism. 
\end{enumerate}
Then the hypotheses of 
\Cref{finitedesc} are satisfied. 
\end{prop} 
\begin{proof} 
It is easy to see that $B$ is a perfect $A$-module.
Applying \Cref{rem:connective}, it suffices to prove this proposition after
replacing the map $A \to B$ with the map $\tau_{\geq 0} A \to \tau_{\geq 0} B$, since $B \simeq A\otimes_{\tau_{\geq 0}A} \tau_{\geq 0} B$.
Therefore, we may assume that $A$ and $B$ are connective themselves. 
In this case, $K_0(A) = K_0( \pi_0  A )$ (see, for example, \cite[Lec.~20, Cor.~3]{lurie-alg-k-theory}).
The result now follows from \Cref{algebraicunit}. \end{proof}

\begin{example}\label{projectivedescent}
Let $f \colon X \to Y$ be a morphism between two noetherian schemes with affine diagonal morphisms. 
Consider the symmetric monoidal functor $f^*\colon \perf(Y) \to \perf(X)$.

Suppose
$f$ is projective and of finite Tor-dimension, so that $f_*$
defines a functor $f_*\colon \perf(X) \to \perf(Y)$.  Suppose further that $f$ is surjective.

In this case, we claim that $f_*$ defines a rational surjection $K_0(X) \otimes
\mathbb{Q}
\twoheadrightarrow K_0(Y) \otimes \mathbb{Q}$. Without loss of generality, we
may assume $Y$ connected. Choose a
relatively ample line bundle $\mathcal{O}(1)$ on $X$.  Then $f_*(\mathcal{O}(n))$
is a perfect complex of nonzero Euler characteristic on $Y$ for $n \gg 0$. (Base-change to the local ring of a codimension-zero point to reduce to the case $Y=\operatorname{Spec}(B)$ with $B$ artinian; then since $f_*(\mathcal{O}(n))$ is a perfect complex of $B$-modules homologically concentrated in degree $0$, the Auslander-Buchsbaum formula implies that $f_*(\mathcal{O}(n))$ is finite free.  On the other hand, it is nonzero because its pullback surjects onto $\mathcal{O}(n)$.)  Therefore the class of $f_*(\mathcal{O}(n))$ is invertible in $K_0(Y) \otimes \mathbb{Q}$ by
\Cref{algebraicunit}.

Thus, by \Cref{mainmotdesc}, for any weakly additive invariant $F$ of $\mathrm{Perf}(Y)$-linear $\infty$-categories,
we have
that
\[ F( \perf(Y)) \to \left( F( \perf(X)) \rightrightarrows F( \perf(X \times_Y X))
\triplearrows \right)   \]
is an $\epsilon$-nilpotent limit diagram.
We use \cite[Thm.~1.2]{BZFN10} (whose hypotheses are satisfied by \cite[Prop.~3.19, Cor.~3.23]{BZFN10}) to identify $\perf(X \times_Y X) \simeq\perf(X) \otimes_{\perf(Y)} \perf(X)$ and so on for the higher terms.  

Note that if $f$ is not flat, then the fiber products such as $X\times_Y X$
need to be interpreted in the derived sense.  However, if $\ell$ is a prime
invertible on $Y$, then the $\ell$-completed $K$-theory will be the same
whether we take the derived or ordinary fiber products.  In fact, one uses Zariski descent to reduce to the affine case; there, by the group-completion theorem, we need to see that if $A$ is a connective $\mathbb{E}_\infty$-ring with $1/\ell\in A$ then $BGL_d(A)\rightarrow BGL_d(\pi_0A)\simeq\tau_{\leq 1}BGL_d(A)$ is a mod $\ell$ homology equivalence.  This follows from the fact that for $n>1$, $\pi_n BGL_d(A) \cong (\pi_{n-1}A)^{d^2}$ is uniquely $\ell$-divisible.
\end{example}

\subsection{Rognes's Galois extensions}

We now obtain our general descent result for Galois extensions of structured ring spectra in the sense of Rognes \cite{Rognes08}.
\begin{thm} 
\label{Galoisdesc} 
Let $A \to B$ be a $G$-Galois extension of $\mathbb{E}_\infty$-ring spectra where $G$ is finite.
Suppose that the image of the transfer map $K_0(B) \otimes \mathbb{Q} \to K_0(A) \otimes \mathbb{Q}$
contains the unit. Then, for all $n\ge 0$ and implicit primes, we have:
\begin{enumerate}
\item  
The canonical maps of spectra
\[ L^f_n K(A) \to L^f_n (K(B)^{hG}) \to ( L^f_n K(B))^{hG}
\]
are equivalences. 
\item
If $\mathcal{M}$ is any $A$-linear $\infty$-category and $F\colon
\md_{\perf(A)}^{\kappa}(\catst) \to \sp$ is any weakly additive invariant for
$\kappa$ large enough, 
the maps
\[  L^f_n F( \mathcal{M}) \to L^f_n (F( \mathcal{M}\otimes_A B )^{hG}) \to (L^f_n
F(\mathcal{M} \otimes_A B))^{hG}  \]
are equivalences.
\item Notation as above, the map
\[ F(\mathcal{M}) \to  ( F(\mathcal{M} \otimes_A B))^{hG} \]
is an $\epsilon$-equivalence (e.g., $K(A) \to K(B)^{hG}$ is an
$\epsilon$-equivalence). 
\item The $\mathrm{Tot}$ tower that computes $L_n^f F(\mathcal{M} \otimes_A
B)^{hG}$ is quickly converging. 
\end{enumerate}
\end{thm} 
\begin{proof} 
This follows from \Cref{finitedesc}. We note that $B$ is a dualizable, and hence perfect, $A$-module by \cite[Prop.\ 6.2.1]{Rognes08}, and that the equivalence $B \otimes_A B
\simeq \prod_{g \in G} B$ identifies the cobar construction $B
\rightrightarrows B \otimes_A B \triplearrows \dots $ with the 
diagram computing the $G$-homotopy fixed points of $B$. Note also that each of the
functors \[B\mapsto \mathrm{Perf}(B)\mapsto \uadd(\perf(B))\mapsto
F(\mathrm{Perf}(B))\] preserves finite products: in fact, any weakly additive invariant
preserves finite products.
\end{proof}

We next observe that the condition of rational surjectivity of the transfer
map is particularly transparent in the case of a Galois extension.

\newtheorem{condition}{Condition}
\renewcommand{\thecondition}{\Alph{condition}}

\begin{prop}
\label{prop:condition}
Let $G$ be a finite group and $A\to B$ be a $G$-Galois extension of $\mathbb{E}_\infty$-ring spectra.
Then the following are equivalent:

\begin{enumerate}
\item  $K_0(B) \otimes \mathbb{Q} \to K_0(A ) \otimes
\mathbb{Q}$ has image containing the unit.
\item The class of $[B] \in K_0(A) \otimes \mathbb{Q}$ is a unit.
\item The class $[B] \in K_0(A) \otimes
\mathbb{Q}$ is equal to $|G|$. 
\end{enumerate}
\end{prop}

\begin{proof}
Clearly, the third condition implies the second, which implies the first. 
It suffices to argue that the first condition implies the third to prove
that they are all equivalent.  

Let $i^*\colon  K_0(A) \otimes \mathbb{Q}  \to K_0(B) \otimes \mathbb{Q}$ be the canonical ring map and $i_*\colon 
K_0(B) \otimes \mathbb{Q} \to K_0(A) \otimes \mathbb{Q}$ be the restriction map. Recall that if we regard $K_0(B)\otimes \mathbb{Q}$ as a $K_0(A)\otimes \mathbb{Q}$-module via $i^*$, then $i_*$ is a module map.
Suppose $x \in K_0(B) \otimes \mathbb{Q}$ has $i_*(x) = 1$. Then
$i^*$ is injective, since it has a left inverse given by $y \mapsto i_*(x y)$. 
To show that $[B] = |G|$, it thus suffices to apply $i^*$ and check it in
$K_0(B) \otimes \mathbb{Q}$. 
We have \[i^*[B]=[B\otimes_A B]=\left[\prod_{g \in G} B\right]= |G|,\] and we are done. 
\end{proof}

For ease of reference, we give this property a name. 
\begin{condition}
\label{condition}
A $G$-Galois extension $A \to B$ satisfies the equivalent statements of 
\Cref{prop:condition}.
\end{condition}

We now give some examples. 

\begin{example} 
Let $n \geq 1$ and let $k$ be a field of characteristic 
zero containing the $n$th roots of unity.
We let 
$A$ be the unique $\e{\infty}$-algebra over $k$ with homotopy groups given by
$\pi_*( A) \simeq k[x_{2}^{\pm 1}]$ where $|x_{2}| = 2$. 
Here $A$ can be obtained by starting with the free $\mathbb{E}_\infty$-algebra
under $k$ on a degree two class and then inverting it.
There is a $C_n$-action on $A$ (obtained, e.g., using this presentation) such that a  generator acts on $x_2$ by
multiplication by a fixed primitive 
$n$th root of unity, and $A' = (A)^{hC_n}$ has homotopy groups given by $k[(x_2^{
n})^{\pm 1}]$. The map $A' \to A$ is a $C_n$-Galois extension and
Condition~\ref{condition} is satisfied, so we find that 
\[ K( A') \to K(A)^{hC_n}  \]
is an $\epsilon$-equivalence. 
\end{example}

\begin{example}\label{ex:ko}
Consider the $C_2$-Galois extension $\KO \to \KU$. 
In this case, the class in $K_0(KO)$ of $\KU \in \perf(\KO)$ is equal to 2, in view of Wood's
theorem (cf.~\cite[p.~206]{Ada95}) $\KO \otimes \Sigma^{-2} \mathbb{C}P^2 \simeq \KU$. 
Therefore, 
Condition~\ref{condition}
is satisfied and we find that the fiber of 
\( K(KO ) \to K(KU)^{hC_2}  \)
is an $\epsilon$-spectrum. 

We note that the comparison problem here was raised in \cite[Rem.~ 2.13]{AR12}. In \cite{BL14}, it was shown
that the map $K(\KO)\otimes \mathbb{Q} \to (K(\KU)\otimes \mathbb{Q}))^{hC_2}$ is an equivalence using localization sequences. 
\end{example}

As indicated in the introduction, we can actually do better in this case.

\begin{thm} \label{f2modulekoku}
\begin{enumerate}
	\item \label{koku} The fiber of the comparison map $K(\KO) \to K(\KU)^{hC_2}$  admits the structure
of an $\mathbb{F}_2$-module spectrum. 

	\item \label{eo} Let $E_{p-1}$ denote Lubin-Tate $E$-theory at the height
	$p-1$, so that $C_p$ is a subgroup of the Morava stabilizer group \cite{Hew95}. Then, the fiber of the comparison map $K(E_{p-1}^{hC_p}) \to K(E_{p-1})^{hC_p}$  admits the structure
of an $\mathbb{F}_p$-module spectrum. 
\item More generally, if $R \to R'$ is a $G$-Galois extension and the
wrong-way map $K_0(R') \to K_0(R)$
has image containing a prime number $p$, then the fiber of the map $K(R) \to
K(R')^{hG}$ admits the structure of an $\mathbb{F}_p$-module. 
\end{enumerate}
\end{thm} 

\begin{proof}
We will prove the third claim and then explain
why the first and second claims are special cases. 
For the third claim, we will apply \Cref{Fpmodulespectrumresult}. 

To apply this result we will first show that if $R \to R'$ is a $G$-Galois extension
of $\e{\infty}$-rings 
in the sense of Rognes (with $G$ finite), then $\perf(R')$ is dualizable in $\md_{\mathrm{Perf(R)}}(\catst)$.
 By \cite[Thm.\ 3.15]{AG14} (originally due to \cite[Prop. 1.5]{Toen12} for
 simplicial commutative rings), this statement is equivalent to
the assertion that $R'$ is a \emph{smooth and proper} $R$-algebra, that is: 
\begin{enumerate}
\item $R'$ is a perfect $R$-module, so all mapping $R$-module spectra in 
$\perf(R')$ are perfect as $R$-modules.
\item  $R' $ is perfect as an $R' \otimes_R R'$-module. 
\end{enumerate}
The first condition is a general fact about Galois extensions: the  perfect $R$-modules are the dualizable $R$-modules and $R'$ is dualizable by \cite[Prop.\ 6.2.1]{Rognes08}. The second condition follows from the
equivalence of algebras $R' \otimes_R R' \simeq \prod_{g \in G} R'$, which implies that $R'$
is a direct factor of a free $R' \otimes_R R'$-module of rank one.

Suppose $R \to R'$ is a $G$-Galois extension and the image of $K_0(R') \to
K_0(R)$ contains the prime $p$. In this case, we find that the fiber of the map
in $\mot_{\aleph_0}(R)$
\[ \uadd( \perf(R)) \to \mathrm{Tot}\left(\uadd( \perf(R')) \rightrightarrows
\uadd(\perf(R' \otimes_R R')) \triplearrows \dots \right)   \]
or equivalently
\[ \mathrm{fib}\left(
\uadd(\perf(R)) \to \uadd(\perf(R'))^{hG}\right) \]
admits the structure of an $\mathbb{F}_p$-module in $\mot_{\aleph_0}(R)$. 
We now apply the functor $$\hom_{\mot_{\aleph_0}(R)}( \mathbf{1}, \cdot) \colon
\mot_{\aleph_0}(R) \to \sp$$
(which is continuous and cocontinuous) and
use the identification (cf.\  \cite[Thm.~7.13]{BGT13} when $R = S^0$ and
\cite[Thm.~5.24]{HSS15} for $R$ arbitrary) 
with $K$-theory to conclude that
\[ \mathrm{fib}\left(K(R) \to K(R')^{hG}\right)  \]
admits the structure of an $\mathbb{F}_p$-module spectrum.

Finally, we need to check that  the relevant prime $p$ belongs to the image of the transfer map in
the cases of $\KO \to \KU$ and $E_{p-1}^{hC_p} \to E_{p-1}$. We already observed this for the extension $\KO\to \KU$ in \Cref{ex:ko}. The claim for the extensions $E_{p-1}^{hC_p}\to E_{p-1}$ is \Cref{cor:better_real_transfer}.
\end{proof}

\subsection{Further Galois examples}

In this subsection, we will give several further examples of our Galois descent
results. We will need various tools for verifying Condition~\ref{condition}.
We begin with the observation that if we can control $K_0$ of a ring spectrum,
checking Condition~\ref{condition} (on an arbitrary Galois extension) is often straightforward.

We consider the following condition on a ring spectrum.
Given an $\mathbb{E}_\infty$-ring $A$, let $\mathrm{Idem}(A)$ be the set of
idempotents in $\pi_0(A)$. 
We obtain a map $\mathrm{Idem}(A) \to
\mathrm{Idem}(K(A) \otimes \mathbb{Q})$ that sends an idempotent $e \in \pi_0 A$
to the class of the perfect $A$-module $A[e^{-1}]$. 
\begin{condition}
\label{cb}
The map $\mathrm{Idem}(A) \to \mathrm{Idem}(K(A) \otimes \mathbb{Q})$ is bijective.
\end{condition}

We first show that the map is often injective. 
\begin{lemma} 
\label{ofteninjective}
Let $A$ be an $\mathbb{E}_\infty$-ring. 
Suppose that $A$ has no nontrivial torsion idempotents (equivalently, for every nonzero idempotent $e \in \pi_0(A)$, the rationalization
$A[e^{-1}]_{\mathbb{Q}}$ is nontrivial).
Then the map $\mathrm{Idem}(A) \to \mathrm{Idem}(K(A) \otimes \mathbb{Q})$ is
injective. 
\end{lemma} 
\begin{proof} 
Note first that the map 
$\mathrm{Idem}(A) \to \mathrm{Idem}(K(A) \otimes \mathbb{Q})$ is a map of
Boolean algebras. 
It thus suffices to show the map has trivial kernel. 
Let $e$ be a nontrivial idempotent. We need to show that the class of the
perfect $A$-module $A[e^{-1}]$ is nontrivial in $K_0(A) \otimes \mathbb{Q}$. 
However, under the map $K_0(A) \otimes \mathbb{Q} \to K_0(A[e^{-1}]) \otimes
\mathbb{Q}$, this class is carried to the unit. 
Therefore, it suffices to show that $K_0(A[e^{-1}]) \otimes \mathbb{Q} \neq 0$. 
This follows from the existence of a multiplicative trace map
\[  K_0(A[e^{-1}]) \otimes \mathbb{Q} \to 
\THH_0(A[e^{-1}]_{\mathbb{Q}}) \to \pi_0 A[e^{-1}]_{\mathbb{Q}}.
\]
Since the target is nonzero by assumption, the source must be nonzero too. 
\end{proof}

Condition~\ref{cb} will play a basic role in this subsection, and most of
our results will state that specific $\mathbb{E}_\infty$-ring spectra
satisfy this condition. In particular, it
(together with a mild statement about rationalization) implies the previous 
Condition~\ref{condition} for every Galois extension. 

\begin{prop} 
\label{k0criterion}
Suppose $A$ is an $\mathbb{E}_\infty$-ring such that: 
\begin{enumerate}
\item $A$ satisfies Condition~\ref{cb}.
\item
$A$ has no nontrivial torsion idempotents. 
\end{enumerate}
Then every $G$-Galois extension $A \to B$ satisfies Condition~\ref{condition}.
\end{prop} 
\begin{proof} 
We use the notation in the discussion immediately following
Condition~\ref{condition}. 
Namely, we let $i_* \colon K_0(B) \otimes \mathbb{Q} \to K_0(A) \otimes
\mathbb{Q}$ the rational restriction of scalars map and $i^*
\colon K_0(A) \otimes \mathbb{Q} \to K_0(B) \otimes \mathbb{Q}$ be the usual
map. Let $y = i_*(1)/|G| \in K_0(A) \otimes \mathbb{Q}$. We need to argue that
$y = 1$. 

We have $y^2 = y$
by the relation $B \otimes_A B \simeq \prod_G B$, so $y \in K_0(A) \otimes
\mathbb{Q}$ is an idempotent. 
By Condition~\ref{cb}, every idempotent in $K_0(A) \otimes \mathbb{Q}$ arises
from an idempotent $e \in \pi_0 A$. 
Consider the $\mathbb{E}_\infty$-ring $A' = A[(1-e)^{-1}]$ and the $G$-Galois
extension $A' \to B' \simeq B \otimes_A A'$. 
(Since $A'$ is a perfect $A$-module, base-change along $A \to A'$ preserves the
Galois condition). 
It follows that we have a $G$-Galois extension $A' \to B'$ such that $[B'] = 0
\in K_0(A') \otimes \mathbb{Q}$. We will now show that $A'_{\mathbb{Q}} = 0$.
The hypothesis (2) we have assumed implies that $e = 1$, and we will then be done.

The image of $[B']$ under the map $K_0(A') \otimes \mathbb{Q} \to K_0(B') \otimes 
\mathbb{Q}$ is given by $|G|$ via the Galois property $B' \otimes_{A'} B' \simeq \prod_G B'$.
Thus, if $[B'] \in K_0(A') \otimes \mathbb{Q}$ vanishes,  we conclude that
$K_0(B') \otimes \mathbb{Q}$ is the zero ring and, by the argument with the
trace of \Cref{ofteninjective}, we find that $B'_{\mathbb{Q}} \simeq 0$. 
By
\Cref{galoisexttorsion}, we get that $A'_{\mathbb{Q}} \simeq 0$.\footnote{When $A' \to B'$ is a faithful Galois
extension, then $A' \to B'$ is descendable \cite[Def. 3.18]{MGal}, so that $B'_{\mathbb{Q}} \simeq 0$
implies $A'_{\mathbb{Q}} \simeq 0$ by descent. The use of
\Cref{galoisexttorsion} is to cover the non-faithful case.} 
\end{proof} 

\begin{lemma} 
\label{galoisexttorsion}
Let $B$ be an $\mathbb{E}_\infty$-ring. Suppose the finite group $G$ acts on $B$ in
the $\infty$-category of $\mathbb{E}_\infty$-rings and set $A = B^{hG}$.
Suppose an integer $N$ is
nilpotent in $\pi_0(B)$. Then $N$ is nilpotent in $\pi_0(A)$ too.
\end{lemma} 
\begin{proof} 
First of all, $B$ splits into a direct product of 
its $p$-completions for $p \mid N$, and the $p$-completion $\widehat{B}_p$
has the property that it is annihilated by a power of $p$. 
The splitting is compatible with the action of $G$. Therefore, we may assume
that $N = p$ is a prime number. 

Let $G_p$ be a $p$-Sylow subgroup of $G$. Then $A$ is a retract of $B^{hG_p}$.  It suffices to
show that $p$ is nilpotent in $B^{hG_p}$. 
Therefore, we can reduce us to the case where $G = G_p$ is itself a $p$-group. 
Since $p$-groups are nilpotent, an induction reduces to the case of $G = C_p$,
which we now consider.

So suppose $B$ is an $\mathbb{E}_\infty$-ring with a $C_p$-action and $A \simeq B^{hC_p}$. Suppose  $p$ is nilpotent in
$\pi_0(B)$. By \cite[Thm.\ 1.4]{BlHi15},  the $C_p$-action on $B$ determines a
cofree or Borel-complete
genuine commutative $C_p$-ring spectrum $R$  
with $R^{\left\{1\right\}}\simeq B, R^{C_p} \simeq A$. By the
results of \cite{Bru07}, its $\pi_0$ is
therefore endowed with the structure of a Tambara functor \cite{Tam93}. 
Therefore, we can consider the multiplicative norm $N_{e}^{C_p}\colon \pi_0(B) \to
\pi_0(A)$ as well as the additive norm $\mathrm{Tr}_{e}^{C_p}\colon \pi_0(B) \to
\pi_0(A)$. 
We have, using \cite[p.\ 1398, (v)]{Tam93},
\[ N_e^{C_p}(p) = p + (p^{p-1} - 1) \mathrm{Tr}_{e}^{C_p}(1) \in \pi_0(A).  \]
Note that $x = \mathrm{Tr}_e^{C_p}(1)$ satisfies $x^2 = px$. Since for $n\ge 0$ sufficiently large, 
$p^n x = \mathrm{Tr}_e^{C_p}(p^n) =0 $, it follows that $x$ is nilpotent:
in fact, $x^{n+1} = 0$. 
Since $N_e^{C_p}(p)$ is nilpotent in $\pi_0(A)$ too, it follows that $p$ is
nilpotent in $\pi_0(A)$. 
\end{proof} 

We next exhibit a class of $\mathbb{E}_\infty$-rings $A$ for which 
$K_0(A)$ admits a purely algebraic description, and as a result, the above conditions are more amenable for them.
In particular, we will show that in this case Condition~\ref{cb} is satisfied.

\begin{prop} 
\label{controlk0}
Let $A$ be an $\mathbb{E}_\infty$-ring spectrum. Suppose $\pi_*(A)$ is even, regular, and noetherian of finite Krull dimension.
Then we have an isomorphism
of commutative rings
\[ K_0(A) \simeq K_0( \pi_*(A)) ,   \]
where we consider $\pi_*(A)$ as an ungraded ring. 

\end{prop} 
\begin{proof} 
We define the morphism $$K_0(A) \to K_0(\pi_*(A))$$
as follows.
Note that by regularity, the target is the Grothendieck group of finitely generated
$\pi_*(A)$-modules. 
Given a perfect $A$-module $M$, we consider $\pi_*(M) \simeq
\pi_{\mathrm{even}}(M) \oplus \pi_{\mathrm{odd}}(M)$ as a finitely generated
$\pi_*(A)$-module and take the class $[\pi_{\mathrm{even}}(M)] -
[\pi_{\mathrm{odd}}(M)]$ in $K_0$. The long exact sequence in
homotopy implies easily that this  defines a map  
$K_0(A) \to K_0(\pi_*(A))$  as desired. 

To define the map in the opposite direction, let $\mathcal{C}$ be the
category of finitely generated, \emph{evenly} graded projective $\pi_*(A)$-modules. To any
object of $\mathcal{C}$ we can associate a perfect $A$-module (given by a
retract of a sum of shifts of $A$ itself) and we easily obtain a
multiplicative map
\[ K_0(\mathcal{C}) \to K_0( A).  \]
For any $P \in \mathcal{C}$, the class of $P$ and its shift $P(2)$
(defined such that $P(2)_k = P_{k+2}$) map to
the same class in $K_0(A)$.
By a theorem of van den Bergh \cite{vdB86} (cf.\ also \cite[Cor.~6.4.2]{Haz16},
but applied to the graded ring $B_*$ with $B_{\ast} = A_{2 \ast}$), 
we have 
an isomorphism of rings
\begin{equation}  \label{kgraded} K_0(\pi_*(A)) \simeq K_0(\mathcal{C})/\left \langle  [P] -
[P(2)]\,\mid\, P\in{\mathcal C}\right\rangle  ,\end{equation}
so we obtain a map
\[ K_0(\pi_*(A)) \to K_0(A).  \]
It suffices now to show that the two composite maps $K_0(A) \to K_0(\pi_*(A)) \to K_0(A) $
and $ K_0(\pi_*(A)) \to K_0(A) \to K_0(\pi_*(A)) $ are the identity. 

For the first claim, we first observe that $K_0(A)$ is generated by
classes corresponding to retracts of graded free $A$-modules. To see this, we fix a
perfect $A$-module $M$ and induct on the homological dimension of $\pi_*(M)$ as a
$\pi_*(A)$-module. If $\dim_{\pi_*(A)} \pi_*(M) =0 $, then 
$\pi_*(M)$ is projective so that 
$M$ itself is a retract of a graded free $A$-module. If 
$\dim_{\pi_*(A)} \pi_*(M) >0 $, then we choose a graded free $A$-module and a
map $F \to M$ inducing a surjection on homotopy. We have then a cofiber
sequence $M' \to F \to M$, where $\dim_{\pi_*(A)}\pi_*(M') = \dim_{\pi_*(A)}
\pi_*(M) -1$. Using $[M] = [F] - [M']$ and induction, we can conclude that $[M]$
belongs to the subgroup of $K_0(A)$ generated by the classes corresponding to
retracts of graded free $A$-modules. 

Suppose $M$ is a retract of a graded free $A$-module. In this case, one checks
directly the map $K_0(A) \to K_0(\pi_*(A)) \to K_0(A)$ carries $[M]$ to
itself, as desired. Since the classes $[M]$ generate $K_0(A)$, this completes
the verification of this case.

For the second map, we use \eqref{kgraded}. Consider a class in $K_0(\pi_*(A))$
represented by an evenly graded $\pi_*(A)$-module $P_*$ which 
is projective. 
We can find an $A$-module $P$ with $\pi_*(P) \simeq P_*$.
The map $K_0(\pi_*(A)) \to K_0(A)$ carries $[P_*] \mapsto [P]$. It follows
easily that the composite $K_0(\pi_*(A)) \to K_0(\pi_*(A))$ is the identity
on $[P_*]$, and therefore in general. 

\end{proof} 

\begin{cor}\label{cor:check-K-0} Assume $A$ is an $\mathbb{E}_\infty$-ring such
that $\pi_*(A)$ is even, regular, noetherian and of finite Krull dimension. Then $A$
satisfies Condition~\ref{cb}. 
\end{cor}

\begin{proof} Without loss of generality, we can assume that $\pi_0(A)$ (and
thus $\pi_*(A)$) has no nontrivial idempotents,  after splitting into
finitely many factors. 
The augmentation map $K_0(A)\simeq K_0(\pi_*(A))\to\mathbb{Z}$ takes values in $\mathbb{Z}$, because $\pi_0(A)$ is connected. By 
\Cref{algebraicunit}, its kernel is locally nilpotent, hence
$K_0(A)\otimes\mathbb{Q}$ is a nilpotent thickening of $\mathbb{Q}$, and does therefore not contain any non-trivial idempotents.
\end{proof}

Next, we include a basic tool that enables us to check
Condition~\ref{cb} via descent.

\begin{prop} 
\label{criterionforgaloisdesc}
Suppose $A$ 
is an $\mathbb{E}_\infty$-ring. 
Suppose that there exists an $\mathbb{E}_\infty$-$A$-algebra $A'$ such that: 
\begin{enumerate}
\item  
$A'$ is
perfect as an $A$-module and the map $K_0(A') \otimes \mathbb{Q} \to K_0(A)
\otimes \mathbb{Q}$ from restriction of scalars is surjective. 
\item
$A$ is $A'$-complete. 
\item $A'$ satisfies Condition~\ref{cb}.
\item  Suppose that either $A'$ has no nontrivial idempotents, or $A'$ and $A' \otimes_A A'$ have no nontrivial torsion idempotents. 
\end{enumerate}
Then $A$ satisfies Condition~\ref{cb} and has no nontrivial torsion idempotents.\end{prop} 
\newcommand{\idem}{\mathrm{Idem}}
\begin{proof} 
We have, since $A$ is $A'$-complete and $A'$ is dualizable in $\md(A)$, an equivalence \cite[Prop.~2.21]{MNN15i}
\[ A \simeq \mathrm{Tot}\left(A' \rightrightarrows A' \otimes_A A' \triplearrows
\dots \right). \]
We also know that we have an equivalence
\[ K(A)_{\mathbb{Q}} \simeq  \mathrm{Tot}\left(K(A')_{\mathbb{Q}}
\rightrightarrows K(A' \otimes_A A')_{\mathbb{Q}} \triplearrows
\dots \right). \]
as a (very) special case of \Cref{finitedesc}: in fact, in this case the above
augmented cosimplicial object is seen to be split. 
The functor $\mathrm{Idem}\colon \clg(\sp) \to \mathrm{Set}$ commutes with all
limits (as it is corepresentable by the $\mathbb{E}_\infty$-ring $S^0 \times S^0$), and
we find that 
we have equalizer diagrams
\begin{gather*}
\idem(A) \to \left( \idem(A') \rightrightarrows \idem(A' \otimes_A A') \right)
\\
\idem(K(A)_{\mathbb{Q}}) \to \left( \idem(K(A')_{\mathbb{Q}})
\rightrightarrows \idem(K(A' \otimes_A A')_{\mathbb{Q}}) \right)
.\end{gather*}
We have a natural transformation between them. 

A diagram-chase now completes the proof. First of all, since $\idem(A') \to
\idem(K(A')_{\mathbb{Q}})$ is an isomorphism, 
we get that $\idem(A) \to \idem(K(A)_{\mathbb{Q}})$ is at least \emph{injective.}
For example, suppose that $A', A'
\otimes_A A'$ have no nontrivial torsion idempotents. 
Since by assumption $\idem(A')
\to \idem(K(A')_{\mathbb{Q}})$ is an isomorphism, and since by
\Cref{ofteninjective} the map 
$\idem(A' \otimes_A A')
\to \idem(K(A' \otimes_A A')_{\mathbb{Q}})$
is injective, we can conclude that $\idem(A) \to \idem(K(A)_{\mathbb{Q}})$ is
an isomorphism. The case where $A' $ has no nontrivial idempotents is easier
and we find that $\idem(A) \simeq \idem(K(A)_{\mathbb{Q}}) \simeq
\left\{0, 1\right\}$. 
\end{proof} 

For our example with topological modular forms, we will need to check
Condition~\ref{cb} by Zariski localization on the base.  We will use the following result.

\begin{cor} 
\label{reducecblocal}
Suppose $A$ is an $\mathbb{E}_\infty$-ring. Suppose every localization 
$A_{(p)}$  at a prime $p$ satisfies
Condition~\ref{cb}. Then $A$ satisfies
Condition~\ref{cb}.
\label{1primecondition}
\end{cor} 
\begin{proof} 
We consider the following two examples of sheaves with values in the $\infty$-category
$\clg(\sp_{\geq 0})$ of connective $\mathbb{E}_\infty$-ring spectra on the Zariski site of
$\mathrm{Spec}(\mathbb{Z})$: 
\begin{enumerate}
\item $\mathcal{F}_1$ 
assigns to the open subset 
$\mathrm{Spec}(\mathbb{Z}[1/N])$ the $\mathbb{E}_\infty$-ring $\tau_{\geq 0}A[1/N]$. 
\item
$\mathcal{F}_2$ 
assigns to the open subset 
$\mathrm{Spec}(\mathbb{Z}[1/N])$ the $\mathbb{E}_\infty$-ring
$K(A[1/N])_{\mathbb{Q}}$. Zariski descent in $K$-theory 
(a special case of Nisnevich descent in $K$-theory, cf.\ \Cref{nisnevichexcision})
implies that this is a
sheaf of connective $\mathbb{E}_\infty$-ring spectra.
\end{enumerate}
Recall that $\idem\colon \clg(\sp_{\geq 0}) \to \mathrm{Set}$ commutes with limits,
so we obtain sheaves of sets $\idem(\mathcal{F}_1), \idem(\mathcal{F}_2)$. We
have a canonical map of sheaves
\[ \idem(\mathcal{F}_1) \to \idem(\mathcal{F}_2)  \]
on Spec$(\mathbb{Z})$,
and our assumption (and the fact that $\idem$ and $K(-)$ commute with filtered
colimits) implies that the map on stalks at closed points is an isomorphism.
Since the set of closed points is very dense, this implies that
$\idem(\mathcal{F}_1) \simeq \idem(\mathcal{F}_2)$ (see \cite[Expose IV, ex.~ 9.1.7.2,c]{SGA41} and \cite[ex.~10.15 and 10.16]{wedhorn} for more on this). Taking global sections, we
find that $A$ satisfies Condition~\ref{cb} as desired. 
\end{proof}

We will now need various tools that enable one to reduce checking
Condition~\ref{cb} via other types of localization and completion. 
The next lemma is well-known in the algebraic 
context (where $x$ has degree zero) (cf.\ \cite[Ex.~3.19.2]{TT90}). 
\begin{lemma}[] 
Let $A$ be an $\mathbb{E}_\infty$-ring and let $x \in \pi_*(A)$ be a
homogeneous element. Then the diagram \[ \xymatrix{
K(A) \ar[d] \ar[r] & K(A[x^{-1}])  \ar[d]  \\
K(\widehat{A}_x) \ar[r] &  K( \widehat{A}_x[x^{-1}])
}\]
is homotopy cartesian in the $\infty$-category $\sp_{\geq 0}$ of connective
spectra.
\end{lemma} 
\begin{proof} 
Let $\perf_{\xtor}(A)$ denote the $\infty$-category of perfect $A$-modules which are
$x$-power torsion. The functor
\begin{equation}   \label{xtoryinverse}  \perf_{\xtor}(A) \to
\perf_{\xtor}(\widehat{A}_x)\end{equation}
given by tensoring with $\widehat{A}_x$ is an equivalence. 
In fact, any \emph{perfect} $x$-power torsion $A$-module is automatically $x$-adically
complete as $x$ will act nilpotently on it. 

Now comparing the fiber sequences
of \emph{connective} spectra (cf.\ \cite[Prop.~11.16]{Bar16})
\[ \xymatrix{
K(\perf_{\xtor}(A)) \ar[d]^{\simeq}  \ar[r] &  K( \perf(A)) \ar[d]  \ar[r] &  K( \perf(A[x^{-1}])) \ar[d]  \\
K(\perf_{\xtor}(\widehat{A}_x))   \ar[r] &  K( \perf(\widehat{A}_x))  \ar[r] &
K( \perf(\widehat{A}_x[x^{-1}])),   
}\]
shows that we have a fiber square in $\sp_{\geq 0}$ as desired.
\end{proof} 

Using similar arguments as in the proof of
\Cref{reducecblocal}, we find the following result: 

\begin{prop} 
\label{checkaftercompletion}
Let $A$ be an $\mathbb{E}_\infty$-ring. Suppose $x \in \pi_*(A)$ is a homogenous element. 
Suppose $\widehat{A}_x, A[x^{-1}], \widehat{A}_x[x^{-1}]$ satisfy Condition~\ref{cb}. 
Then $A$ satisfies
Condition~\ref{cb}. 

\end{prop}

Using similar reasoning, one also has: 
\begin{prop} 
\label{grZdesc}
\label{quotientbytwoelements}
Let $A$ be an $\mathbb{E}_\infty$-ring and $x_1, \dots, x_n \in \pi_*(A)$. 
Suppose $A/(x_1, \dots ,x_n)$ is contractible and the $\mathbb{E}_\infty$-rings $A[ ( x_{i_1} \dots x_{i_k})^{-1}]$
satisfy Condition~\ref{cb} for any nonempty collection of indices $i_1,
\dots, i_k \in \left\{1, \dots, n\right\}$. 
Then $A$ satisfies
Condition~\ref{cb}. 
\end{prop} 
\begin{proof} 

We use induction on $n$. When $n = 1$, the assertion is obvious, so we
assume $n > 1$ and the result proved for $n - 1$ replacing $n$. 
Let $A'$ be the localization of $A$ away from the perfect $A$-module 
$\mathrm{End}_{\md(A)}( A/(x_1, \dots, x_{n-1}))$ (cf.\ \cite[\S 3]{MNN15i}
for an exposition). 
We then have 
$A'[x_i^{-1}] \simeq A[x_i^{-1}]$ for $1 \leq i\leq n-1$ and $A'/(x_1, \dots,
x_{n-1}) = 0$.  In addition, the map $A \to A'$ induces an
equivalence $A/x_n \simeq A'/x_n$.
As a result, the induced map $\perf(A) \to \perf(A')$ restricts to an
equivalence on $x_n$-power torsion objects. 
We therefore obtain a fiber square
of connective spectra
\[ \xymatrix{
K(A) \ar[d] \ar[r] &  K(A')  \ar[d]   \\ 
K(A[x_n^{-1}]) \ar[r] &  K(A'[x_n^{-1}]).
}\]
By induction, it follows that $A'$ and $A'[x_n^{-1}]$
satisfy Condition~\ref{cb}: in fact, we consider the sequence $x_1, \dots,
x_{n-1} \in \pi_* A'$. 
Therefore, the above fiber square implies that $A$ satisfies Condition~\ref{cb}. 
\end{proof}

\newcommand{\mellb}{\mathfrak{M}_{ell}}

\newcommand{\mell}{{M}_{ell}}
\newcommand{\mellc}{{M}_{\overline{ell}}}
\newcommand{\otop}{\mathcal{O}^{\mathrm{top}}}

We now include our main examples involving $\TMF$, the spectrum of topological modular forms (cf.\ \cite{tmf_book} for a textbook reference).
We first need a general lemma about even periodic derived stacks, cf.\ \cite[Section 2]{MM15} for an exposition of this.
Given an even periodic derived stack $(X, \otop)$, we let $\omega$ denote the
line bundle 
$\pi_2 \otop$ on $X$. 

\begin{lemma} 
\label{qaffbgm}
Let $\mathfrak{X} = (X, \otop)$ be a regular, noetherian even periodic derived
Deligne-Mumford stack. 
Suppose that the map $X \to B \mathbb{G}_m$ classifying $\omega$ is
quasi-affine. Then $\otop(X)$ satisfies Condition~\ref{cb}.
If $X$ is flat over $\spec \mathbb{Z}$, then $\otop(X)$ has no nontrivial torsion idempotents.
\end{lemma} 
\begin{proof} 
The quasi-affineness hypothesis is equivalent to the assertion that there are sections 
$s_1, \dots, s_n \in H^0(X, \omega^{\otimes \bullet})$   such that $X[s_i^{-1}]$ is affine over $B
\mathbb{G}_m$, i.e., arises as the quotient of a $\mathbb{G}_m$-action (or
grading) on the spectrum of a graded ring $R_{i,*}$, and such that the
$\left\{s_i\right\}$ have no common vanishing locus. 
Note that a power of each $s_i$ survives the descent spectral sequence (cf.
the argument of \cite[Prop.~3.24]{MM15});
by passage to such a power, we may assume that $s_i  $ arises from an element
(which for convenience we still denote by $s_i$) in $\pi_{2n_i} \otop(X)$. 

We observe that $\otop(X)[s_i^{-1}]$ arises as the sheaf of global sections of
an even periodic derived stack of the form $\mathrm{Spec}(R_{i,*})/\mathbb{G}_m$
where $R_{i,*}$ is a regular noetherian graded-commutative ring concentrated in
even degrees. In particular, we get
\[ \pi_*( \otop(X))[s_i^{-1}] \simeq R_{i,*},  \]
so by \Cref{cor:check-K-0} we find that $\otop(X)[s_i^{-1}]$ satisfies
Condition~\ref{cb}. 
Similarly, for any collection of indices $i_1, \dots , i_k$, we have that
$\otop(X)[ (s_{i_1} \dots s_{i_k})^{-1}]$ satisfies Condition~\ref{cb}. It
follows that $\otop(X)$ satisfies Condition~\ref{cb} in view of \Cref{grZdesc}.

Finally,  one sees that the idempotents of
$\otop(X)$ correspond precisely to the idempotents in the discrete
ring $\Gamma(X, \mathcal{O}_X)$, and by hypotheses the latter is
torsion-free. 
In fact, on the \'etale site of $X$, we consider the following two presheaves
of $\mathbb{E}_\infty$-rings: the first sends $ \spec R \to X$ to $\otop( \spec
R)$ and the second to $R$ itself. The idempotents of 
each are canonically identified. Since the functor $\mathrm{Idem}\colon \clg(
\sp) \to \mathrm{Set}$ commutes with limits, we find the identification 
$\mathrm{Idem}( \Gamma(X, \otop)) \simeq \mathrm{Idem}( \Gamma(X,
\mathcal{O}_X))$. 
\end{proof}

Let $\mell$ denote the moduli stack of elliptic curves (\cite{deligne-rapoport}, \cite{katz-mazur}, \cite{conrad}). Let $\mellb = (\mell,
\otop)$ be the Goerss-Hopkins-Miller derived stack, i.e., the 
pair of $\mell$ together with the sheaf $\otop$ of even
periodic, elliptic $\mathbb{E}_\infty$-ring
spectra that they construct. 
For any \'etale morphism of DM-stacks $X \to \mell$, we can form
the $\mathbb{E}_\infty$-ring of sections
$\otop( X)$. For example, we have by definition $\otop(\mell) = \TMF$, see \cite{tmf_book}
for more details.

\begin{thm} 
\label{TMFcb}
Given any representable, \'etale and separated morphism of DM-stacks $X \to \mell$, the $\mathbb{E}_\infty$-ring
$\otop(X)$ satisfies Condition~\ref{cb} and has no nontrivial torsion idempotents. In
particular, any Galois extension 
of $\otop(X)$ satisfies Condition~\ref{condition}.
\end{thm}

\begin{proof} 
The final claim follows from the first one together with \Cref{k0criterion}.
By \Cref{1primecondition}, it suffices to prove Condition~\ref{cb} for $\otop(X)_{(p)}$ for a
prime number $p$. We consider three different cases.

\begin{enumerate}
\item  $p \geq 5$. 
In this case,  we have an equivalence of stacks $(\mell)_{(p)} \simeq \mathbb{P}(4, 6)[\Delta^{-1}]$.
Note that $X_{(p)}$ is quasi-affine over 
$(\mell)_{(p)}$ by Zariski's main theorem \cite[Th.~18.12.13]{EGAIV-3}. By \Cref{qaffbgm}, we find that $\otop(X)_{(p)}$ satisfies Condition~\ref{cb}.
\item  $p = 3$. In this case, we have an equivalence $\TMF_{(3)}  \otimes C
\simeq \TMF_1(2)_{(3)}$ for a three cell complex $C$ with even cells (cf.\ 
\cite[Thm.~4.13]{Htmf}). 
The moduli stack $Y = M_{ell, 1}(2)_{(3)}$ has the property that 
$Y\to B\mathbb{G}_m$ is affine (cf. the explicit presentation in \cite[\S
7]{St12}). Note also that 
$Y\to (\mell)_{(3)}$ is a finite \'etale cover.
It follows that $X_{(3)} \times_{\mell}Y $ is quasi-affine over $B
\mathbb{G}_m$ and the map  
of $\mathbb{E}_\infty$-rings
\[ \otop(X)_{(3)}  \to \otop(X_{(3)} \times_{\mell} Y )
\]
exhibits the target as a perfect module over the source; in fact
by 0-affineness of the derived moduli stack (cf.\ \cite[Thm.~7.2]{MM15}) 
\[  \otop(X_{(3)} \times_{\mell}Y )  \simeq \otop( X_{(3)}) \otimes_{\TMF_{(3)}} \TMF_1(2)_{(3)}\simeq \otop(X)_{(3)}
\otimes C. \]

By \Cref{qaffbgm}, 
$\otop(X_{(3)} \times_{\mell}Y )$
satisfies Condition~\ref{cb} and has no nontrivial torsion idempotents. 
In addition, $$\otop(X_{(3)} \times_{\mell}Y ) \otimes_{\otop( X_{(3)})}
\otop(X_{(3)} \times_{\mell}Y )
\simeq  \otop(Y\times_{\mell}Y ) 
$$
also has no nontrivial torsion idempotents. 
The complex $C$ enables us to conclude that 
$\otop(X)_{(3)}$ satisfies Condition~\ref{cb} and has no nontrivial
torsion idempotents, by 
\Cref{criterionforgaloisdesc}. 
\item $p = 2$. Here we argue similarly as in (2), using an eight cell complex $DA(1)$
and the equivalence $\TMF_{(2)} \otimes DA(1)  \simeq \TMF_1(3)_{(2)}$ due to
Hopkins-Mahowald (see \cite[\S~4]{Htmf}).
\end{enumerate}

\end{proof} 

\begin{remark} It seems remarkable that to establish Condition~\ref{cb}
for Galois extensions of $\TMF$,
we need the existence of certain specific finite complexes. In particular, we do not know the analogous result for the finite Galois extensions provided by the theory of topological automorphic forms from 
\cite{TAF}. In fact, we do not know a single example of a finite Galois extension which violates Condition~\ref{condition}.
\end{remark}

\begin{example} 
Consider the Galois extension $\TMF[1/n]=\mathcal{O}^{\mathrm{top}}(M_{ell}\left[ \frac{1}{n}\right])
\to \TMF(n)$ with Galois group $GL_2(
\mathbb{Z}/n)$ (cf.\ \cite[Thm.~7.6]{MM15}). 
It follows that the map 
\[ K( \TMF[1/n]) \to K( \TMF(n))^{h GL_2(\mathbb{Z}/n\mathbb{Z})}   \]
becomes an equivalence after any periodic localization. 
\end{example}

We now describe the analog of our results for the Hill-Lawson extension of the
sheaf $\otop$ on the \'etale site of $\mell$. 
Let $\mellc$ denote the compactified moduli stack of elliptic curves. 
In \cite{HiLa16}, Hill and Lawson describe the \emph{log-\'etale site} of
$\mellc$ and endow it with a sheaf  $\otop$ of $\mathbb{E}_\infty$-ring
spectra which restricts to the previous sheaf on the \'etale site of $\mell$.

\begin{thm} \label{thm:log-tmf}
Let $(X, M_X)\to \mellc$ be a representable and separated log-\'etale map of DM-stacks. Then $\otop(X, M_X)$ satisfies Condition~\ref{cb} and has no nontrivial
torsion idempotents. For example, every finite Galois extension of the
$\mathbb{E}_\infty$-ring spectra $\Tmf,
\Tmf_0(n), \Tmf_1(n)$ ($n\ge 1$) satisfies Condition~\ref{condition}.
\end{thm} 
\begin{proof} 
We use \Cref{checkaftercompletion}. 

After inverting the modular form $\Delta^{24} \in \pi_{576} \Tmf$, the log-\'etale site of $\mellc$ is
identified with the \'etale site of $\mell$.  
It follows that $\otop( X, M_X)[\Delta^{-24}]$ satisfies Condition~\ref{cb}
in view of \Cref{TMFcb}. 

The $\mathbb{E}_\infty$-ring $\otop(X)$, and more generally the sheaf $\otop$, is defined in \cite{HiLa16} by gluing 
together $\otop( X \times_{\mellc}\mell)$ and the completion at the cusp. In
particular, we
can also evaluate $\otop$ on log-\'etale morphisms to the completion. 
Now, we need to consider the completion $\otop(\widehat{X}) = \widehat{\otop(X, M_X)}_{\Delta^{24}}$. 
This is obtained by completing the stack $X$ at the modular form $\Delta$. 
Note first that $\widehat{(\mellc)}_{\Delta} \simeq \left( \mathrm{Spf}
\mathbb{Z}[[q]] \right)/C_2$ via the Tate curve and its automorphism given by
$-1$. 
We consider the $\Delta$-completed log stack $\widehat{Y} = X \times_{\mellc} \mathrm{Spf}
\mathbb{Z}[[q]]$, so that $\widehat{Y}$ is log-\'etale over the log-scheme $( \mathrm{Spf}
\mathbb{Z}[[q]], q)$. 
We claim first that $\otop( \widehat{Y})$ is even periodic and $\pi_0
\otop(\widehat{Y})$ is
regular. 
This follows from the discussion \cite[Cor.~2.19]{HiLa16} of the log-\'etale site of $(\mathbb{Z}[[q]],
q)$ and the construction in \cite[Sec.~5.1]{HiLa16}. Finally, the map
$\otop(\widehat{X}) \to \otop( \widehat{Y})$ exhibits an equivalence
\[ \otop( \widehat{X}) \otimes \Sigma^{-2}\mathbb{CP}^2 \simeq \otop(
\widehat{Y}),  \]
by Wood's theorem, which implies that the $C_2$-action on 
$\otop(
\widehat{Y}) \otimes \Sigma^{-2} \mathbb{CP}^2 \simeq KU[[q]] \otimes
\Sigma^{-2}\mathbb{CP}^2$ is given by the coinduced representation. Therefore, we find by
descent (\Cref{criterionforgaloisdesc}) along $\otop( \widehat{X}) \to \otop( \widehat{Y})$ and 
$\otop( \widehat{X})[q^{-1}] \to \otop( \widehat{Y})[q^{-1}]$
that $\otop( \widehat{X}), \otop( \widehat{X})[q^{-1}]$ satisfy
Condition~\ref{cb}. Putting everything together, we find by
\Cref{checkaftercompletion} that $\otop(X)$ satisfies Condition~\ref{cb}. 
\end{proof}

\begin{example} 
Let $n$ be square-free. Then we have a $(\mathbb{Z}/n)^{\times}$-Galois
extension $\Tmf_0(n) \to \Tmf_1(n)$ under $\Tmf[1/n]$ \cite[Thm.~7.12]{MM15}. 
This Galois extension satisfies Condition~\ref{condition}, so that the map
\[ K( \Tmf_0(n)) \to K( \Tmf_1(n))^{h (\mathbb{Z}/n)^{\times}}  \]
is an $\epsilon$-equivalence. 
\end{example}

Finally, for completeness we give an example of a torsion $\mathbb{E}_\infty$-ring that does not
satisfy Condition~\ref{cb}. We do not know any non-torsion examples. 
\begin{example} 
 Let $G$ be a nontrivial finite $p$-group and let $k$ be a
field of characteristic $p$. Then there is an identification between
$\mathrm{Perf}(k^{tG})$ and the \emph{stable module $\infty$-category} of
finite-dimensional $k[G]$-modules modulo projectives (cf.\ \cite{Kel94} or
\cite{MTorus}). Using this, we can calculate $K_0(k^{tG})$.

Every
finite-dimensional $k[G]$-representation has a finite filtration with subquotients given by the
trivial representation, and the representation $k[G]$ is identified with zero
in the stable module $\infty$-category. This forces $K_0(k^{tG}) \simeq
\mathbb{Z}/r$ where $r \mid |G|$. We also have a homomorphism $K_0(k^{tG}) \to
\mathbb{Z}/|G|$ that sends a representation to its dimension modulo $|G|$. 
Since two representations become isomorphic in the stable module
$\infty$-category if and only if they are \emph{stably isomorphic} as
representations (i.e., become isomorphic after adding free summands), 
it follows that $K_0( k^{tG}) \simeq \mathbb{Z}/|G|$. 
In particular, $K_0(k^{tG}) \otimes \mathbb{Q} = 0$. 
\end{example} 

\subsection{Non-Galois examples of descent}

In this subsection, we record a few additional examples where one has descent
but which are not Galois.  

\begin{example} \label{ex:koconn}
We consider the connective version of \Cref{ex:ko} above. 
Consider the map of $\mathbb{E}_\infty$-rings $\ko \to \ku$. 
Since $\ku \simeq \ko \otimes \Sigma^{-2} \mathbb{CP}^2$, we observe that the class of $\ku$
in $K_0(\ko)$ is equal to $2$. 
As a result, 
we conclude by 
\Cref{finitedesc} that we have 
an $\epsilon$-nilpotent limit diagram
given by the augmented cosimplicial object
\[ K(\ko ) \to \left( K( \ku) \rightrightarrows K( \ku \otimes_{\ko} \ku)
\triplearrows \dots \right).    \]

To compare this to \Cref{ex:ko}, we consider the following diagram of
localization sequences \cite{BM08,BL14}:
\[ \xymatrix{
K(\mathbb{Z}) \ar[d] \ar[r] & K(\ko) \ar^{\simeq_\epsilon}[d] \ar[r] & K(\KO)\ar^{\simeq_\epsilon}[d]\\
\mathrm{Tot}(F^\bul) \ar[d] \ar[r] & \mathrm{Tot}(K(\ku^{\otimes_{\ko}\bul+1})) \ar[d] \ar[r] & \mathrm{Tot}(K(\KU^{\otimes_{\KO}\bul+1})) \ar^{\simeq}[d]\\ 
F({BC_2}_+, K(\mathbb{Z}))\simeq  K(\mathbb{Z})^{hC_2} \ar[r] & K(\ku)^{hC_2}\ar[r] & K(\KU)^{hC_2}
}\]

We have established the first two $\epsilon$-equivalences in the first row of
vertical arrows, which implies the induced map on the fibers is an
$\epsilon$-equivalence. 
 It follows that $K(\ko)\to K(\ku)^{hC_2}$ is an $\epsilon$-equivalence if and only if the composite map $K(\mathbb{Z})\to F({BC_2}_+, K(\mathbb{Z}))$ is an $\epsilon$-equivalence. 
However, by comparing with $\KU$ one sees that the map 
\[ K(\mathbb{Z}) \to K(\mathbb{Z})^{hC_2}  \]
is not even a rational equivalence. 
To see this, we observe that the unit and the transfer of the unit 
from $\pi_0 K(\mathbb{Z})  \otimes \mathbb{Q}\xrightarrow{\mathrm{Tr}}
\pi_0(K(\mathbb{Z})^{hC_2}) \otimes \mathbb{Q}$
are
linearly independent in $\pi_0( K(\mathbb{Z})^{hC_2}) \otimes \mathbb{Q}$
because they are linearly independent 
in 
$\pi_0( \KU^{hC_2})\otimes \mathbb{Q}$ and we have a $K$-theoretic map
$K(\mathbb{Z}) \to \KU$. As a result, we conclude that $K(\ko) \to
K(\ku)^{hC_2}$ is \emph{not} an $\epsilon$-equivalence (or even a rational
equivalence).  The same argument also shows that $L_{K(1)}K(\ko)\to
(L_{K(1)}K(\ku))^{hC_2}$ fails to be an equivalence when $K(1)$ is the first Morava $K$-theory at the prime 2.

On the other hand, it is evident that if we rationalize before taking homotopy fixed points we obtain an equivalence $K(\mathbb{Z})\otimes \mathbb{Q}\to (K(\mathbb{Z})\otimes\mathbb{Q})^{hC_2}$ and hence the vertical composites are all equivalences if we rationalize first and then take homotopy fixed points.
\end{example} 

\begin{example} \label{ex:tmf13}
Consider the map $\tmf[1/3] \to \tmf_1(3)$. 
We claim that this satisfies the assumptions of \Cref{finitedesc}, so we have an
$\epsilon$-nilpotent limit diagram
\[ K(\tmf[1/3] ) \to \left( K( \tmf_1(3)) \rightrightarrows K( \tmf_1(3)
\otimes_{\tmf[1/3]} \tmf_1(3))
\triplearrows \dots \right).    \]
In fact, it suffices to show that the class $[\tmf_1(3)] - 8 $ is nilpotent in
$K_0( \tmf[1/3])$. It suffices to check this after localizing at 2 and inverting
2. When localizing at 2, this follows from the complex $DA(1)$ and the
equivalence $\tmf_{(2)} \otimes DA(1) \simeq \tmf_1(3)_{(2)}$ (cf.\ 
\cite[\S 4]{Htmf}). After inverting $2$, we see the equivalence just by
considering homotopy. 
\end{example}

\appendix

\section{\'Etale descent for spectral algebraic spaces}

\label{sec:etaledesc} 

We have seen in Proposition \ref{finiteetalesatisfied} that periodically
localized algebraic $K$-theory (or indeed any periodically localized additive
invariant) satisfies finite flat descent on $\mathbb{E}_\infty$-rings.
In this appendix, we will describe the argument for Nisnevich descent for such
invariants in the setting of spectral algebraic spaces as in \cite[Ch. 3]{lurie_sag}. This argument is due to
Thomason-Trobaugh \cite{TT90}, and we will indicate the necessary modifications
in the present setting.
Compare also the treatment by Barwick \cite[Prop.~12.12]{Bar16}. 
As a result, one obtains a basic \'etale descent result
(\Cref{thm:etaledesc} below) which applies to invariants such as algebraic
$K$-theory after periodic localization.

We need the definition of a \emph{localizing invariant} in $\catst$, as in
\cite{BGT13}, although we shall not assume that our invariant commutes with
filtered colimits.

\begin{definition} 
Let $\mathcal{R} \in \clg( \catst)$. 
Suppose $\mathcal{A}_1 \to \mathcal{A}_2$ is a morphism 
in $\md_{\mathcal{R}}(\catst)$ which is fully faithful on underlying
$\infty$-categories. 
The \emph{Verdier quotient} $\mathcal{A}_2/\mathcal{A}_1$ is the pushout
$\mathcal{A}_2 \cup_{\mathcal{A}_1} \left\{0\right\}$ in $\md_{\mathcal{R}}(\catst)$. 
A \emph{weakly localizing invariant} of $\mathcal{R}$-linear $\infty$-categories with
values in a presentable, stable $\infty$-category $\mathcal{D}$ is a functor
\[ F \colon  \md_{\mathcal{R}}(\catst) \to \mathcal{D}  \]
which carries Verdier quotient sequences to cofiber sequences in
$\mathcal{D}$. 
It follows in particular that $F$ is weakly additive in the sense of \Cref{weaklyadditive}.
\end{definition}

\newcommand{\qcoh}{\mathrm{QCoh}}
Let $X$ be 
a quasi-compact quasi-separated (hereafter \emph{qcqs}) spectral algebraic space. 
Recall \cite[Ch.~2]{lurie_sag} that one has a presentable, symmetric monoidal stable
$\infty$-category $\qcoh(X)$ of \emph{quasi-coherent sheaves} on $X$.
By \cite[Prop.~9.6.1.1]{lurie_sag}, the $\infty$-category $\qcoh(X)$ is
compactly generated, and the compact objects are 
given by the dualizable objects, which are denoted $\perf(X)$.

\begin{remark} 
The compact generation of $\qcoh(X)$ has a long history. 
For classical quasi-compact, separated schemes, the result is due
to Neeman \cite[Prop. 2.5]{Nee96}. 
For classical qcqs schemes, the result appears in \cite[Thm.~3.1.1]{BvB03}. 
That argument is extended to derived schemes in \cite[Prop.~3.19]{BZFN10}. 
\end{remark} 
Here $\perf(X) \in \clg(\catst)$. 
Given a morphism $f \colon Y \to X$ of qcqs spectral algebraic spaces, one obtains a symmetric monoidal pull-back
functor $f^* \colon \qcoh(X) \to \qcoh(Y)$ which restricts to dualizable or
compact objects and yields a functor $f^* \colon \perf(X) \to \perf(Y)$.

\begin{definition} 
We denote by $X_{\acute{e}t}$ the \'etale site of $X$, which has objects 
\'etale maps $U\to X$ with $U$ quasi-compact, and carries the \'etale 
topology. \end{definition} 
Our descent result is the following.

\begin{thm}
\label{thm:etaledesc}
Let $X$ be a qcqs spectral algebraic space.  Suppose that
$F\colon \md_{\operatorname{Perf}(X)}(\catst)\rightarrow \mathcal{D}$ 
is a weakly localizing invariant where $\mathcal{D}$ is $T(n)$-local for some implicit
prime $p$ and height $n$.
Then the presheaf
$$(U \to X) \mapsto F(\perf(U))$$
on the \'etale site $X_{\acute{e}t}$ of $X$ is a sheaf.
\end{thm}

\newcommand{\qst}{\mathrm{QStk}^{\mathrm{st}}}

For the proof, we will require a local-to-global argument for which the theory
of \emph{stable quasi-coherent stacks} of \cite[Ch.~10]{lurie_sag} will be 
useful.
 A \emph{stable quasi-coherent stack} $\mathcal{C}$ on $X$
assigns to every $R$-point $\eta$ of $X$ (for $R$ a connective
$\mathbb{E}_\infty$-ring) an $R$-linear presentable stable
$\infty$-category $\mathcal{C}_{\eta}$.  
Given a map $f \colon R \to R'$ of connective $\mathbb{E}_\infty$-rings, we
obtain a new $R'$-point of $X$ given by $f^* \eta$ and we have a compatibility
equivalence 
$\mathcal{C}_{f^* \eta} \simeq \mathcal{C}_{\eta} \otimes_R R'$. 
More precisely: 

\begin{definition}[{\cite[Def. 10.1.2.1]{lurie_sag}}]
The $\infty$-category $\qst(X)$ of stable quasi-coherent stacks on $X$ is defined
as $\varprojlim_{\spec R \to X} \mathrm{LinCat}^{st}_R$ where
$\mathrm{LinCat}^{st}_R$ denotes the $\infty$-category of presentable
$R$-linear $\infty$-categories. 
Given a stable quasi-coherent stack $\mathcal{C}$, we can form the \emph{global
sections} $\qcoh(X; \mathcal{C})$ which form a $\qcoh(X)$-linear presentable
stable $\infty$-category, cf. \cite[Construction 10.1.7.1]{lurie_sag}.
\end{definition}

We remind the reader of the relevant notion of support. Let $X$ be  a qcqs spectral algebraic space, $Z\subseteq |X|$ a closed subset, and assume that the complementary open subset $j \colon U \subset X$ is quasi-compact. 

\begin{definition}[{\cite[Def. 7.1.5.1]{lurie_sag}}]
Let $\qcoh_Z(X) \subset \qcoh(X)$ 
be the full 
subcategory of those $\mathcal{F} \in \qcoh(X)$ such that $j^* \mathcal{F} =
0$. 
Let $\perf_Z(X) = \qcoh_Z(X) \cap \perf(X)$ denote those perfect modules on
$X$ which restrict to zero on $U$. 
\end{definition}

We now show that $\qcoh_Z(X)$ is compactly generated,
with compact objects precisely $\perf_Z(X)$. This will be relatively easy to
check when $X$ is affine. 
In general, the basic local-to-global observation that will be used throughout
is that $\qcoh_Z(X)$ arises from a
stable quasi-coherent stack.

\begin{cons}
We define the following three quasi-coherent stacks $\mathcal{C}_1,
\mathcal{C}_2, \mathcal{C}_3 \in \qst(X)$ on $X$. Let $\eta \colon \spec R \to X$ be an $R$-point of $X$ (with $R$ a connective $\mathbb{E}_\infty$-ring).
\begin{enumerate}
\item  We define $\mathcal{C}_1$ via $(\mathcal{C}_1)_{\eta} = \md(R) = \qcoh( \spec R)$ itself. This is the unit in the
$\infty$-category of stable quasi-coherent stacks.
Clearly $\qcoh(X; \mathcal{C}_1) \simeq \qcoh(X)$ by definition.
\item We define $\mathcal{C}_2$ via $(\mathcal{C}_2)_{\eta} = \qcoh( \spec R \times_X U)$. This is the
push-forward of the unit in $\qst(U)$ along $j \colon U \to X$.
We have $\qcoh(X; \mathcal{C}_2) \simeq \qcoh(U)$. 
\item 
Let $(\mathcal{C}_3)_{\eta} \subset \qcoh( \spec R) = \md(R)$ denote the
subcategory of those quasi-coherent sheaves on $\spec R$ which restrict to zero
on $\spec R \times_X U \subset \spec R$. Then
$\left\{(\mathcal{C}_3)_{\eta}\right\}_{\eta \colon \spec R \to X}$ assembles into
a quasi-coherent stack $\mathcal{C}_3$: in fact, $\mathcal{C}_3$ is the
pull-back $ \mathcal{C}_1 \times_{\mathcal{C}_2} 0$. We note that limits in
$\qst(X)$ are computed pointwise (cf. \cite[10.1.3]{lurie_sag}). 
Unwinding the definitions again, we find that $\qcoh(X; \mathcal{C}_3) \simeq
\qcoh_Z(X)$.
\end{enumerate}
\end{cons}

\begin{lemma} 
\label{loccompactgen}
Suppose $X = \spec(A)$ is an affine spectral algebraic space. 
Then $\qcoh_Z(X)$ is compactly generated, and the inclusion $\qcoh_Z(X)\subset
\operatorname{QCoh}(X)$ preserves compact objects.
\end{lemma}
\begin{proof}
This is proved in \cite[Prop.~7.1.1.12, (e)]{lurie_sag} in a more general situation. The argument for the special case at hand runs as
follows. 
Since $U=|X|-Z$ is quasi-compact, we can choose finitely many elements
$f_1,\ldots f_k\in \pi_0A$ such that $U=\bigcup_i \spec(A[f_i^{-1}])$.  Then $\operatorname{QCoh}_Z(X)$ identifies with the full subcategory of $\operatorname{QCoh}(X)=\operatorname{Mod}(A)$ spanned by those $A$-modules $M$ such that
$M[f_i^{-1}]=0$
for all $1\leq i\leq k$. It is not hard to check that one can take as a single compact generator the iterated cofiber $A/(f_1, \dots, f_k)$. The counit of the right-adjoint
to the inclusion $\qcoh_Z(X)\subset\operatorname{QCoh}(X)$ can be given explicitly as
$\mathrm{colim}_{n}\left( M\otimes_A \Sigma^{-k} A/(f_1^n, \ldots, f_k^n)\right)\to M$
for $M\in\operatorname{QCoh}(X)=\operatorname{Mod}(A)$, and thus the right-adjoint visibly commutes with colimits. It follows that the inclusion preserves compact objects.
\end{proof}

The next result was known previously for classical qcqs schemes, see \cite[Thm. 6.8]{rouquier_dimensions}.

\begin{prop} \label{A1}
If $X$ is a qcqs spectral algebraic space and $Z\subset |X|$ a closed subset with quasi-compact open complement, 
then $\qcoh_Z(X)$ is compactly generated.
In addition, the inclusion $\qcoh_Z(X) \subset \qcoh(X)$ preserves compact objects. 
\end{prop} 
\begin{proof} 
We saw that $\qcoh_Z(X)$ arises as the global sections of a 
stable
quasi-coherent
stack $\mathcal{C}_3$. By \cite[Ex. 10.3.0.2, (4)]{lurie_sag}, the discussion following
\cite[Prop. 10.3.0.3]{lurie_sag}, the fact that $X$ is \'etale-locally affine and \Cref{loccompactgen}, we see that $\mathcal{C}_3$ is compactly generated.
By \cite[Prop. 10.3.2.1, $(b)$]{lurie_sag}, the global sections $\qcoh_Z(X)$ are
compactly generated. 
To see that $\qcoh_Z(X) \subset \qcoh(X)$ preserves compact objects, 
note that by \cite[Prop. 10.3.2.6]{lurie_sag}, the assertion is \'etale-local on $X$, hence we can assume that $X=\spec(A)$ is affine and invoke \Cref{loccompactgen} again.
\end{proof}

\begin{cor} 
In the situation of \cref{A1}, the sequence

\begin{equation} \label{keysequence}  \perf_Z(X) \to \perf(X) \stackrel{j^*}{\to}\perf(U)
\end{equation}
is a Verdier quotient sequence in $\md_{\perf(X)}(\catst)$.
\end{cor} 
\begin{proof} 
This follows from \Cref{A1} which implies that the $\mathrm{Ind}$-completion of
\cref{keysequence} is precisely 
\[ \qcoh_Z(X) \to \qcoh(X) \to \qcoh(U).  \]
Compare also \cite[Prop.~7.2.3.1]{lurie_sag}, which implies that we have a
semi-orthogonal decomposition of $\qcoh(X)$. 
\end{proof} 

\begin{cor}[{Cf. \cite[Prop.~12.12]{Bar16}}] 
If $F \colon \md_{\perf(X)}(\catst) \to
\mathcal{D}$ is any weakly localizing invariant, then we have a fiber sequence in $\mathcal{D}$,
\begin{equation} \label{locseq} F(\operatorname{Perf}_Z(X))\rightarrow F(\operatorname{Perf}(X))\rightarrow
F(\operatorname{Perf}(U)).\end{equation}
\end{cor} 
We call \eqref{locseq} the \emph{localization sequence,} and next establish 
an excision result:

\begin{prop} 
\label{A2}
Let $X, Y$ be qcqs spectral algebraic spaces and let $j \colon U \subset X$ be
a quasi-compact open immersion. 
Suppose $f \colon Y \to X$ is a flat morphism. 
Let $Z$ be the reduced, discrete closed complement of $j$ and suppose the map
$Y \times_X Z \xrightarrow{f \times_X Z } Z$ is an equivalence. Then the map
\[ f^*\colon \qcoh_Z(X) \to \qcoh_{f^{-1}(Z)}(Y)  \]
is an equivalence of $\perf(X)$-linear $\infty$-categories.
\end{prop} 
\begin{proof} 
We need to see that if $M\in\operatorname{QCoh}_Z(X)$, then the map
$M\rightarrow f_\ast f^\ast M$
is an equivalence, and that if $N\in\operatorname{QCoh}_{f^{-1}Z}(Y)$ then $f_\ast N\in \operatorname{QCoh}_Z(X)$ and
$f^\ast f_\ast N\rightarrow N$
is an equivalence.
Without loss of generality, we may assume that $X = \spec(A)$ is affine 
It suffices to check the desired claim on homotopy group sheaves. 
Since $f$ is flat, the functor $f^*$ simply tensors up over $\pi_0$ on homotopy
group sheaves. 
If $M \in \qcoh_{f^{-1} Z} (Y)$, then the homotopy groups of $M$ are supported
on $f^{-1}(Z)$ and therefore have no higher derived pushforwards along $f
\colon Y \to X$. 
As a result, the claim follows from the
analog of our lemma in ordinary algebraic geometry, which is 
\cite[Lemma 5.12, (2)]{bhatt} (itself an extension of \cite[Thm. 2.6.3]{TT90}).
\end{proof}

\begin{prop}
\label{nisnevichexcision}
Let $X$ be a qcqs spectral algebraic space, $U$ a quasi-compact open subset of $X$, $f\colon Y\rightarrow X$ an \'etale map which is an isomorphism above $Z=X\setminus U$, and $F$ a weakly localizing invariant as
above.  Then the diagram
$$\xymatrix{F(\operatorname{Perf}(X))\ar[r]\ar[d] & F(\operatorname{Perf}(U))\ar[d] \\
F(\operatorname{Perf}(Y))\ar[r] & F(\operatorname{Perf}(Y\times_X U)) }$$
is a pullback square.
\end{prop}
\begin{proof}
This will follow from the above localization sequence \eqref{locseq} provided we can show that pullback by $f$ induces an equivalence
$\operatorname{Perf}_Z(X)\simeq \operatorname{Perf}_{f^{-1}Z}(Y).$ (recalling that $F$ takes values in a stable $\infty$-category).
This follows by taking compact objects in  \Cref{A2}.
\end{proof}

Finally, we can prove our main result. 

\begin{proof}[proof of \cref{thm:etaledesc}] 

Since Nisnevich excision is equivalent to Nisnevich descent by
a theorem of Morel-Voevodsky (cf.\ 
\cite[3.7.5.1]{lurie_sag}), it follows from \Cref{nisnevichexcision} that
$U\mapsto F(\perf(U))$ is a Nisnevich sheaf. In particular, showing that it is in fact an
\'etale sheaf is a Nisnevich-local problem, hence by \cite[Ex.~3.7.1.5]{lurie_sag}, we can assume
$X$ is affine. In this case, to conclude \'etale descent, by \cite[B.6.4.1]{lurie_sag} it suffices to establish
finite \'etale descent, which is given by \Cref{finiteetalesatisfied}.

\end{proof}

\begin{remark} 
We note that \Cref{thm:etaledesc} does not extend to the case where $X$ is a
Deligne-Mumford stack in general; counterexamples are easy to come by for the
classifying stacks of finite groups.  
For example, one does not have descent for the $C_2$-Galois cover $\spec \mathbb{C} \to (
\spec \mathbb{C})/(C_2)$, even rationally: we have
\[ K_0((\spec \mathbb{C})/C_2) \simeq R(C_2) \simeq \mathbb{Z}[x]/(x^2 - 1),
 \mbox{ but }
K_0 ( \spec \mathbb{C}) = \mathbb{Z}.
\]
However, \Cref{TMFcb} together with the
derived-affineness result of \cite{MM15} implies that it nonetheless holds for
the spectral Deligne-Mumford moduli stack 
$\mathfrak{M}_{ell}$ underlying the theory of $\TMF$, even though it fails for the underlying usual moduli stack $M_{ell}$.
\end{remark}

\section{Descent for higher real $K$-theories \\by Lennart Meier, Justin Noel,
and Niko Naumann}\label{sec:appendix}
In this appendix we record the existence of finite complexes with controlled Morava $K$-theory,
as implicit in the work of Hopkins, Ravenel, and Smith. These are analogs of the results of Mitchell \cite{Mit85} constructing finite complexes, the mod $p$-cohomology of which is finite free over certain finite sub-algebras of the Steenrod algebra. In our case, the algebra of operations replacing the Steenrod algebra is a group algebra of a Morava stabilizer group, and the finite sub-algebras are given by the group algebras of finite subgroups.

These results will in particular verify the assumption of our main results on descent for the Galois
extensions afforded by higher real $K$-theories.
We refer to  \cite[Sec. 3.1]{HMS15} for a proof that these extensions are
globally Galois.

In order to formulate the results, we start with a reminder on Lubin-Tate theories, cf.\cite{Rez97,GoH05}. 
Fix a prime $p>0$, a perfect field $k$ of characteristic 
$p$ and a (one-dimensional, commutative) formal group $G$ of finite height $n\ge 1$ over $k$.

Associated with this there is the automorphism group $\mathrm{Aut}(G,k)$ of the pair $(G,k)$,
consisting of pairs $(\alpha,\varphi)$ with an automorphism $\alpha:k\to k$ and an isomorphism 
$\varphi:\alpha^*G\to G$ of formal groups over $k$. There is an evident exact sequence of groups 

\begin{equation}\label{eq:extendedstabilzergroup} 1\to \mathrm{Aut}_k(G)\to \mathrm{Aut}(G,k)\to \mathrm{Aut}(k),\end{equation}
the final map sending a pair $(\alpha,\varphi)$ as above to $\alpha$.
The central division algebra $D$ over $\mathbb{Q}_p$ of invariant $1/n$ admits a unique maximal order $\mathcal{O}_D\subset D$, and its group of units $\mathcal{O}_D^*$
is isomorphic to the automorphism group of the unique formal group of height $n$ over $\overline{\mathbb{F}}_p$. Since $G$ becomes isomorphic to this group over an algebraic closure of 
$k$, we can identify $\mathrm{Aut}_k(G)\subset \mathcal{O}_D^*$ as a closed subgroup.

The action of $\mathrm{Aut}_k(G)$ on the Lie-algebra of $G$ over $k$ affords a character, and we denote by 
$\mathrm{Aut}_k^1(G)$ its kernel:

\[ 1\to \mathrm{Aut}_k^1(G)\to \mathrm{Aut}_k(G)\to k^*.\]
It is easy to see that $\mathrm{Aut}_k^1(G)$ is a pro-$p$-group, and hence the unique pro-$p$-Sylow-subgroup
of $\mathrm{Aut}_k(G)$ because it is normal.

There is an even periodic $\mathbb{E}_\infty$-ring spectrum $E(G,k)$ acted upon by $\mathrm{Aut}(G,k)$ such that $\pi_0 E(G,k)\cong W(k)[[u_1,\ldots, u_{n-1}]]$ identifies with the universal deformation ring
of $G$ over $k$. This $E(G,k)$ is Lubin-Tate theory.
One can construct a map of $\mathbb{E}_1$-algebras $E(G,k)\to K(G,k)$ which on $\pi_0$ has the effect of 
quotienting out the maximal ideal \cite[Cor.~3.7]{Ang08}. The ring spectrum $K(G,k)$ is even periodic with $\pi_0K(G,k)\cong k$. This $K(G,k)$ is the associated Morava $K$-theory. For every spectrum $X$, $K(G,k)^0(X)$ is canonically a continuous
module over the completed twisted group-ring $k[[\mathrm{Aut}(G,k)]]$
 (twisted with respect to the action of $\mathrm{Aut}(G,k)$ on $k$ given by  \Cref{eq:extendedstabilzergroup}), and for every finite subgroup $H\subset\mathrm{Aut}(G,k)$, the twisted group ring $k[H]\subset k[[\mathrm{Aut}(G,k)]]$ is a finite-dimensional sub-algebra.

Our existence result for finite complexes is the following.

\begin{thm}\label{thm:existence_finite_complexes} 
For every finite subgroup $H\subset \mathrm{Aut}(G,k)$, there exists a finite, $p$-local complex $X$ with cells in even dimensions such that $K(G,k)^0(X)$ is a non-trivial, finite free $k[H]$-module. 
\end{thm}


Given a finite subgroup $H\subset \mathrm{Aut}(G,k)$, the spectrum $E(G,k)$ is a (Borel complete) 
commutative algebra in genuine $H$-spectra, see for example \cite[\S 6.3]{MNN15i} for background on this. By a {\em semi-linear $E(G,k)$-$H$-module}, we will mean an $E(G,k)$-module internal to Borel complete genuine $H$-spectra. The first example is $M=E(G,k)$, as is more generally $M=E(G,k)\wedge X$ for any finite spectrum $X$, endowed with the $H$-action through the first smash factor. 

The free example is the $E(G,k)$-module $M=H_+\wedge E(G,k)$, endowed with the diagonal $H$-action. By the projection formula \cite[Prop.~5.14]{MNN15i}, this is equivalent to $\Ind_{ \{ e \} }^H\Res_{ \{ e \} }^H E(G,k)$. Since $H$ is finite, induction and coinduction agree, and given any semi-linear $E(G,k)$-$H$-module $N$, the datum of a homotopy class of a semi-linear map $map(H,E(G,k))\simeq \Coind_{ \{ e \} }^H\Res_{ \{ e \} }^H E(G,k)\to N$ is equivalent to the datum of the element of $\pi_0 N$ obtained by evaluation at the unit.

We denote by $DX=F(X,S)$ the Spanier-Whitehead dual of $X$.



\begin{corollary}\label{cor:split_higher_real} 
In the situation of \Cref{thm:existence_finite_complexes}, there is an equivalence of semi-linear $E(G,k)$-$H$-modules $E(G,k)\wedge DX\cong map(H,E(G,k))^{\vee n}$ for some $n\neq 0$, and consequently there is an equivalence of $E(G,k)^{hH}$-modules $E(G,k)^{hH}\wedge DX\cong E(G,k)^{\vee n}$.
\end{corollary}

This result will very easily implies the next, which shows that the  extension $A:=E(G,k)^{hH}\to B:=E(G,k)$ satisfies the assumption of \Cref{Galoisdesc}.

\begin{corollary}\label{cor:higher_real_transfer}
	In the situation of \Cref{cor:split_higher_real}, the rationalized transfer map
	\[ K_0(E(G,k))\otimes\mathbb{Q}\to K_0(E(G,k)^{hH})\otimes\mathbb{Q}\]
	is surjective.
\end{corollary}

In some special cases, we can obtain stronger descent results in the algebraic $K$-theory of higher real $K$-theories, by using the following
sharper variant of \Cref{thm:existence_finite_complexes}. For this, we
denote by $X$ the $p$-local $p$-cell complex which has all attaching maps
equal to $\alpha_1$, and which is called $T(0)_{(1)}$ in \cite[Example 7.1.17]{ravenel_green_book}.

\begin{thm}\label{thm:p_cell_complexes} Assume that $n=p-1$ and that $\mathbb{F}_{p^n}\subset k$. Then, for every $C_p\subset \mathrm{Aut}_k(G)$, the 
$k[C_p]$-module $K(G,k)^0(X)$ is free of rank $1$.
\end{thm}

\begin{corollary}\label{cor:better_real_transfer}
	Assume that $n=p-1$ and that $\mathbb{F}_{p^n}\subset k$. 
	Then $p$ is in the image of the transfer map
	\[ K_0(E(G,k))\to K_0(E(G,k)^{hC_p}).\]
\end{corollary}

The proof of \Cref{thm:p_cell_complexes} is a direct application of Ravenel's computation of $BP_*(X)$. It will be explained at the very end of this appendix.
We now begin working on the proof of \Cref{thm:existence_finite_complexes}, while the proofs of
\Cref{cor:split_higher_real} and \Cref{cor:higher_real_transfer} will appear immediately after that.

Our proof is a digest of some parts of \cite{ravenel_orange_book}, and more exactly of J. Smith's
construction of finite complexes \cite{smith_finite_complexes}. Since the finite complex $X$ will be constructed out of some complex projective space, we will first need
some information about the $k[[\mathrm{Aut}(G,k)]]$-module $K(G,k)^0(\mathbb{CP}^\infty)$. This is isomorphic to $k[[T]]$, after fixing a coordinate $T$ of the formal group of the complex orientable ring 
spectrum $K(G,k)$. Writing $F(X,Y)\in k[[X,Y]]$ for the resulting formal group law, we obtain 
\[ \mathrm{Aut}_k(G)\cong \left\{ f\in T\cdot k[[T]]\, \mid \, f'(0)\neq 0\mbox{ and } F(f(X),f(Y))=f(F(X,Y)) \right\}.\]

Since $K(G,k)^0(\mathbb{CP}^\infty)\cong k[[T]]$ {\em is} the ring of functions on our formal group, an element $f\in \mathrm{Aut}_k(G) $ acts on it as the unique continuous map of $k$-algebras sending $T$ to $f(T)$.
This completely determines the $k[[\mathrm{Aut}_k(G)]]$-module $K(G,k)^0(\mathbb{CP}^\infty)$, and
shows in particular that it is faithful. Jointly with $K(G,k)^0(\mathbb{CP}^\infty)\cong \lim_N K(G,k)^0(\mathbb{CP}^N)$, this implies the following.
  
\begin{prop}\label{prop:non-trivial-action}
For every $1\neq f\in \mathrm{Aut}_k(G)$, there is some $N$ such that $f$ acts non-trivially on $K(G,k)^0(\mathbb{CP}^N)$.
\end{prop}

For the sake of readability, we now write $\mathbb{CP}(N):=\mathbb{CP}^N$.

\begin{prop}\label{prop:free_summand}
There is some $N\ge 0$ such that for every inclusion ${C_p\subset H\cap\mathrm{Aut}_k(G)}$, we have
\[ K(G,k)^0(\mathbb{CP}(N)^{\times (p-1)})=U\oplus k[C_p] \]
as $k[C_p]$-modules, for some $U$.
\end{prop}

\begin{proof} 
Since $H$ is finite, and using \Cref{prop:non-trivial-action}, we can find some $N\ge 0$ such that for every inclusion $C_p\subset  H$, the $k[C_p]$-module $V:=K(G,k)^0(\mathbb{CP}(N))$
is non-trivial, and the claim then follows from $K(G,k)^0(\mathbb{CP}(N)^{\times (p-1)})\cong V^{\otimes (p-1)}$
and the following algebraic result, established during the proof of \cite[Theorem C.3.3]{ravenel_orange_book}: If $V$ is a non-trivial
$k[C_p]$-module, then $V^{\otimes(p-1)}$ splits off a free module of rank $1$. (The proof in {\em loc.~cit.} is written for some specific finite field, but evidently works for every field of characteristic $p$).
\end{proof}

\begin{proof}[Proof of {\Cref{thm:existence_finite_complexes}}]
Put $m:=\mathrm{dim}_{k}(U)+1$ with $U$ as in \Cref{prop:free_summand} and $\kappa:=(p-1){m+1 \choose 2}$. Then there is an idempotent $e\in\mathbb{Z}_{(p)}[\Sigma_\kappa]$ such that for every finite-dimensional
$k$-vector space $W$, the direct summand $eW^{\otimes \kappa}\subset W^{\otimes \kappa}$ is non-zero
if and only if $\mathrm{dim}_{\kappa}(W)\ge m$ \cite[Theorem C.1.5]{ravenel_orange_book}.
In particular, we have $eU^{\otimes \kappa}=0$. 

With $N\ge 0$ as in \Cref{prop:free_summand}, we now consider the complex 
\[ Y:= e\cdot\left(  \left( \mathbb{CP}(N)^{\times (p-1)} \right)  ^{\times \kappa}   \right)_{(p)} .\]
It clearly is a finite, $p$-local complex with cells in even dimension.

We denote $H'':=H\cap\mathrm{Aut}_k^1(G)$ and
claim that the $k[H'']$-module $K(G,k)^0(Y)$ is (non-trivial and) finite free. Since $H''$ is a $p$-group, $k[H'']$ is an Artin local ring. By \cite[Thm.~19.29]{Lam} $K(G,k)^0(Y)$ is $k[H'']$-free if and only if it is $k[H'']$-projective. By Chouinard's theorem \cite[Cor.\ 1.1]{chouinard}, this holds if for every inclusion $E\subset H''$ of an elementary $p$-abelian subgroup, this module is
projective over $k[E]$. Every such $E$ is a finite subgroup of the group of units of a commutative subfield
of $D$, and thus cyclic. So we can assume that we are given some inclusion $E=C_p\subset H''$.

By construction, \Cref{prop:free_summand} and the K\"unneth isomorphism, we have 
\[ K(G,k)^0(Y)=e\cdot\left((U\oplus k[C_p])^{\otimes \kappa}\right) \]
as a $k[C_p]$-module, using the notation above. Observing that every $k[C_p]$-module of the form $k[C_p]\otimes M$ is free
(say, as a consequence of the projection formula), we can multiply out 
\[ (U\oplus k[C_p])^{\otimes \kappa} = U^{\otimes \kappa}\oplus\tilde{F} \]
for some finite free $k[C_p]$-module $\tilde{F}\neq 0$.
We can now conclude that 
\[  K(G,k)^0(Y)=e\cdot\left(  U^{\otimes \kappa}\oplus\tilde{F} \right) = e\cdot\tilde{F} \]
is a non-trivial finite free $k[C_p]$-module, and hence finite free over $k[H'']$, as claimed.

We now want to induce this up along the inclusions of groups
$H''\subset H':=H\cap\mathrm{Aut}_k(G)\subset H$. The $k[H']$-module $K(G,k)^0(Y)\otimes_{k[H'']}k[H']$ is clearly finite free, and by the projection formula, it is isomorphic to $K(G,k)^0(Y)\otimes_{k}k[H'/H'']$, endowed with the diagonal $H'$-action. So we next find a finite even complex $Z$ with
$K^0(G,k)(Z)\cong k[H'/H'']$ as a $k[H']$-module, for then $K(G,k)^0(Y\wedge Z)$
will be finite free over $k[H']$. 

Now, since $\mathrm{char}(k)=p$, multiplication by $p$ on $k^*$ is injective, and hence $H'/H''\subseteq k^*$ is a cyclic group of order coprime to $p$. Since $k[H'/H'']$ is semi-simple and $k$ evidently contains the $|H'/H''|$-roots of unity, as an $H'/H''$-module $k[H'/H'']$ is a sum of powers of a generating character. Pulling back to an $H'$-action and recalling our initial discussion, we have \[ k[H'/H'']\cong \bigoplus\limits_{i=0}^{(\mid H'/H'' \mid -1)} Lie^{\otimes i}.\]
 We can thus take $Z:=\bigvee_{i=0}^{(\mid H'/H'' \mid -1)} S^{2i}$ for the desired complex.

To induct up further along $H'\subset H$, we observe that \[ K(G,k)^0(Y\wedge Z)\otimes_{k[H']}k[H]\cong K(G,k)^0(Y\wedge Z)\otimes_k k[\Gamma],\] where we denote $\Gamma:=H/H'\subset\mathrm{Aut}(k)$. As above, we then need to find a finite even complex $W$ with $K^0(G,k)(W)\cong k[\Gamma]$ as a $k[H]$-module, for then
$X:=Y\wedge Z\wedge W$ will be as desired. By Galois descent, we have an isomorphism of semi-linear $k-\Gamma$-modules
\[k[\Gamma]\cong k\otimes_{k^\Gamma}k\cong k \otimes_{k^{\Gamma}} (k^\Gamma)^{\oplus \mid\Gamma\mid}\cong k^{\oplus \mid\Gamma\mid},\] where the semi-linear action on the tensor products is through the left tensor factors. Hence we can take $W:=\bigvee\limits_{\mid G \mid} S^0$ as the desired complex.
\end{proof}

\begin{proof}[Proof of {\Cref{cor:split_higher_real}}]
Fix a basis $\{\alpha_1,\ldots,\alpha_n\}$ of the $k[H]$-module $K^0(G,k)(X)$. Since $X$ is even, these lift to elements in $E(G,k)^0(X)=\pi_0(E(G,k)\wedge DX)$ which determine a semi-linear map $map(H,E(G,k))^{\vee n}
\to E(G,k)\wedge DX$, and it suffices to show that this map is an equivalence (of spectra).
Since both spectra are finite free $E(G,k)$-modules, it suffices to check this after application of $\pi_0(-)\otimes_{E^0(G,k)}K^0(G,k)$ which yields the $K^0-H$ semi-linear map $map(H,K^0)^{\oplus n}\longrightarrow K^0(X)$ determined by the basis above. This map is clearly an isomorphism.
\end{proof}

\begin{proof}[Proof of \Cref{cor:higher_real_transfer}]
By \Cref{cor:split_higher_real}, there is a finite even complex $X$ such that the $E(G,k)^{hH}$-module
$E(G,k)^{hH}\wedge DX$ admits the structure of an $E(G,k)$-module, hence the class $[E(G,k)^{hH}\wedge DX]\in K_0(E(G,k)^{hH})$
lies in the image of the transfer map. Using $[\Sigma^2 DX]=[DX]$, that $DX$ has only even cells and induction on this number $M$ of cells
shows that $[E(G,k)^{hH}\wedge DX]=M\cdot [E^{hH}]$ is a positive multiple of the unit. So the image of the rationalized transfer is an ideal which contains $1$;
hence the rationalized transfer is surjective.
\end{proof}

\begin{proof}[Proof of \Cref{thm:p_cell_complexes}] We need to compute the $k[C_p]$-module $V:=K(G,k)^0(X)$, which is a $p$-dimensional representation
of $C_p$ over $k$. We will deduce this from Ravenel's result 
\cite[Lemma 7.1.11]{ravenel_green_book}, which states 
that $BP_*(X)$ is isomorphic, 
as a $BP_*BP$-comodule, to the subcomodule of $BP_*BP\cong 
BP_*[t_1,t_2,\ldots]$ generated freely over $BP_*$ by the
 set $\{t_1^i\,\mid\, 0\leq i\leq p-1\}$. This implies that the $K(G,k)_0K(G,k)$-comodule $M:=K(G,k)_0(X)$ is free over $k$ on the images of the $t_1^i$.

Recall that we can identify \[K(G,k)_0K(G,k)\cong 
maps_c(\mathrm{Aut}^1_k(G),k)\cong k[t_1,t_2,\ldots]/(t_i^{p^n}-t_i)\] in such 
a way that the continuous map $t_i$ on $\mathrm{Aut}^1_k(G)$ is 
given as follows. We have $\mathrm{Aut}^1_k(G)\subset 
1+\Pi\cdot {\mathcal O}_D$, where $D$ is the skew
 field over $\mathbb{Q}_p$
 of invariant $1/n$, and $\Pi$ is a uniformizer. Since 
 $\mathbb{F}_{p^n}\subset k$, this inclusion is in fact an equality, and
  we can write every $g\in \mathrm{Aut}^1_k(G)$ uniquely as 
  $g=1+\sum\limits_i^\infty [t_i(g)]\Pi^i$ using the Teichm\"uller 
  lift \[[-]:\mathbb{F}_{p^n}\longrightarrow
W(\mathbb{F}_{p^n})\subset {\mathcal O}_D.\]
 Briefly, the $t_i$ are the digits in the $\Pi$-adic expansion.

 To ease the notation, we introduce $e_i:=t_1^i$, a $k$-basis of $M$, and denote by $f_i\in V$ the dual basis.
Writing $\psi:M\longrightarrow K(G,k)_0K(G,k)\otimes_k M$ for the comodule structure, our representation $V$ is the dual of $M$, in the sense that
we have $(g\cdot v)(m)=(\mathrm{ev}_g\otimes v)(\psi(m))$ for every $g\in\mathrm{Aut}^1_k(G)$,
$v\in V$, $m\in M$, and denoting $\mathrm{ev}_g:K(G,k)_0K(G,K)\to k$ the evaluation of a continuous map at $g$. Finally, it is clear that the map
$t_1:\mathrm{Aut}^1_k(G)\longrightarrow (\mathbb{F}_{p^n},+)\subset k$
is a homomorphism, and hence $t_1$ is primitive, and in particular
that implies that $\psi(e_j)=(t_1\otimes 1 + 1 \otimes e_1)^j$.
Given the above, a formal computation which we leave to the reader, gives
\[ g\cdot f_i = f_i + \sum\limits_{j>i}^{p-1} {j \choose i}\alpha^{j-i}\cdot f_j\]
for all $g\in\mathrm{Aut}^1_k(G), 0\leq i\leq p-1$ and we abbreviate $\alpha:=t_1(g)$. For the endomorphism $\varphi:=g-1 \in End_k(V)$, this
implies that $\varphi^p=0$ and that $\varphi^{p-1}(f_0)=(p-1)! \alpha^{p-1} f_{p-1}$. In particular,
if $\alpha\neq0 $, then $\varphi^{p-1}\neq 0$. 

We now specialize this to the
case of interest, namely when $\langle g \rangle = C_p$, and claim that in this case we have $\alpha=t_1(g)\neq 0$. By the above reminder on the relation between the 
$t_i$ and the $\Pi$-adic expansion of $g$, this is equivalent to saying that 
$g-1$ and $\Pi$ have the same valuation in $D$. This is clear, because both valuations are $\frac{1}{p-1}$ times the valuation of $p$: $g$ is a primitive $p$th root of unity in $D$ and $1/n=1/(p-1)$ is the invariant of $D$. 

We conclude that the order of $g-1$ acting on $V$ is exactly $p$, and the Jordan
normal form implies, that $V$ is a free module (of rank $1$) over $k[C_p]$, as desired.
\end{proof}

\begin{proof}[Proof of \Cref{cor:better_real_transfer}]
This follows from \Cref{thm:p_cell_complexes} in the same way in which
\Cref{cor:higher_real_transfer} follows from \Cref{thm:existence_finite_complexes}.

\end{proof}

\bibliographystyle{alpha}
\bibliography{GeneralLemma}
\end{document}